\documentclass[11pt,american,english]{article}
\usepackage[latin9]{inputenc}
\usepackage{geometry}
\usepackage{color}
\usepackage{babel}
\usepackage{array}
\usepackage{multirow}
\usepackage{amsmath}
\usepackage{amsthm}
\usepackage{amssymb}
\usepackage{graphicx}
\usepackage[unicode=true,pdfusetitle,
 bookmarks=true,bookmarksnumbered=false,bookmarksopen=false,
 breaklinks=false,pdfborder={0 0 1},backref=page,colorlinks=true]
 {hyperref}

\makeatletter

\providecommand{\tabularnewline}{\\}

\numberwithin{equation}{section}
\numberwithin{figure}{section}
\theoremstyle{plain}
\newtheorem{thm}{\protect\theoremname}[section]
\theoremstyle{plain}
\newtheorem{prop}[thm]{\protect\propositionname}
\theoremstyle{definition}
\newtheorem{defn}[thm]{\protect\definitionname}
\theoremstyle{plain}
\newtheorem{cor}[thm]{\protect\corollaryname}
\theoremstyle{remark}
\newtheorem{rem}[thm]{\protect\remarkname}
\newenvironment{lyxcode}
	{\par\begin{list}{}{
		\setlength{\rightmargin}{\leftmargin}
		\setlength{\listparindent}{0pt}% needed for AMS classes
		\raggedright
		\setlength{\itemsep}{0pt}
		\setlength{\parsep}{0pt}
		\normalfont\ttfamily}%
	 \item[]}
	{\end{list}}
\theoremstyle{plain}
\newtheorem{conjecture}[thm]{\protect\conjecturename}
\theoremstyle{remark}
\newtheorem{claim}[thm]{\protect\claimname}
\theoremstyle{plain}
\newtheorem{lem}[thm]{\protect\lemmaname}
\theoremstyle{plain}
\newtheorem{fact}[thm]{\protect\factname}
\theoremstyle{plain}
\newtheorem*{thm*}{\protect\theoremname}

\usepackage{graphicx}
\usepackage{appendix}

\usepackage{enumitem}

\makeatother

\addto\captionsamerican{\renewcommand{\claimname}{Claim}}
\addto\captionsamerican{\renewcommand{\conjecturename}{Conjecture}}
\addto\captionsamerican{\renewcommand{\corollaryname}{Corollary}}
\addto\captionsamerican{\renewcommand{\definitionname}{Definition}}
\addto\captionsamerican{\renewcommand{\factname}{Fact}}
\addto\captionsamerican{\renewcommand{\lemmaname}{Lemma}}
\addto\captionsamerican{\renewcommand{\propositionname}{Proposition}}
\addto\captionsamerican{\renewcommand{\remarkname}{Remark}}
\addto\captionsamerican{\renewcommand{\theoremname}{Theorem}}
\addto\captionsenglish{\renewcommand{\claimname}{Claim}}
\addto\captionsenglish{\renewcommand{\conjecturename}{Conjecture}}
\addto\captionsenglish{\renewcommand{\corollaryname}{Corollary}}
\addto\captionsenglish{\renewcommand{\definitionname}{Definition}}
\addto\captionsenglish{\renewcommand{\factname}{Fact}}
\addto\captionsenglish{\renewcommand{\lemmaname}{Lemma}}
\addto\captionsenglish{\renewcommand{\propositionname}{Proposition}}
\addto\captionsenglish{\renewcommand{\remarkname}{Remark}}
\addto\captionsenglish{\renewcommand{\theoremname}{Theorem}}
\providecommand{\claimname}{Claim}
\providecommand{\conjecturename}{Conjecture}
\providecommand{\corollaryname}{Corollary}
\providecommand{\definitionname}{Definition}
\providecommand{\factname}{Fact}
\providecommand{\lemmaname}{Lemma}
\providecommand{\propositionname}{Proposition}
\providecommand{\remarkname}{Remark}
\providecommand{\theoremname}{Theorem}

\begin{document}
\global\long\def\Hom{\mathrm{Hom}}
 \global\long\def\wedger{{\textstyle \bigvee^{r}S^{1}} }
 \global\long\def\MCG{\mathrm{MCG}}
\global\long\def\mcg{\MCG}
 \global\long\def\trw{{\cal T}r_{w} }
 \global\long\def\tr{{\cal T}r }
\global\long\def\trwl{{\cal T}r_{w_{1},\ldots,w_{\ell}} }
\global\long\def\wl{w_{1},\ldots,w_{\ell}}
 \global\long\def\cl{{\cal \mathrm{cl}} }
\global\long\def\ch{\chi_{\max}}
\global\long\def\wg{{\cal \mathrm{Wg}} }
\global\long\def\moeb{\mathrm{M\ddot{o}b} }
\global\long\def\F{\mathrm{\mathbf{F}} }
\global\long\def\id{\mathrm{id}}
\global\long\def\e{\varepsilon}
\global\long\def\U{\mathcal{U}}
\global\long\def\Aut{\mathrm{Aut}}
\global\long\def\E{\mathbb{E}}
 \global\long\def\T{{\cal T}}
\global\long\def\tps{\left|{\cal T}\right|_{\ps}}
\global\long\def\tips{\left|{\cal T}_{\infty}\right|_{\ps}}
\global\long\def\ps{\mathrm{poly}}

\global\long\def\match{\mathrm{MATCH}}
\global\long\def\matchr{\mathrm{\overline{MATCH}}}
\global\long\def\matchmap{\mathbf{match}}
\global\long\def\bijs{\mathrm{bijs}}
\global\long\def\g{\gamma}
\global\long\def\sur{\surfaces}
\global\long\def\surfaces{{\cal S}\mathrm{urfaces}}
\global\long\def\sx{\overline{\sigma_{x}}}
\global\long\def\os{\overrightarrow{S^{1}}}

\global\long\def\rank{\mathrm{rank}}
\global\long\def\chain{\mathbf{Ch}}
\global\long\def\mod{\mathbf{mod}}
\global\long\def\L{{\cal L}}
\global\long\def\res{\mathrm{res}}
\global\long\def\eq{\mathrm{eq}}
 \global\long\def\Stab{\mathrm{Stab}}
\global\long\def\Cells{\mathrm{Cells}}
\global\long\def\C{\mathcal{C}}
\global\long\def\BIJS{\mathcal{BIJS}}
\global\long\def\cd{\mathrm{cd}}
\global\long\def\vcd{\mathrm{vcd}}

\global\long\def\Z{\mathbf{\mathbf{Z}}}
\global\long\def\Q{\mathbb{\mathbb{\mathbf{Q}}}}
\global\long\def\d{\delta}
\global\long\def\G{\Gamma}
\global\long\def\bij{\mathrm{BIJ}}
\global\long\def\g{\gamma}
\global\long\def\surfaces{\mathsf{Surfaces}}
\global\long\def\sx{\overline{\sigma_{x}}}
\global\long\def\os{\overrightarrow{S^{1}}}
\global\long\def\B{\mathcal{B}}
\global\long\def\sing{\mathrm{sing}}

\global\long\def\R{\mathbf{R}}
\global\long\def\Q{\mathbb{\mathbb{\mathbf{Q}}}}
\global\long\def\image{\mathrm{image}}
\global\long\def\sing{\mathrm{sing}}
\global\long\def\Rep{\mathrm{Rep}}
\global\long\def\N{\mathcal{N}}
\global\long\def\thick{\mathbb{D}}
\global\long\def\P{{\cal P}}

\title{Matrix Group Integrals, Surfaces, and Mapping Class Groups I: $\U\left(n\right)$}

\author{Michael Magee\thanks{M. Magee was partially supported by the N.S.F. via grants DMS-1128155
and DMS-1701357.} ~~and~~ Doron Puder\thanks{D. Puder was partially supported by the Rothschild Fellowship, N.S.F.
via grant DMS-1128155, and I.S.F. via grant 1071/16.}}
\maketitle
\begin{abstract}
Since the 1970's, physicists and mathematicians who study random matrices
in the GUE or GOE models are aware of intriguing connections between
integrals of such random matrices and enumeration of graphs on surfaces.
We establish a new aspect of this theory: for random matrices sampled
from the group $\U\left(n\right)$ of unitary matrices. 

More concretely, we study measures induced by free words on $\U\left(n\right)$.
Let $\F_{r}$ be the free group on $r$ generators. To sample a random
element from $\U\left(n\right)$ according to the measure induced
by $w\in\F_{r}$, one substitutes the $r$ letters in $w$ by $r$
independent, Haar-random elements from $\U\left(n\right)$. The main
theme of this paper is that every moment of this measure is determined
by families of pairs $\left(\Sigma,f\right)$, where $\Sigma$ is
an orientable surface with boundary, and $f$ is a map from $\Sigma$
to the bouquet of $r$ circles, which sends the boundary components
of $\Sigma$ to powers of $w$. A crucial role is then played by Euler
characteristics of subgroups of the mapping class group of $\Sigma$.

As corollaries, we obtain asymptotic bounds on the moments, we show
that the measure on $\U\left(n\right)$ bears information about the
number of solutions to the equation $\left[u_{1},v_{1}\right]\cdots\left[u_{g},v_{g}\right]=w$
in the free group, and deduce that one can ``hear'' the stable commutator
length of a word through its unitary word measures.
\end{abstract}
\tableofcontents{}

\section{Introduction\label{sec:Introduction}}

Let $\U\left(n\right)$\marginpar{$\protect\U\left(n\right)$} denote
the group of $n\times n$ unitary complex matrices, and let $\F_{r}$\marginpar{$\protect\F_{r}$}
denote the free group on $r$ generators with fixed basis (free generating
set) \marginpar{$B$}$B=\left\{ x_{1},\ldots,x_{r}\right\} $. For
a word $w\in\F_{r}$, we define the $w$-measure on $\U\left(n\right)$
as the push-forward of the Haar measure on $\U\left(n\right)^{r}$
through the word map $w\colon\U\left(n\right)^{r}\to\U\left(n\right)$.
In plain terms, assume that $w=x_{i_{1}}^{\varepsilon_{1}}\cdots x_{i_{m}}^{\varepsilon_{m}}$.
To sample a random element from $\U\left(n\right)$ by the $w$-measure,
sample $r$ independent Haar-random elements $A_{1},\ldots,A_{r}\in\U\left(n\right)$
and evaluate $w\left(A_{1},\ldots,A_{r}\right)=A_{i_{1}}^{\varepsilon_{1}}\cdots A_{i_{m}}^{\varepsilon_{m}}\in\U\left(n\right)$. 

The motivation to study $w$-measures on unitary groups or on compact
groups in general originates in questions revolving around random
walks on these groups, in the study of representation varieties, in
problems in the theory of Free Probability, and in challenges in the
study of free groups. However, as the current paper shows, the study
of $w$-measures is interesting for its own sake and reveals deep
and surprising connections with other mathematical concepts. See also
\cite{PP15}. 

\subsubsection*{Expected trace}

We study word measures by considering their moments, and more particularly
the expected product of traces. For every $\ell\in\mathbb{N}_{\ge1}$
and $\wl\in\F_{r}$, consider the quantity
\begin{equation}
\trwl\left(n\right)\stackrel{\mathrm{def}}{=}\int_{A_{1},\ldots,A_{r}\in\U\left(n\right)}\mathrm{tr}\left(w_{1}\left(A_{1},\ldots,A_{r}\right)\right)\cdot\ldots\cdot\mathrm{tr}\left(w_{\ell}\left(A_{1},\ldots,A_{r}\right)\right)d\mu\label{eq:def of Tr_w}
\end{equation}
where $A_{1},\ldots,A_{r}\in\U\left(n\right)$ are independent Haar-random
unitary matrices\footnote{Let us mention that the $w$-measure on $\U\left(n\right)$ is completely
determined by moments of this type where the words are taken to be
powers of $w$: $\tr_{w^{\alpha_{1}},\ldots,w^{\alpha_{\ell}}}\left(n\right)$
with $\alpha_{1},\ldots,\alpha_{\ell}\in\mathbb{Z}$. See, for example,
\cite[Section 2.2]{MP}. (We comment about the pre-print \cite{MP}
in Remark \ref{remark:MP16}.)}. The development of ``Weingarten calculus'' for computing integrals
on $\U\left(n\right)$ \cite{weingarten1978asymptotic,xu1997random,collins2003moments,CS}
leads readily to the following result:
\begin{prop}
\label{prop:rational expression}Let $\ell\in\mathbb{N}_{\ge1}$ and
$\wl\in\F_{r}$. Then for large enough $n$, the quantity $\trwl\left(n\right)$
is given by a rational expression in $n$ with rational coefficients,
namely, by an element of $\mathbb{Q}\left(n\right)$.
\end{prop}

Here ``large enough $n$'' means that $n\ge\max_{x\in B}L_{x}$,
where $L_{x}$ is the total number of instances of $x^{+1}$ in the
words $\wl$. 

In Section \ref{sec:A-Formula-for trwl} we give explicit combinatorial
formulas for $\trwl\left(n\right)$, and the main innovation here
is the emergence of surfaces in these formulas. In Table \ref{tab:examples}
we list some examples\footnote{Every example in Table \ref{tab:examples} satisfies that for every
generator $x_{i}$, the total number of occurrences in $\wl$ of $x_{i}^{+1}$
is equal to the number of occurrences of $x_{i}^{-1}$. The reason
is the simple fact that otherwise $\trwl\left(n\right)$ is constantly
zero -- see Claim \ref{claim: tr=00003D0 for non-balanced words}
below.} for these rational expressions for concrete words. 

\renewcommand{\arraystretch}{2}
\begin{center}
\begin{table*}
\begin{centering}
\begin{tabular}{|c|c|c|c|}
\hline 
\noalign{\vskip\doublerulesep}
$\ell$ & $\wl$ & $\trwl\left(n\right)$ & Laurent Series\tabularnewline[\doublerulesep]
\hline 
\hline 
\noalign{\vskip\doublerulesep}
\multirow{6}{*}{1} & $\left[x,y\right]$ & {\Large{}$\frac{1}{n}$} & {\Large{}$\frac{1}{n}$}\tabularnewline[\doublerulesep]
\cline{2-4} 
\noalign{\vskip\doublerulesep}
 & $\left[x^{3},y\right]$ & {\Large{}$\frac{3}{n}$} & {\Large{}$\frac{3}{n}$}\tabularnewline[\doublerulesep]
\cline{2-4} 
\noalign{\vskip\doublerulesep}
 & $\left[x,y\right]^{2}$ & {\Large{}$\frac{-4}{n^{3}-n}$} & {\Large{}$\frac{-4}{n^{3}}+\frac{-4}{n^{5}}+\frac{-4}{n^{7}}+\cdots$}\tabularnewline[\doublerulesep]
\cline{2-4} 
\noalign{\vskip\doublerulesep}
 & $\left[x,y\right]^{3}$ & {\Large{}$\frac{9\left(n^{2}+4\right)}{n^{5}-5n^{3}+4n}$} & {\Large{}$\frac{9}{n^{3}}+\frac{81}{n^{5}}+\frac{369}{n^{7}}+\cdots$}\tabularnewline[\doublerulesep]
\cline{2-4} 
\noalign{\vskip\doublerulesep}
 & $\left[x,y\right]\left[x,z\right]$ & {\Large{}$0$} & {\Large{}$0$}\tabularnewline[\doublerulesep]
\cline{2-4} 
\noalign{\vskip\doublerulesep}
 & $\left[x,y\right]\left[x,z\right]\left[x,t\right]$ & {\Large{}$0$} & {\Large{}$0$}\tabularnewline[\doublerulesep]
\hline 
\noalign{\vskip\doublerulesep}
\multirow{3}{*}{2} & $x^{2}y^{2},xy^{-3}x^{-3}y$ & \multirow{1}{*}{{\Large{}$\frac{4\left(n^{2}-5\right)}{n^{4}-5n^{2}+4}$}} & {\Large{}$\frac{4}{n^{2}}+\frac{0}{n^{4}}+\frac{-16}{n^{6}}+\frac{-80}{n^{8}}+\cdots$}\tabularnewline[\doublerulesep]
\cline{2-4} 
\noalign{\vskip\doublerulesep}
 & $w,w^{-1}$ for $w=x^{2}yxy^{-1}$ & {\Large{}$1$} & {\Large{}$1$}\tabularnewline[\doublerulesep]
\cline{2-4} 
\noalign{\vskip\doublerulesep}
 & $w,w^{-1}$ for $w=x^{2}y^{2}xy^{-1}$ & {\Large{}$\frac{n^{4}-5n^{2}}{n^{4}-5n^{2}+4}$} & {\Large{}$1+\frac{0}{n^{2}}+\frac{-4}{n^{4}}+\frac{-20}{n^{6}}+\cdots$}\tabularnewline[\doublerulesep]
\hline 
\end{tabular}
\par\end{centering}
\caption{Some examples for the rational expression for $\protect\trwl\left(n\right)$
and (the beginning of) its Laurent series expansion. All these examples
contain words in $\protect\F_{4}$ with generators $\left\{ x,y,z,t\right\} $.
The notation $\left[x,y\right]$ is for the commutator $xyx^{-1}y^{-1}$
.\label{tab:examples}}
\end{table*}
\par\end{center}

\vspace{-30bp}
The main theme of the current paper is the interpretation of these
expressions for $\trwl\left(n\right)$ in terms of properties of $w$.
We explain their degree and their leading coefficient. More generally,
we show how the entire Laurent series for $\trwl\left(n\right)$ is
determined by natural objects related to $\wl$. 

\medskip{}

\subsubsection*{Extending maps from circles to surfaces}

Our main result, Theorem \ref{thm:main} below, states that the expressions
for $\trwl\left(n\right)$ can be described in terms of certain surfaces
and maps. Roughly, consider a bouquet of $r$ circles $\wedger$\marginpar{$\protect\wedger$}
with fundamental group identified with $\F_{r}$. Now consider $\ell$
disjoint circles (one-spheres) $C_{1},\ldots,C_{\ell}$ and a map
\[
f_{\wl}\colon C_{1}\sqcup\ldots\sqcup C_{\ell}\to\wedger
\]
sending $C_{i}$ to a loop at the bouquet representing $w_{i}$. We
now construct pairs $\left(\Sigma,f\right)$ of an orientable surface
$\Sigma$ with $\ell$ boundary components together with a map $f\colon\Sigma\to\wedger$,
so that the restriction of $f$ to the boundary $\partial\Sigma$
is equal to $f_{\wl}$. From this set one can fully recover the expressions
for $\trwl\left(n\right)$. See Figure \ref{fig:Two-surfaces-extending}. 
\begin{center}
\begin{figure}
\includegraphics{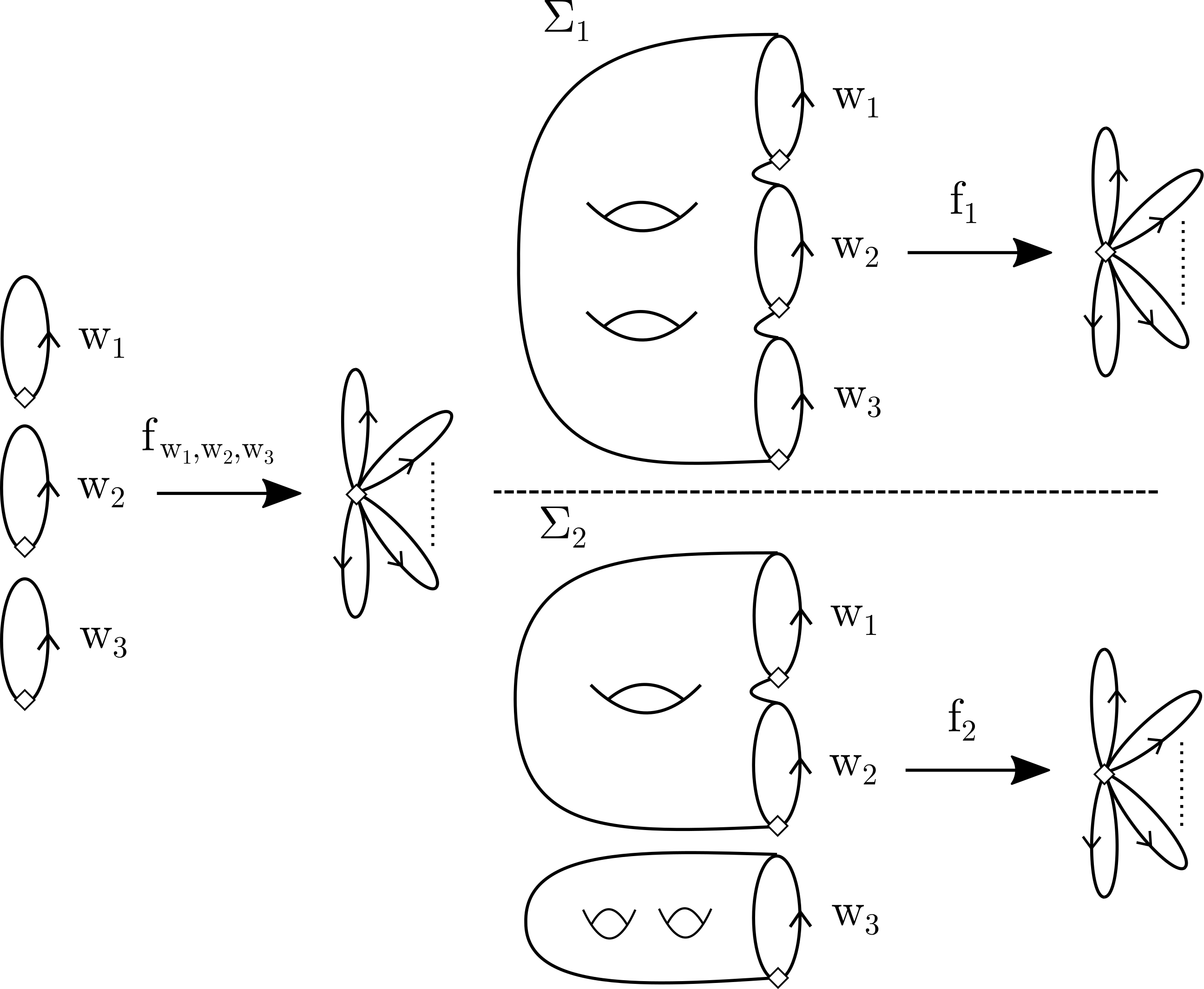}
\centering{}\caption{\label{fig:Two-surfaces-extending}Two surfaces extending the map
from three circles to $\protect\wedger$ corresponding to the triple
of words $w_{1},w_{2},w_{3}$. }
\end{figure}
\par\end{center}

\vspace{-30bp}
More formally, identify the free group $\F_{r}$ with the fundamental
group of $\wedger$, by orienting every circle in the bouquet and
determining a bijection between the circles and the generators $x_{1},\ldots,x_{r}$
of $\F_{r}$. Mark the wedge point by $o$\marginpar{$o$}. We have
\[
\F_{r}\cong\pi_{1}\left(\wedger,o\right).
\]
Let $C_{1}\sqcup\ldots\sqcup C_{\ell}$ be a disjoint union of $\ell$
oriented 1-spheres with a marked point $v_{i}\in C_{i}$\marginpar{$v$}
for every $i=1,\ldots,\ell$. The map $f_{\wl}\colon C_{1}\sqcup\ldots\sqcup C_{\ell}\to\wedger$
sends $v_{1},\ldots,v_{\ell}$ to $o$, and the induced map on fundamental
groups sends the loop at $v_{i}$ around the oriented $C_{i}$ to
$\left[w_{i}\right]$. 
\begin{defn}
\label{def:surface of w1,...,wl}Let $\Sigma$ be a surface with $\ell$
boundary components $\partial\Sigma_{1},\ldots,\partial\Sigma_{\ell}$
and a marked point \marginpar{$v_{i}$}$v_{i}\in\partial\Sigma_{i}$
in each boundary component. Let $f\colon\Sigma\to\wedger$ be a map
to the bouquet. We say that $\left(\Sigma,f\right)$ is \textbf{admissible}
for $\wl\in\F_{r}$ if the following two conditions hold:
\begin{enumerate}
\item $\Sigma$ is oriented and compact, with no closed connected components. 
\item The restriction of $f$ to the boundary of $\Sigma$ is homotopic
to $f_{\wl}$ relative to the marked points $v_{1},\ldots,v_{\ell}$.
Namely, for every $i=1,\ldots,\ell$, 
\[
f_{*}\left(\left[\overrightarrow{\partial_{i}\Sigma}\right]\right)=w_{i}\in\pi_{1}\left(\wedger,o\right),
\]
where $\overrightarrow{\partial_{i}\Sigma}$ is the closed loop at
$v_{i}$ around $\partial_{i}\Sigma$ with orientation induced from
the orientation of $\Sigma$.
\end{enumerate}
\end{defn}

In particular, we assume in the above definition that $f\left(v_{i}\right)=o$
for every $i=1,\ldots,\ell$. 

There is a natural equivalence relation between different admissible
pairs: first, if $f_{1},f_{2}\colon\Sigma\to\wedger$ are homotopic
relative to the marked points $v_{1},\ldots,v_{\ell}$, then we think
of $\left(\Sigma,f_{1}\right)$ and $\left(\Sigma,f_{2}\right)$ as
equivalent. We denote by $\left[f\right]$\marginpar{$\left[f\right]$}
the homotopy class of $f$ relative to $v_{1},\ldots,v_{\ell}$. Second,
there is a natural action of $\mcg\left(\Sigma\right)$\marginpar{$\protect\mcg\left(\Sigma\right)$},
the mapping class group of $\Sigma$, on homotopy classes of maps
$\Sigma\to\wedger$, and we define different maps in the same orbit
to be equivalent (see Definition \ref{def:surfaces(w1,...,wl)}).
Here, $\MCG\left(\Sigma\right)$ is defined as the group of homeomorphisms
of $\Sigma$ which fix the boundary $\partial\Sigma$ pointwise, modulo
such homeomorphisms which are isotopic to the identity. The action
of $\mcg\left(\Sigma\right)$ on homotopy classes of maps 
\[
\left\{ \left[f\right]\,\middle|\,\left(\Sigma,f\right)~\mathrm{admissible~for}~\wl\right\} 
\]
is by precomposition: the action of $\left[\rho\right]\in\mcg\left(\Sigma\right)$
on $\left[f\right]$ results in $\left[f\circ\rho^{-1}\right]$. We
gather these considerations in the following definition: 
\begin{defn}
\label{def:surfaces(w1,...,wl)}Let $\left(\Sigma,f\right)$ and $\left(\Sigma',f'\right)$
be admissible for $\wl$. They are \textbf{equivalent}, denoted $\left(\Sigma,f\right)\sim\left(\Sigma',f'\right)$\marginpar{${\scriptstyle \left(\Sigma,f\right)\sim\left(\Sigma',f'\right)}$},
if there is an orientation preserving homeomorphism $\rho\colon\Sigma\to\Sigma'$,
such that for every $i=1,\ldots,\ell$, $\rho\left(v_{i}\right)=v_{i}'$
and $f\simeq f'\circ\rho$ are homotopic relative to the marked points
$v_{1},\ldots v_{\ell}$. We denote by $\left[\left(\Sigma,f\right)\right]$\marginpar{$\left[\left(\Sigma,f\right)\right]$}
the equivalence class of $\left(\Sigma,f\right)$. We denote the set
of equivalence classes by $\surfaces\left(\wl\right)$\marginpar{${\scriptstyle \protect\surfaces\left(\protect\wl\right)}$}:
\[
\surfaces\left(\wl\right)\stackrel{\mathrm{def}}{=}\left\{ \left[\left(\Sigma,f\right)\right]\,\middle|\,\left(\Sigma,f\right)~\mathrm{is~admissible~for}~\wl\right\} .
\]
\end{defn}

The main goal of this paper is to show how one can read the terms
of the Laurent series of $\trwl\left(n\right)$ from this set $\surfaces\left(\wl\right)$
of equivalence classes of pairs of surfaces and maps.

\subsubsection*{The $L^{2}$-Euler characteristic of stabilizers}

The Laurent series of $\trwl\left(n\right)$ gets some contribution
from every\linebreak{}
$\left[\left(\Sigma,f\right)\right]\in\surfaces\left(\wl\right)$.
As stated in Theorem \ref{thm:main} below, this contribution is of
the form $c\cdot n^{\alpha}$, where $c$ and $\alpha$ are integers.
The order of magnitude of the contribution is controlled by the Euler
characteristic of the surface: $\alpha=\chi\left(\Sigma\right)$.
However, to determine the integer coefficient $c$, an important role
is played by the stabilizer of $\left[f\right]$ under the action
of $\mcg\left(\Sigma\right)$, which we denote by \marginpar{$\protect\MCG\left(f\right)$}$\MCG\left(f\right)$:
\[
\MCG\left(f\right)\stackrel{\mathrm{def}}{=}\MCG\left(\Sigma\right)_{\left[f\right]}.
\]
Note that by definition, the elements of $\mcg\left(\Sigma\right)$
permute homotopy classes of maps inside the same equivalence class
$\left[\left(\Sigma,f\right)\right]$. Yet, occasionally, they may
stabilize $\left[f\right]$, in the sense that $f\circ\rho$ and $f$
are homotopic relative to $v_{1},\ldots,v_{\ell}$. Given the class
$\left[\left(\Sigma,f\right)\right]$, the stabilizer $\mcg\left(f\right)$
is defined up to conjugation.

The actual invariant of the stabilizer that appears in the contribution
of $\left[\left(\Sigma,f\right)\right]$ to $\trwl\left(n\right)$
is its \textbf{$L^{2}$-Euler characteristic}. The $L^{2}$-Euler
characteristic of a group is defined for groups with nice enough properties
and can take any real value. It is the alternating sum of the von
Neumann dimensions of the homology groups of a natural chain complex
of modules over the group von Neumann algebra, as we explain in more
detail in Section \ref{subsec:L2-Betti-numbers-and-L2-EC} below,
and see \cite{L}. Thus, to state our main result, we first need the
following auxiliary theorem which is interesting for its own sake. 
\begin{thm}
\label{thm:stabilizers have L2-EC}Let $\Sigma$ be a compact orientable
surface with no closed connected components. Let $f\colon\Sigma\to\wedger$
be a map. Then the stabilizer $\mcg\left(f\right)=\mcg\left(\Sigma\right)_{\left[f\right]}$
has a well-defined $L^{2}$-Euler characteristic. Moreover, this $L^{2}$-Euler
Characteristic is an integer.
\end{thm}

Note that in the statement of the theorem it does not matter whether
$\left[f\right]$ is the homotopy class of $f$ relative to $\partial\Sigma$
or relative to some marked points in every boundary component - this
nuance does not modify the action of $\mcg\left(\Sigma\right)$ on
the homotopy classes of maps.

Theorem \ref{thm:stabilizers have L2-EC} can be strengthened in an
important special case we now introduce:
\begin{defn}
\label{def:null-curve-incompressible}A \emph{null-curve}\textbf{}\marginpar{null-curve}
of $\left(\Sigma,f\right)$ is a non-nullhomotopic simple closed curve
$\gamma$ in $\Sigma$ with $f\left(\gamma\right)$ nullhomotopic
in $\wedger$. A pair $\left(\Sigma,f\right)$ is called \emph{incompressible}\marginpar{incompressible}\emph{
}if it admits no null-curves. It is called \emph{compressible} otherwise.\emph{ }
\end{defn}

If $\left(\Sigma,f\right)$ is admissible for $\wl$ and is compressible,
then one can cut $\Sigma$ along a null-curve, fill the two new boundary
components with discs to obtain $\Sigma'$ and extend $f$ to a map
$f'\colon\Sigma'\to\wedger$. If $\Sigma'$ contains a closed component,
remove it to obtain $\Sigma''$ and let $f''$ denote the restriction
of $f'$ to $\Sigma''$. The new pair $\left(\Sigma'',f''\right)$
is admissible for $\wl$ and satisfies $\chi\left(\Sigma''\right)\ge\chi\left(\Sigma'\right)=\chi\left(\Sigma\right)+2$,
as the possibly closed component of $\Sigma'$ cannot be a sphere.
Thus, pairs $\left(\Sigma,f\right)$ with $\Sigma$ having maximal
Euler characteristic are necessarily incompressible. When $f$ is
incompressible we have the following stronger version of Theorem \ref{thm:stabilizers have L2-EC}:
\begin{thm}
\label{thm:K(g,1) for incompressible}Let $\Sigma$ be a compact orientable
surface with boundary in every connected component, and let $f\colon\Sigma\to\wedger$
be incompressible. Then the stabilizer 
\[
\Gamma=\mcg\left(f\right)=\mcg\left(\Sigma\right)_{\left[f\right]}
\]
admits a finite simplicial complex as a $K\left(\Gamma,1\right)$-space.
In particular, $\Gamma$ has a well-defined Euler characteristic in
the ordinary sense\footnote{The ``ordinary'' Euler characteristic of a group is defined for
a large class of groups of certain finiteness conditions -- see \cite[Chapter IX]{BROWN}.
The simplest case is when a group $\Gamma$ admits a finite $CW$-complex
as Eilenberg-MacLane space of type $K\left(\Gamma,1\right)$, namely,
a path-connected complex with fundamental group isomorphic to $\Gamma$
and a contractible universal cover. In this case, the Euler characteristic
of $\Gamma$ coincides with the Euler characteristic of the $K\left(\Gamma,1\right)$-space. }, which coincides with its $L^{2}$-Euler characteristic. 
\end{thm}

\subsubsection*{Main result}

Our main theorem shows that the Laurent expansion of $\trwl\left(n\right)$
is given by Euler characteristics of both the stabilizers of maps
in $\surfaces\left(\wl\right)$ and also the Euler characteristics
of the surfaces. When the $L^{2}$-Euler characteristic of a group
$\Gamma$ is defined, we denote it by $\chi^{\left(2\right)}\left(\Gamma\right)$\marginpar{$\chi^{\left(2\right)}\left(\Gamma\right)$}.
\begin{thm}[Main Theorem]
\label{thm:main}Let $\wl\in\F_{r}$. For large enough\footnote{As in Proposition \ref{prop:rational expression}, the equality (\ref{eq:main thm})
holds for every $n\ge\max_{x\in B}L_{x}$, where $L_{x}$ is the total
number of appearance of $x^{+1}$ in the words $\wl$. See also Section
\ref{sec:A-Formula-for trwl}.} $n$, 
\begin{equation}
\trwl\left(n\right)=\sum_{\left[\left(\Sigma,f\right)\right]\in\surfaces\left(\wl\right)}\chi^{\left(2\right)}\left(\mcg\left(f\right)\right)\cdot n^{\chi\left(\Sigma\right)}.\label{eq:main thm}
\end{equation}
Indeed, for any given exponent $\chi_{0}$, there are only finitely
many non-zero terms of order $n^{\chi_{0}}$, namely, the set 
\[
\left\{ \left[\left(\Sigma,f\right)\right]\in\surfaces\left(\wl\right)~\,\middle|~\,\chi\left(\Sigma\right)=\chi_{0}~\mathrm{and}~\chi^{\left(2\right)}\left(\mcg\left(f\right)\right)\ne0\right\} 
\]
is finite.
\end{thm}

The last statement of the theorem explains why the theorem yields
a well-defined coefficient for every term in the Laurent series of
$\trwl\left(n\right)$. However, we do not know yet how to derive
from this theorem the rationality of $\trwl\left(n\right)$, which
we prove directly using Weingarten calculus - see Proposition \ref{prop:rational expression}
and Section \ref{sec:A-Formula-for trwl}. This rationality means
that in a way we do not yet fully understand, the $L^{2}$-Euler characteristics
of different pairs $\left[\left(\Sigma,f\right)\right]\in\surfaces\left(\wl\right)$
``know about each other'' -- see Question \ref{enu:rationality from Main}
in Section \ref{sec:Open-Questions}.\medskip{}

As an immediate corollary of Proposition \ref{prop:rational expression}
and Theorem \ref{thm:main}, we get an asymptotic upper bound on $\trwl\left(n\right)$.
Denote\marginpar{${\scriptstyle \protect\ch\left(\protect\wl\right)}$}
\[
\ch\left(\wl\right)\stackrel{\mathrm{def}}{=}\max\left\{ \chi\left(\Sigma\right)\,\middle|\,\left[\left(\Sigma,f\right)\right]\in\surfaces\left(\wl\right)\right\} ,
\]
where $\ch\left(\wl\right)=-\infty$ if $\surfaces\left(\wl\right)$
is empty, which is equivalent to $w_{1}\cdots w_{\ell}\notin\left[\F_{r},\F_{r}\right]$
- see Claims \ref{claim: tr=00003D0 for non-balanced words} and \ref{cor:surfaces not empty}.
A well-known fact going back at least to Culler \cite[Paragraph 1.1]{CULLER}
is that $\mathrm{chi}\left(w\right)=1-2\cdot\mathrm{cl}\left(w\right)$,
where $\mathrm{cl}\left(w\right)$ is the commutator length of $w$,
defined as\footnote{A more general concept of the commutator length was introduced by
Calegari (e.g.~\cite[Definition 2.71]{calegari2009scl}), and applies
to finite sets of words $\wl$. This number can be related, under
certain restrictions, to $\ch\left(\wl\right)$, in a similar fashion
to the $\ell=1$ case.}\marginpar{$\protect\cl\left(w\right)$}
\[
\mathrm{cl}\left(w\right)\stackrel{\mathrm{def}}{=}\min\left\{ g\,\middle|\,w=\left[u_{1},v_{1}\right]\cdots\left[u_{g},v_{g}\right]~\mathrm{with}~u_{i},v_{i}\in\F_{r}\right\} .
\]
Thus,
\begin{cor}
\label{cor:chi and cl}Let $\wl\in\F_{r}$. Then 
\begin{equation}
\trwl\left(n\right)=O\left(n^{\ch\left(\wl\right)}\right).\label{eq:upper bound on trwl}
\end{equation}
In particular, for $w\in\F_{r}$, 
\begin{equation}
\trw\left(n\right)=O\left(n^{\ch\left(w\right)}\right)=O\left(\frac{1}{n^{2\cdot\cl\left(w\right)-1}}\right).\label{eq:upper bound on trw}
\end{equation}
\end{cor}

\begin{rem}
Recall that the Euler characteristic of an orientable compact surface
of genus-$g$ and $\ell$ boundary components is $2-2g-\ell$. Thus,
Theorem \ref{thm:main} yields that the Laurent series of $\trwl\left(n\right)$
is supported on odd (respectively even) powers of $n$ if $\ell$
is odd (respectively even). This is a nice interpretation of a fact
that can also be derived directly from analysis involving Weingarten
calculus.
\end{rem}

\subsubsection*{Algebraic interpretation}

The connection between the commutator length of a word $w$ and $\ch\left(w\right)$
led to the algebraic interpretation (\ref{eq:upper bound on trw})
in Corollary \ref{cor:chi and cl}. This algebraic perspective also
gives an interesting interpretation to our main theorem. Because a
connected surface $\Sigma$ is a $K\left(\pi_{1}\left(\Sigma\right),1\right)$-space,
the Dehn-Nielsen-Baer \label{Dehn-Nielsen-Baer}Theorem states there
is a natural isomorphism between $\mcg\left(\Sigma\right)$ and a
certain subgroup of $\mathrm{Aut}\left(\pi_{1}\left(\Sigma\right)\right)$
(see, for example, \cite[Chapter 8]{FM} for its version for closed
surfaces, and \cite[Thm 2.4]{MP}). For example, if $\Sigma_{g,1}$
is a connected genus $g$ surface with one boundary component, then
$\pi_{1}\left(\Sigma\right)\cong\F_{2g}=\F\left(a_{1},b_{1},\ldots,a_{g},b_{g}\right)$,
and $\mcg\left(\Sigma\right)$ is isomorphic to the stabilizer of
$\left[a_{1},b_{1}\right]\cdots\left[a_{g},b_{g}\right]$ in $\mathrm{Aut}\left(\F_{2g}\right)$
-- stabilizing this element reflects the fact that mapping classes
in $\mcg\left(\Sigma\right)$ fix the boundary of $\Sigma_{g,1}$. 

Along these lines, the set $\surfaces\left(w\right)$ can be interpreted
as equivalence classes of solutions to the equations 
\begin{equation}
\left[u_{1},v_{1}\right]\cdots\left[u_{g},v_{g}\right]=w\label{eq:commutator-equation}
\end{equation}
with $u_{i},v_{i}\in\F_{r}$ and varying $g$, where the equivalence
relation is given by the action of the stabilizer $\mathrm{Aut}\left(\F_{2g}\right)_{\left[a_{1},b_{1}\right]\cdots\left[a_{g},b_{g}\right]}$.
In particular, the pairs $\left[\left(\Sigma,f\right)\right]\in\surfaces\left(w\right)$
with maximal $\chi\left(\Sigma\right)$ correspond to equivalence
classes of solutions to (\ref{eq:commutator-equation}) with $g=\cl\left(w\right)$
minimal. Often, these solutions have trivial stabilizers, in which
case $\chi^{\left(2\right)}\left(\mcg\left(f\right)\right)=1$. For
example, the stabilizer is trivial if the solutions consist of $2g$
free words, or, equivalently, if $f\colon\Sigma\to\wedger$ is $\pi_{1}$-injective
-- see Lemma \ref{lem:pi1-injective means trivial stab}. Thus, one
could say 
\begin{lyxcode}
\textrm{\emph{``The~leading~coefficient~of~$\trw(n)$~counts~the~number~of~equivalence~classes~of~solutions~to~(\ref{eq:commutator-equation})~with~~$g=\cl\left(w\right)$,~up~to~corrections~for~the~existence~of~non-trivial~stabilizers.''}}
\end{lyxcode}

\subsubsection*{Examples\label{subsec:Examples}}

Let us now illustrate Theorem \ref{thm:main} and Corollary \ref{cor:chi and cl}
on some of the examples from Table \ref{tab:examples}. The techniques
by which we obtain some of the details in the following cases are
explained throughout the paper, especially in Section \ref{subsec:Classifying-all-incompressible}. 
\begin{itemize}
\item The commutator length of $w=\left[x,y\right]$ is obviously one, and
there is a single equivalence class of solutions to the equation $\left[u,v\right]=w$,
or, equivalently, a single element $\left[\left(\Sigma,f\right)\right]\in\surfaces\left(w\right)$
with $\chi\left(\Sigma\right)=\ch\left(w\right)=-1$. The stabilizer
$\mcg\left(f\right)$ is trivial and so the first term in the Laurent
expansion of $\tr_{\left[x,y\right]}\left(n\right)$ is $\frac{1}{n}$.
Every other element $\left[\left(\Sigma,f\right)\right]\in\surfaces\left(w\right)$
has $\chi\left(\Sigma\right)\le-3$ and $\chi^{\left(2\right)}\left(\mcg\left(f\right)\right)=0$.
\item We have $\cl\left(\left[x^{3},y\right]\right)=1$. There are exactly
three in-equivalent solutions to $\left[u,v\right]=w$: $\left[x^{3},y\right]$,
$\left[x^{3},yx\right]$ and $\left[x^{3},yx^{2}\right]$. In contrast,
the solution $\left[x^{3},yx^{3}\right]$ is equivalent to $\left[x^{3},y\right]$
because the automorphism of $\F_{2}=\F\left(a,b\right)$ fixing $a$
and mapping $b\to ba$, stabilizes $\left[a,b\right]$. In this case,
all three solutions have trivial stabilizers, hence the leading term
of $\tr_{\left[x^{3},y\right]}$ is $\frac{3}{n}$. It seems like
there are no other elements of $\surfaces\left(\left[x^{3},y\right]\right)$
with non-vanishing $\chi^{\left(2\right)}\left(\mcg\left(f\right)\right)$
(at least there are none with $\chi\left(\Sigma\right)\ge-7$).
\item In general, if $\cl\left(w\right)=1$, then every solution to $\left[u,v\right]=w$
has trivial stabilizer, because $u$ and $v$ are necessarily free
words inside $\F_{r}$ (namely, $\left\langle u,v\right\rangle \cong\F_{2}$).
Thus, for words of commutator length one we have $\trw\left(n\right)=\frac{K}{n}+O\left(\frac{1}{n^{3}}\right)$,
where $K$ is the number of equivalence classes of ways to write $w$
as a commutator. Likewise, every solution to (\ref{eq:commutator-equation})
with $\left\langle u_{1},v_{1},\ldots,u_{g},v_{g}\right\rangle \cong\F_{2g}$
has a trivial stabilizer -- see Lemma \ref{lem:pi1-injective means trivial stab}.
\item For $w=\left[x,y\right]^{2}$ we have $\cl\left(w\right)=2$. There
is a single equivalence class of solutions to (\ref{eq:commutator-equation})
with $g=2$, and $\mcg\left(f\right)\cong\F_{5}$. As a bouquet of
five circles is a $K\left(\F_{5},1\right)$-space, we have $\chi\left(\F_{5}\right)=-4$.
This explains the leading term of $\tr_{\left[x,y\right]^{2}}$.
\item The somewhat surprising fact that $\cl\left(\left[x,y\right]^{3}\right)=2$
was pointed out in \cite{CULLER}. (Interestingly, Culler shows in
that paper that $\cl\left(\left[x,y\right]^{n}\right)=\left\lfloor \frac{n}{2}\right\rfloor +1$.)
For example, $\left[x,y\right]^{3}=\left[xyx^{-1},y^{-1}xyx^{-2}\right]\left[y^{-1}xy,y^{2}\right]$.
There are nine in-equivalent solutions in this case, each with a trivial
stabilizer. This explain the leading term $\frac{9}{n^{3}}$. There
is a single pair $\left[\left(\Sigma,f\right)\right]\in\surfaces\left(\left[x,y\right]^{3}\right)$
with $\chi\left(\Sigma\right)=-5$ and $\chi^{\left(2\right)}\left(\mcg\left(f\right)\right)\ne0$.
The stabilizer in this single pair satisfies $\chi^{\left(2\right)}\left(\mcg\left(f\right)\right)=81$.
This explain the term $\frac{81}{n^{5}}$.
\item The word $w=\left[x,y\right]\left[x,z\right]$ has $\cl\left(w\right)=2$
and admits a single solution to (\ref{eq:commutator-equation}) with
$g=2$. The stabilizer of this solution is isomorphic to $\mathbb{Z}$
and $\chi^{\left(2\right)}\left(\mathbb{Z}\right)=\chi\left(\mathbb{Z}\right)=0$.
Note that this explains why the coefficient of $n^{-3}$ vanishes,
but not why $\tr_{w}\left(n\right)=0$. 
\item The word $w=\left[x,y\right]\left[x,z\right]\left[x,t\right]$ has
$\cl\left(w\right)=3$ and admits a single solution to (\ref{eq:commutator-equation})
with $g=3$. The stabilizer of this solution is isomorphic to $\mathbb{Z}\times\F_{2}$
and $\chi^{\left(2\right)}\left(\mathbb{Z}\times\F_{2}\right)=\chi\left(\mathbb{Z}\times\F_{2}\right)=0$
(consult \cite[Page 59]{MP} for more details).
\item For $w_{1}=x^{2}y^{2}$ and $w_{2}=xy^{-3}x^{-3}y$ we have $\ch\left(w_{1},w_{2}\right)=-2$.
There are four $\left[\left(\Sigma,f\right)\right]\in\surfaces\left(w_{1},w_{2}\right)$
with $\chi\left(\Sigma\right)=-2$, each with a trivial stabilizer,
hence the leading term $\frac{4}{n^{2}}$. All $\chi\left(\Sigma\right)=-4$
solutions have $\chi^{\left(2\right)}\left(\mcg\left(f\right)\right)=0$,
while there is a single solution with non-vanishing contribution and
$\chi\left(\Sigma\right)=-6$, for which $\chi^{\left(2\right)}\left(\mcg\left(f\right)\right)=-16$.
\item For every $w\ne1$, $\ch\left(w,w^{-1}\right)=0$ because there is
an obvious annulus in\linebreak{}
$\surfaces\left(w,w^{-1}\right)$. In both examples of this sort in
Table \ref{tab:examples}, there is a single such annulus, and with
a trivial stabilizer, hence the leading term $1$. In both cases there
are no other incompressible pairs in $\surfaces\left(w,w^{-1}\right)$,
but while for $w=x^{2}yxy^{-1}$, it seems that every compressible
pair has vanishing contribution to $\tr_{w,w^{-1}}\left(n\right)$,
for $w=x^{2}y^{2}xy^{-1}$ there is a compressible pair $\left[\left(\Sigma,f\right)\right]$
with $\chi\left(\Sigma\right)=-4$ and $\chi^{\left(2\right)}\left(\mcg\left(f\right)\right)=-4$.
\end{itemize}

\subsubsection*{Compressible vs.~incompressible pairs $\left[\left(\Sigma,f\right)\right]\in\protect\surfaces\left(\protect\wl\right)$}

The difference between compressible and incompressible pairs $\left[\left(\Sigma,f\right)\right]\in\surfaces\left(\wl\right)$
is already apparent from the fact that Theorem \ref{thm:K(g,1) for incompressible},
or at least its proof, apply only to the incompressible case. The
crucial property of incompressible pairs will be pointed out in Section
\ref{subsec:Incompressible-maps} in the sequel of the paper. But
there are some further differences we point out here.

First, there are finitely many incompressible elements in $\surfaces\left(\wl\right)$
-- see Corollary \ref{cor:finitely many incompressible}. Because
highest-Euler-characteristic elements are always incompressible, we
deduce there are finitely many elements $\left[\left(\Sigma,f\right)\right]\in\surfaces\left(\wl\right)$
with $\chi\left(\Sigma\right)=\ch\left(\wl\right)$. In addition,
as the examples above illustrate, the stabilizer $\mcg\left(f\right)$
of an incompressible solution is often trivial.

In contrast, there are infinitely many compressible elements in $\sur\left(\wl\right)$.
In fact, there are often even infinitely many compressible elements
$\left[\left(\Sigma,f\right)\right]$ with $\chi\left(\Sigma\right)=\chi_{0}$
for a given non-maximal $\chi_{0}$, namely, for $\chi_{0}=\ch\left(\wl\right)-2k$
with $k\in\mathbb{Z}_{\ge1}$, although, as stated in Theorem \ref{thm:main},
almost all of them have zero contribution to $\trwl\left(n\right)$.
Moreover, the stabilizer of a compressible pair is never trivial:
a Dehn twist along a null-curve is a non-trivial element in the stabilizer.

\subsection{Related lines of work}

The evaluation of the integrals in (\ref{eq:def of Tr_w}) is a fundamental
issue relating to several different areas of mathematics.

\paragraph{I. Matrix integrals in Gaussian models}

The connection between the enumeration of graphs on surfaces and matrix
integrals in the classical GUE, GOE and GSE models was first established
by 't Hooft \cite{Hooft1974} and later rediscovered by Harer and
Zagier \cite{HARERZAGIER}. For example, let $\mathrm{GUE}\left(n\right)$
denote the probability space of $n\times n$ Hermitian complex matrices
endowed with complex Gaussian measure on each entry, where the $\left(i,j\right)$
entry is independent of all other entries except for $\left(j,i\right)$.
The following equation \cite[Proposition 3.3.1]{Lando2004} illustrates
this connection: 
\begin{equation}
\mathbb{E}_{H\in\mathrm{GUE}\left(n\right)}\left[\left(\mathrm{tr}H\right)^{\alpha_{1}}\left(\mathrm{tr}H^{2}\right)^{\alpha_{2}}\cdots\left(\mathrm{tr}H^{k}\right)^{\alpha_{k}}\right]=\sum_{\sigma}n^{F\left(\sigma\right)}.\label{eq:GUE expectation}
\end{equation}
The summation on the right hand side is over ribbon graphs (also known
as fat-graphs) with $\alpha_{i}$ vertices of degree $i$ for $i=1,\ldots,k$,
where $\sigma$ is a perfect matching of the half-edges emanating
from these vertices. The exponent $F\left(\sigma\right)$ is the number
of faces in the embedding of the resulting ribbon graph on the surface
of smallest possible genus. We stress that (\ref{eq:GUE expectation})
is only an illustration of the theory, and there are many generalizations
(e.g., for integrals over tuples of independent Hermitian matrices)
and deep applications. For an excellent presentation of this theory,
we refer the reader to \cite[Chapter 3]{Lando2004}.

There are many similarities between this by-now classical theory and
the theory we develop in the current paper. For example, apart from
the natural emergence of surfaces, the combinatorial formulas for
$\trwl\left(n\right)$ we develop in Section \ref{sec:A-Formula-for trwl}
also involve a summation over perfect matchings. In addition, these
matrix integrals over $\mathrm{GUE}$ were used, inter alia, to compute
the Euler characteristic of the mapping class group of closed surfaces
with punctures \cite{HARERZAGIER,Penner1988}. In fact, these Euler
characteristics appear as coefficients in certain generating functions
for integrals as in (\ref{eq:GUE expectation}) (e.g., \cite[Theorem 1.1 and Corollary 3.1]{Penner1988}).

There are also substantial differences. Among others, $\U\left(n\right)$
being a group endowed with Haar measure means that integrals as in
(\ref{eq:GUE expectation}) over a single Haar-random element can
be completely computed using theoretical properties of the Haar measure,
as was done in \cite{Diaconis1994}, and the computation becomes more
interesting when multiple random elements are involved. It also means
that word-measure on $\U\left(n\right)$ have nice properties, such
as being $\Aut\left(\F_{r}\right)$-invariant, as explained in the
following paragraph. In addition, the crucial role played here by
maps from the surfaces to the bouquet of circles is completely absent
in the classical theory. Another difference is that the summation
in the right hand side of (\ref{eq:GUE expectation}) is finite, with
the exponents increasing as the Euler characteristic of the surface
decreases. The best analogue in the current paper (\ref{eq:Laurent-formula})
involves an infinite summation with exponents decreasing together
with the Euler characteristic of the surfaces. Finally, there is also
a big difference in the role played by Euler characteristics of (subgroups
of) the mapping class groups of surfaces.

\paragraph{II. Word measures on groups}

The same way $w\in\F_{r}$ induces a measure on $\U\left(n\right)$,
it also induces a probability measure on any compact group (consult
\cite{hui2015waring} for recent results and references concerning
the image of the word map $w\colon G^{r}\to G$ on compact Lie groups
including $\U\left(n\right)$). By showing that Nielsen moves on $w$
do not affect the resulting word measure, it is easy to see that two
words in the same $\mathrm{Aut}\left(\F_{r}\right)$-orbit in $\F_{r}$
induce the same measure on every compact group (see \cite[Section 2.2]{MP}
for a proof). But is this the only reason for two words to have such
a strong connection? A version of the following conjecture appears,
for example, in \cite[Question 2.2]{Amit2011} and in \cite[Conjecture 4.2]{Shalev2013}.
\begin{conjecture}
\label{conj:aner}If two words $w_{1},w_{2}\in\F_{r}$ induce the
same measure on every compact group, then there exists $\phi\in\mathrm{Aut}\left(\F_{r}\right)$
with $w_{2}=\phi\left(w_{1}\right)$.
\end{conjecture}

\label{primitivity conjecture in S_n}A special case of this conjecture
deals with the $\mathrm{Aut}\left(\F_{r}\right)$-orbit of the single-letter
word $x_{1}$, namely, with the set of primitive words. Several researchers
have asked whether words inducing the Haar measure on every compact
group are necessarily primitive. This was settled to the affirmative
in \cite[Theorem 1.1]{PP15} using word measures on symmetric groups.
In subsequent work \cite{mp2019surface}, we use the results in this
paper and, mainly, Corollary \ref{cor:chi and cl}, to prove that
if a word $w$ induces the same measure as $u_{g}=\left[x_{1},y_{1}\right]\cdots\left[x_{g},y_{g}\right]$
on every compact group then $w=\phi\left(u_{g}\right)$ for some $\phi\in\mathrm{Aut}\left(\F_{r}\right)$.

Short of proving Conjecture \ref{conj:aner}, one could hope to collect
as many invariants of words as possible that can be determined by
word measures induced on groups. For example, $\cl\left(w\right)$,
the commutator length of a word, and more generally, $\ch\left(\wl\right)$,
the highest possible Euler characteristic of a surface in $\surfaces\left(\wl\right)$,
play an important role in our results. However, because the coefficient
of $n^{\ch\left(\wl\right)}$ in $\trwl\left(n\right)$ occasionally
vanishes, it is not clear whether $\cl\left(w\right)$ or $\ch\left(\wl\right)$
are determined by word measures on $\U\left(n\right)$. 

In contrast, the measures do determine a related number, the \emph{stable
commutator length }of $w$. This algebraic quantity is defined by
\begin{equation}
\mathrm{scl}(w)\equiv\lim_{m\to\infty}\frac{\cl(w^{m})}{m}.\label{eq:scl}
\end{equation}
(There is an analogous definition for finite sets of words.) There
is a deep theory behind this invariant, and for background we refer
to the short survey \cite{CALWHATIS} and long one \cite{calegari2009scl}
by Calegari. Relying on the rationality result of Calegari \cite{CALRATIONAL}
that shows, in particular, that $\mathrm{scl}$ takes on rational
values in $\F_{r}$, we are able to show the following:
\begin{cor}
\label{cor:can hear scl}The stable commutator length of a word $w\in\left[\F_{r},\F_{r}\right]$
can be determined by the measures it induces on unitary groups in
the following way:

\begin{equation}
\mathrm{scl}\left(w\right)=\inf_{\ell>0;\,j_{1},\ldots,j_{\ell}>0}\frac{-\lim_{n\to\infty}\log_{n}\left|\tr_{w^{j_{1}},\ldots,w^{j_{\ell}}}\left(n\right)\right|}{2\left(j_{1}+\ldots+j_{\ell}\right)}.\label{eq:read scl}
\end{equation}
\end{cor}

A similar result is true for the stable commutator length of several
words. We explain how Corollary \ref{cor:can hear scl} follows from
Theorem \ref{thm:main} and Calegari's rationality theorem in Section
\ref{subsec:Stable-commutator-length}.
\begin{rem}
In fact, with regards to Conjecture \ref{conj:aner}, word-measures
on $\U\left(n\right)$ alone do not suffice and Conjecture \ref{conj:aner}
is not true if ``every compact group'' is replaced by ``$\U\left(n\right)$
for all $n$''. Indeed, for every $w\in\F_{r}$ and every $n$, the
$w$-measure on $\U\left(n\right)$ is identical to the $w^{-1}$-measure.
However, in general, $w$ and $w^{-1}$ belong to two different $\mathrm{Aut}\left(\F_{r}\right)$-orbits.
See also Question \ref{enu:primitivity conjecture} in Section \ref{sec:Open-Questions}.
\end{rem}

\paragraph{III. Harmonic analysis on representation varieties.}

The integral in (\ref{eq:def of Tr_w}) can be viewed as an integral
over the space of representations $\Hom\left(\F_{r},\U(n)\right)$
and in fact, as an integral over the representation variety
\[
\Rep\left(\F_{r},\U(n)\right)=\Hom\left(\F_{r},\U(n)\right)/\U(n)
\]
since the functions $\mathrm{tr}\circ w_{i}$ are invariant under
$\U(n)$-conjugation, and so is the Haar measure. More generally,
if $\Sigma_{g}$ is the closed genus $g$ surface, then the spaces
$\Rep\left(\pi_{1}\left(\Sigma_{g}\right),\U(n)\right)$ are of interest
in geometry\emph{, }via\emph{ `Higher Teichmüller theory}', dynamics
as pioneered by Goldman \cite{Goldman1997}, and mathematical physics
\cite{Witten1991}. For an overview see \cite{Labourie2013}. For
any closed curve on the surface, there is a natural function (\emph{Wilson
loop})\emph{ }on the representation variety, given by the trace of
the image of that curve in a given representation. It is natural to
ask what is the integral of this function with respect to the volume
form given by the Atiyah-Bott-Goldman symplectic structure on $\Rep\left(\pi_{1}\left(\Sigma_{g}\right),\U(n)\right)$
\cite{AB,GOLDMANSYMPLECTIC}. Our work answers this question for representations
of free groups.

\paragraph{IV. Free probability theory.}

Voiculescu proved in \cite[Theorem 3.8]{Voiculescu1991} that for
$w\neq1$, 
\begin{equation}
\trw(n)=o(1),\quad n\to\infty.\label{eq:firstorder}
\end{equation}
This is referred to the \emph{asymptotic {*}-freeness }of the non-commutative
independent random variables $\left(u_{1},\ldots,u_{r}\right)\in\U\left(n\right)^{r}$,
meaning that in the limit they can be modeled by the ``Free Probability
Theory'' developed by Voiculescu (see, for example, the monograph
\cite{Voiculescu1992}). Such asymptotic freeness results are known
for broad families of ensembles, including general Gaussian random
matrices (due to Voiculescu in the same paper \cite[Theorem 2.2]{Voiculescu1991}).
In later works (\ref{eq:firstorder}) is strengthened to $\trw\left(n\right)=O\left(\frac{1}{n}\right)$
whenever $w\ne1$ \cite{MSS07,Radulescu06}. Our work gives quantitative
bounds on the decay rate of $\trw(n)$ (in many cases, from above
and below) - see Corollary \ref{cor:chi and cl}.

More generally, free probabilists are interested in the limit of $\trwl\left(n\right)$
as $n\to\infty$. This is given by the following corollary of our
main result, which is essentially \cite[Theorem 2]{MSS07} and \cite[Theorem 4.1]{Radulescu06}:
\begin{cor}
Let $\wl\in\F_{r}$, each not equal to $1$, and write $w_{i}=u_{i}^{d_{i}}$
where $u_{i}\in\F_{r}$ is a non-power and $d_{i}\ge1$. Then the
limit 
\begin{equation}
\lim_{n\to\infty}\trwl\left(n\right)\label{eq:limit of trwl}
\end{equation}
exists, and is equal to the number of ways to match $\wl$ in pairs
so that each word is conjugate to the inverse of its mate, times $\sqrt{\prod_{i=1}^{\ell}d_{i}}$.
\end{cor}

\begin{proof}
As $\wl\ne1$, there are no surfaces of positive Euler characteristic
in\linebreak{}
$\surfaces\left(\wl\right)$. The only possible surface $\Sigma$
in this collection with $\chi\left(\Sigma\right)=0$ is a disjoint
union of annuli. The stabilizer $\MCG\left(f\right)$ is always trivial
in this case, so the limit in (\ref{eq:limit of trwl}) is equal to
the number of such surfaces in $\surfaces\left(\wl\right)$. If $w$
and $w'$ are the words at the boundary of an annulus, then necessarily
$w'$ is conjugate to $w^{-1}$. Moreover, if $w=u^{d}$ with $u\in\F_{r}$
a non-power and $d\ge1$, then the number of non-equivalent annuli
in $\surfaces\left(w,w^{-1}\right)$ is exactly $d$. This yields
the answer above.
\end{proof}

\subsection{Paper organization}

In Section \ref{sec:A-Formula-for trwl} we show how surfaces emerge
in the computation of $\trwl\left(n\right)$, present a formula for
$\trwl\left(n\right)$ as a finite sum (Theorem \ref{thm:trwl as finite sum})
which yields Proposition \ref{prop:rational expression}, and then
a second formula for $\trwl\left(n\right)$, this time as an infinite
sum, but where the contribution of every surface is $\pm n^{\alpha}$
(Theorem \ref{thm:combinatorial Laurent expansion of trwl}). Section
\ref{subsec:Maps-from-the} then explains how every surface we constructed
admits a natural map to the bouquet which makes it (a representative
of) an element in $\surfaces\left(\wl\right)$. Thus, one can group
together all the surfaces we constructed in the second formula (from
Theorem \ref{thm:combinatorial Laurent expansion of trwl}) that belong
to the same class $\left[\left(\Sigma,f\right)\right]\in\surfaces\left(\wl\right)$.
Our main result then reduces to showing that the total contribution
of this set of surfaces is equal to $\chi^{\left(2\right)}\left(\mcg\left(f\right)\right)\cdot n^{\chi\left(\Sigma\right)}$,
as stated in Theorem \ref{thm:main} -{}- this reduction is the content
of Theorem \ref{thm:EC of a single (S,f)}.

In Sections \ref{sec:A-Complex-of} and \ref{sec:The-Action-of MCG(f) on T}
we fix $\left[\left(\Sigma,f\right)\right]\in\surfaces\left(\wl\right)$
and prove Theorem \ref{thm:EC of a single (S,f)}: in Section \ref{sec:A-Complex-of}
we define the complex of transverse maps realizing $\left(\Sigma,f\right)$,
and prove it is a finite-dimensional contractible complex. In Section
\ref{sec:The-Action-of MCG(f) on T} we analyze the action of $\mcg\left(f\right)$
on this complex, show that the finite orbits of cells in this action
are in one-to-one correspondence with the surfaces we constructed
in Section \ref{sec:A-Formula-for trwl}, and finish the proof of
Theorems \ref{thm:EC of a single (S,f)}, \ref{thm:stabilizers have L2-EC}
and \ref{thm:main}. In Sections \ref{subsec:Incompressible-maps}
and \ref{subsec:Nonfiniteness} we discuss the difference between
the compressible case and the incompressible one, and prove Theorem
\ref{thm:K(g,1) for incompressible}. 

Section \ref{sec:Further-Applications} contains three applications:
in Section \ref{subsec:Stable-commutator-length} we discuss stable
commutator length and how it is determined by the $w$-measures on
$\U\left(n\right)$, thus proving Corollary \ref{cor:can hear scl};
in Section \ref{subsec:Classifying-all-incompressible} we explain
how our analysis yields a simple straight-forward algorithm to classify
all incompressible solutions in $\surfaces\left(\wl\right)$, and,
in particular, all solutions to the commutator equation (\ref{eq:commutator-equation})
with $g=\cl\left(w\right)$; and in Section \ref{subsec:Cohomological-dimension-of}
we explain why $\mcg\left(f\right)$ has finite cohomological dimension.
Section \ref{sec:Open-Questions} contains some open questions.
\begin{rem}
\label{rem:w_i ne 1}The case where some of the words among $\wl$
are trivial is not interesting in the point of view of estimating
the integrals $\trwl\left(n\right)$: $\tr_{w_{1},\ldots,w_{\ell-1},1}\left(n\right)=n\cdot\tr_{w_{1},\ldots,w_{\ell-1}}\left(n\right)$.
Yet, some of the results, such as Theorem \ref{thm:stabilizers have L2-EC},
are interesting in this case too. Despite that, for the sake of simplicity,
we assume throughout the rest of the paper that $w_{i}\ne1$ for $i=1,\ldots,\ell$:
this allows us to avoid extra case analysis at some points and shorten
the arguments a bit. We stress, though, that all the results hold
in the general case, and the proofs hold after, possibly, minor adaptations
(with the one exception of Lemma \ref{lem:X is finite dimensional}
where, if one allows trivial words, the bound should be modified).
\end{rem}

\begin{rem}
\label{remark:MP16}The unpublished manuscript \cite{MP} is based
on an earlier stage of the current research. It contains some of the
results of the current paper -- mainly the results for incompressible
maps -- although with quite a different presentation of the proofs.
Since writing \cite{MP}, we have extended our results a great deal,
and decided to rewrite everything in a whole new paper. To keep the
current paper in manageable size, we include only ingredients that
are necessary for proving and clarifying our results. Occasionally,
we refer here to the more elaborated \cite{MP} for some background
material, which is not used in the proofs.
\end{rem}

\begin{rem}
A sequel paper \cite{mp2019orthogonal} shows how the ideas in the
current paper can be extended and twisted to also deal with integrals
over the orthogonal and compact symplectic groups.
\end{rem}

\subsection*{Acknowledgments}

We would like to thank Danny Calegari, Alexei Entin, Mark Feighn,
Alex Gamburd, Yair Minsky, Mark Powell, Peter Sarnak, Zlil Sela, Karen
Vogtmann, Alden Walker, Avi Wigderson, Qi You, and Ofer Zeitouni for
valuable discussions about this work.

Part of this research was carried out during research visits of the
first named author (Magee) to the Institute for Advanced Study in
Princeton, and we would like to thank the I.A.S. for making these
visits possible. 

This is a pre-print of an article published in Inventiones Mathematicae.
The final authenticated version is available online at: https://doi.org/10.1007/s00222-019-00891-4.

\section{Combinatorial formulas for $\protect\tr_{w_{1},\ldots,w_{\ell}}\left(n\right)$
using surfaces\label{sec:A-Formula-for trwl}}

In this section we recall basic results about the Weingarten calculus
for integrals over $\U\left(n\right)$, and derive formulas for $\trwl\left(n\right)$
which involve surfaces. But first, we explain why $\trwl\left(n\right)$
vanishes in the ``non-balanced'' case, where the total exponent
of some letter is not zero.:
\begin{claim}
\label{claim: tr=00003D0 for non-balanced words}Let $\wl\in\F_{r}$.
If $w_{1}w_{2}\cdots w_{\ell}\notin\left[\F_{r},\F_{r}\right]$ then 
\begin{enumerate}
\item $\trwl\left(n\right)\equiv0$.
\item The set $\surfaces\left(\wl\right)$ is empty.
\end{enumerate}
\end{claim}

\begin{proof}
$\left(1\right)$ The assumption $\wl\notin\left[\F_{r},\F_{r}\right]$
is equivalent to that there is some $j\in\left[r\right]$ so that
$\alpha_{j}$, the sum of exponents of the letter $x_{j}$ in $\wl$,
satisfies $\alpha_{j}\ne0$. As the Haar measure of a compact group
is invariant under left multiplication by any element, and the diagonal
central matrix $e^{i\theta}I_{n}$ is in ${\cal U}\left(n\right)$
for $\theta\in\left[0,2\pi\right]$, we obtain
\begin{eqnarray*}
 &  & \trwl\left(n\right)=\\
 &  & =\mathbb{E}_{A_{1},\ldots,A_{r}\in{\cal U}\left(n\right)}\left[\mathrm{tr}\left(w_{1}\left(A_{1},\ldots,A_{j},\ldots,A_{r}\right)\right)\cdots\mathrm{tr}\left(w_{\ell}\left(A_{1},\ldots,A_{j},\ldots,A_{r}\right)\right)\right]\\
 &  & =\mathbb{E}_{A_{1},\ldots,A_{r}\in{\cal U}\left(n\right)}\left[\mathrm{tr}\left(w_{1}\left(A_{1},\ldots,e^{i\theta}A_{j},\ldots,A_{r}\right)\right)\cdots\mathrm{tr}\left(w_{\ell}\left(A_{1},\ldots,e^{i\theta}A_{j},\ldots,A_{r}\right)\right)\right]\\
 &  & =e^{i\theta\alpha_{j}}\cdot\trwl\left(n\right).
\end{eqnarray*}
The first statement follows as this equality holds for every $\theta\in\left[0,2\pi\right]$.

$\left(2\right)$ The second statement follows from the fact that
in every connected, orientable, compact surface $\Sigma$ with boundary,
the product in $\pi_{1}\left(\Sigma\right)$ of loops around the boundary
components belongs to $\left[\pi_{1}\left(\Sigma\right),\pi_{1}\left(\Sigma\right)\right]$.
\end{proof}

\subsection{Weingarten function and integrals over ${\cal U}\left(n\right)$\label{subsec:Weingarten-function-and-Collins-Sniady}}

The ``Weingarten calculus'' for computing integrals over unitary
groups with respect to the Haar measure was developed in a series
of papers, most notably \cite{weingarten1978asymptotic,xu1997random,collins2003moments,CS}.
It is based on the Schur-Weyl duality (see Remark \ref{rem:schur-weyl}
below), and allows the computation of integrals over the entries of
unitary matrices and their complex conjugates, as depicted in Theorem
\ref{thm:collins-sniady} below. This computation is given in terms
of the Weingarten function, which we now describe.

Let $\mathbb{Q}\left(n\right)$ denote the field of rational functions
with rational coefficients in the variable $n$. Let $S_{L}$\marginpar{$S_{L}$}
denote the symmetric group on $L$ elements. For every $L\in\mathbb{Z}_{\ge1}$,
the Weingarten function $\wg_{L}$ maps $S_{L}$ to $\mathbb{Q}\left(n\right)$.
We think of such functions as elements of the group ring $\mathbb{Q}\left(n\right)\left[S_{L}\right]$.
\begin{defn}
\label{def:weingarten}The \textbf{Weingarten function} \marginpar{$\protect\wg_{L}$}$\wg_{L}:S_{L}\to\mathbb{Q}\left(n\right)$
is the inverse, in the group ring $\mathbb{Q}\left(n\right)\left[S_{L}\right]$,
of the function $\sigma\mapsto n^{\#\mathrm{cycles\left(\sigma\right)}}$. 
\end{defn}

That the function $\sigma\mapsto n^{\#\mathrm{cycles\left(\sigma\right)}}$
is invertible for every $L$ follows from \cite[Proposition 2.3]{CS}
and the discussion following it. In particular, $\wg_{L}\left(\sigma\right)$
is in $\mathbb{Q}\left(n\right)$ for every $\sigma\in S_{L}$. Clearly,
$\wg_{L}$ is constant on conjugacy classes. For example, for $L=2$,
the inverse of $\left(n^{2}\cdot\mathrm{id}+n\cdot\left(12\right)\right)\in\mathbb{Q}\left(n\right)\left[S_{2}\right]$
is $\left(\frac{1}{n^{2}-1}\cdot\mathrm{id}-\frac{1}{n\left(n^{2}-1\right)}\cdot\left(12\right)\right)$,
so $\wg_{2}\left(\mathrm{id}\right)=\frac{1}{n^{2}-1}$ while $\wg_{2}\left(\left(12\right)\right)=\frac{-1}{n\left(n^{2}-1\right)}$.
The values of $\wg_{3}$ are
\[
\mathrm{id}\mapsto\frac{n^{2}-2}{n\left(n^{2}-1\right)\left(n^{2}-4\right)}\,\,\,\,\,\,\left(12\right)\mapsto\frac{-1}{\left(n^{2}-1\right)\left(n^{2}-4\right)}\,\,\,\,\,\,\left(123\right)\mapsto\frac{2}{n\left(n^{2}-1\right)\left(n^{2}-4\right)}.
\]
Collins and \'{S}niady provide an explicit formula for $\wg_{L}$
in terms of the irreducible characters of $S_{L}$ and Schur polynomials
\cite[Equation (13)]{CS}: 
\[
\wg_{L}\left(\sigma\right)=\frac{1}{\left(L!\right)^{2}}\sum_{\lambda\vdash L}\frac{\chi_{\lambda}\left(e\right)^{2}}{d_{\lambda}\left(n\right)}\chi_{\lambda}\left(\sigma\right),
\]
where $\lambda$ runs over all partitions of $L$, $\chi_{\lambda}$
is the character of $S_{L}$ corresponding to $\lambda$, and $d_{\lambda}\left(n\right)$
is the number of semistandard Young tableaux with shape $\lambda$,
filled with numbers from $\left[n\right]$. A well known formula for
$d_{\lambda}\left(n\right)$ is $d_{\lambda}\left(n\right)=\frac{\chi_{\lambda}\left(e\right)}{L!}\prod_{\left(i,j\right)\in\lambda}\left(n+j-i\right)$,
where $\left(i,j\right)$ are the coordinates of cells in the Young
diagram with shape $\lambda$ (e.g.~\cite[Section 4.3, Equation (9)]{Fulton1997}).
Thus,
\begin{cor}
\label{cor:poles-of-Wg}For $\sigma\in S_{L}$, $\wg_{L}\left(\sigma\right)$
may have poles only at integers $n$ with $-L<n<L$.
\end{cor}

Below we use the following properties of the Weingarten function.
The standard norm \foreignlanguage{american}{of $\rho\in S_{L}$,
denoted $\left\Vert \rho\right\Vert $, is the shortest length of
a product of transpositions giving $\rho$, and is equal to $L-\#\mathrm{cycles\left(\rho\right)}$.}
\selectlanguage{american}%
\begin{thm}
\label{thm:properties of Wg}Let $\pi\in S_{L}$ be a permutation. 
\selectlanguage{english}%
\begin{enumerate}
\item \label{enu:leading term of Wg}\cite[Corollary 2.7]{CS} Leading term:
\begin{equation}
\wg_{L}\left(\pi\right)=\frac{\moeb\left(\pi\right)}{n^{L+\left\Vert \pi\right\Vert }}+O\left(\frac{1}{n^{L+\left\Vert \pi\right\Vert +2}}\right),\label{eq:leading term of Wg}
\end{equation}
where\marginpar{$\protect\moeb\left(\sigma\right)$} 
\begin{equation}
\moeb\left(\pi\right)=\mathrm{sgn}\left(\pi\right)\prod_{i=1}^{k}c_{|C_{i}|-1},\label{eq:mobius}
\end{equation}
with\footnote{The function $\moeb$ is the Möbius function on a natural poset structure
on $S_{L}$ -- see, for instance, \cite[Lectures 10 and 23]{Nica2006}.} $C_{1},\ldots,C_{k}$ the cycles composing $\pi$, and $c_{m}=\frac{(2m)!}{m!(m+1)!}$
being the $m$-th Catalan number. 
\selectlanguage{american}%
\item \label{enu:Asymptotic-expansion of Wg}\cite[Theorem 2.2]{collins2003moments}
Asymptotic expansion: 
\begin{equation}
\wg_{L}\left(\pi\right)=\frac{1}{n^{L}}\sum_{k\in\mathbb{Z}_{\ge0}}\sum_{\begin{gathered}{\scriptstyle \rho_{1},\ldots,\rho_{k}\in S_{L}\setminus\left\{ \mathrm{id}\right\} }\\
{\scriptstyle \mathrm{s.t.}~\,\rho_{1}\cdots\rho_{k}=\pi}
\end{gathered}
}\frac{\left(-1\right)^{k}}{n^{\left\Vert \rho_{1}\right\Vert +\ldots+\left\Vert \rho_{k}\right\Vert }}.\label{eq:asympt expansion of Wg}
\end{equation}
\end{enumerate}
\end{thm}

\selectlanguage{english}%
(In (\ref{eq:asympt expansion of Wg}), when $\pi=\mathrm{id}$, there
is a term $\frac{1}{n^{L}}$ coming from $k=0$.) 

The Weingarten function is used in the following formula of Collins
and \'{S}niady, which evaluates integrals of monomials in the entries
$A_{i,j}$ and their conjugates $\overline{A_{i,j}}$ of a Haar-random
unitary matrix $A\in{\cal U}\left(n\right)$. As in the proof of Claim
\ref{claim: tr=00003D0 for non-balanced words}, this integral vanishes
whenever the monomial is not balanced, namely whenever the number
of $A_{i,j}$'s is different from the number of $\overline{A_{i,j}}$'s. 
\begin{thm}
\label{thm:collins-sniady}\cite[Proposition 2.5]{CS} Let $L$ be
a positive integer and $\left(i_{1},\ldots,i_{L}\right)$, $\left(j_{1},\ldots,j_{L}\right)$,
$\left(i'_{1},\ldots,i'_{L}\right)$ and $\left(j'_{1},\ldots,j'_{L}\right)$
be $L$-tuples of positive integers. Then for every $n$ for which
the expression
\begin{equation}
\mathbb{E}{}_{A\in{\cal U}(n)}\left[A_{i_{1},j_{1}}A_{i_{2},j_{2}}\ldots A_{i_{L},j_{L}}\overline{A_{i'_{1},j'_{1}}}\overline{A_{i'_{2},j'_{2}}}\ldots\overline{A_{i'_{L},j'_{L}}}\right]\label{eq:collins-sniady}
\end{equation}
makes sense, namely, for $n\ge\max\left\{ i_{1},\ldots,i_{L},j_{1},\ldots,j_{L},i'_{1},\ldots,i'_{L},j'_{1},\ldots,j'_{L}\right\} $,
(\ref{eq:collins-sniady}) is equal to the evaluation of $n$ in a
rational function, which is given by 
\begin{equation}
\sum_{\sigma,\tau\in S_{L}}\delta_{i_{1}i'_{\sigma(1)}}\ldots\delta_{i_{L}i'_{\sigma(L)}}\delta_{j_{1}j'_{\tau(1)}}\ldots\delta_{j_{L}j'_{\tau(L)}}\wg_{L}\left(\sigma^{-1}\tau\right).\label{eq:moment}
\end{equation}
\end{thm}

Put differently, the rational function is given by $\sum_{\sigma,\tau}\wg_{L}\left(\sigma^{-1}\tau\right)$,
where $\sigma$ runs over all rearrangements of $\left(i'_{1},\ldots,i'_{L}\right)$
which make it identical to $\left(i_{1},\ldots,i_{L}\right)$, and
$\tau$ runs over all rearrangements of $\left(j'_{1},\ldots,j'_{L}\right)$
which make it identical to $\left(j_{1},\ldots,j_{L}\right)$. In
particular, the possible poles of the Weingarten function at $n$,
for every $n\ge\max\left\{ i_{1},\ldots,j'_{L}\right\} $, are guaranteed
to cancel out in this summation (see the example following Proposition
2.5 in \cite{CS}). 
\begin{rem}
\label{rem:schur-weyl}The basis for the Weingarten calculus is the
Schur-Weyl duality for $\U\left(n\right)$. One version of this duality
is the following: let $V=\mathbb{C}^{n}$. A unitary matrix $A\in\U\left(n\right)$
acts on the space of functionals $W=\left(V^{\otimes L}\otimes\left(V^{*}\right)^{\otimes L}\right)^{*}$
by 
\[
\left(A\theta\right)\left(v_{1}\otimes\ldots\otimes v_{L}\otimes\varphi_{1}\otimes\ldots\otimes\varphi_{L}\right)=\theta\left(A^{-1}v_{1}\otimes\ldots\otimes A^{-1}v_{L}\otimes A^{-1}\varphi_{1}\otimes\ldots\otimes A^{-1}\varphi_{L}\right),
\]
where we think of $\varphi\in V^{*}$ as a column vector in $\mathbb{C}^{n}$
whose value on $v\in V$ is $\varphi^{*}v\in\mathbb{C}$. Every permutation
$\sigma\in S_{L}$ yields a functional in $W$ defined by:
\[
\theta_{\sigma}\left(v_{1}\otimes\ldots\otimes v_{q}\otimes\varphi_{1}\otimes\ldots\otimes\varphi_{q}\right)=\left(\varphi_{1}^{*}v_{\sigma^{-1}\left(1\right)}\right)\cdots\left(\varphi_{q}^{*}v_{\sigma^{-1}\left(q\right)}\right).
\]
The Schur-Weyl duality says that this embedding of $\mathbb{C}\left[S_{L}\right]$
in $W$ is precisely the set of $\U\left(n\right)$-invariant functionals
in $W$. The family of integrals in (\ref{eq:collins-sniady}) can
be presented as a single functional on $V^{\otimes L}\otimes\left(V^{*}\right)^{\otimes L}\otimes V^{\otimes L}\otimes\left(V^{*}\right)^{\otimes L}$,
which is $\U\left(n\right)\times\U\left(n\right)$-invariant because
the Haar measure is both left- and right-invariant. This roughly explains
why one can expect a result of the type of Theorem \ref{thm:collins-sniady}.
\end{rem}

\subsection{Surfaces from matchings of letters}

Based on Theorem \ref{thm:collins-sniady} we show that $\trwl\left(n\right)$
is a rational expression in $n$, and give concrete formulas which
involve surfaces. These surfaces are constructed from matchings of
the letters in $\wl$, and we begin by describing this construction. 

Recall that $B$ denotes a fixed basis for $\F_{r}$. Following Claim
\ref{claim: tr=00003D0 for non-balanced words}, we assume that $w_{1}\cdots w_{\ell}\in\left[\F_{r},\F_{r}\right]$,
namely, that for every letter $x\in B$, the total number of instances
of $x^{+1}$ in $\wl$ is equal to that of $x^{-1}$, and we denote
this number by $L_{x}\in\mathbb{Z}_{\ge0}$\marginpar{$L_{x}$}. In
particular, $\left|w_{1}\right|+\ldots+\left|w_{\ell}\right|=2\sum_{x\in B}L_{x}$.
We also denote by $\match_{x}\left(\wl\right)$\marginpar{$\protect\match_{x}$}
the set of bijections from the instances of $x^{+1}$ to the instances
of $x^{-1}$, so that $\left|\match_{x}\left(\wl\right)\right|=L_{x}!$. 

Let \marginpar{$\kappa$}$\kappa=\left\{ \kappa_{x}\right\} _{x\in B}\in\left(\mathbb{Z}_{\ge0}\right)^{B}$
be an assignment of a non-negative integer to every basis element.
We denote by \marginpar{$\protect\match^{\kappa}$}$\match^{\kappa}\left(\wl\right)$
the Cartesian product of sets of matchings, with $\kappa_{x}+1$ copies
of $\match_{x}\left(\wl\right)$ for every $x\in B$, namely,
\[
\match^{\kappa}\left(\wl\right)\stackrel{\mathrm{def}}{=}\prod_{x\in B}\match_{x}\left(\wl\right)^{\kappa_{x}+1}.
\]
The following definition presents the construction of a surface from
an element of\linebreak{}
$\match^{\kappa}\left(\wl\right)$. We use the notation $\left[k\right]\stackrel{\mathrm{def}}{=}\left\{ 0,1,\ldots,k\right\} $\marginpar{$\left[k\right]$}
for a non-negative integer $k$.
\begin{defn}
\label{def:surface from matchings}Let $\wl\in\F_{r}\setminus\left\{ 1\right\} $
be a balanced set of words, let $\kappa\in\left(\mathbb{Z}_{\ge0}\right)^{B}$
and let $\sigma\in\match^{\kappa}\left(\wl\right)$ be a tuple of
matchings. We denote by $\sigma_{x,0},\ldots,\sigma_{x,\kappa_{x}}$
the $\kappa_{x}+1$ matchings from $\match_{x}\left(\wl\right)$ in
$\sigma$. From this data we construct a surface \marginpar{$\Sigma_{\sigma}$}$\Sigma_{\sigma}$
as a CW-complex as follows:
\begin{itemize}
\item For $1\ne w\in\F_{r}$ define $S^{1}\left(w\right)$\marginpar{$S^{1}\left(w\right)$}
to be an oriented $1$-sphere $S^{1}$ with additional marked points
as follows: there are\footnote{We use $\left|w\right|$ to denote the number of letters in $w$.}
$\left|w\right|$ points marked $o$, which we call $o$-points\marginpar{$o$-point}.
These points cut the $1$-sphere into $\left|w\right|$ intervals,
which are in bijection with the letters of $w$, in the suitable cyclic
order. For every letter of $w$, if the letter is $x^{\pm1}$, we
mark additional $\kappa_{x}+1$ points on the interval corresponding
to that letter. These marked points are labeled $\left(x,0\right),\ldots,\left(x,\kappa_{x}\right)$
and are ordered according to the orientation of $S^{1}\left(w\right)$
if the letter is $x^{+1}$, or in reverse orientation if the letter
is $x^{-1}$. We call a point labeled $\left(x,j\right)$ for some
$x\in B$ and $j\in\left[\kappa_{x}\right]$ an $\left(x,j\right)$-point\marginpar{$\left(x,j\right)$-point}
or a $z$-point\marginpar{$z$-point} if the exact $x$ and $j$ do
not matter.
\item The one-dimensional skeleton of $\Sigma_{\sigma}$ consists of $S^{1}\left(w_{1}\right),\ldots,S^{1}\left(w_{\ell}\right)$,
together with additional $\sum_{x\in B}L_{x}\left(\kappa_{x}+1\right)$
edges (1-cells), referred to as \emph{matching-edges}: for every $x\in B$
and $j\in\left[\kappa_{x}\right]$, introduce $L_{x}$ edges describing
the matching $\sigma_{x,j}$. Namely, for every $x^{+1}$-letter $\lambda$
of $\wl$, introduce an edge between the $\left(x,j\right)$-point
on the interval corresponding to $\lambda$ and the $\left(x,j\right)$-point
on the interval corresponding to the $x^{-1}$-letter $\sigma_{x,j}\left(\lambda\right)$.
This is illustrated in the left part of Figure \ref{fig:surface from matchings}.
\item Finally, $2$-cells are attached as follows: consider cycles in the
$1$-skeleton which are obtained by starting at some marked point
in $S^{1}\left(w_{i}\right)$ for some $i=1,\ldots,\ell$, moving
orientably along $S^{1}\left(w_{i}\right)$ until the next $z$-point,
then following the matching-edge emanating from this $z$-point and
arriving at some $z$-point in $S^{1}\left(w_{i'}\right)$ for some
$i'$, then moving orientably along $S^{1}\left(w_{i'}\right)$ until
the next $z$-point, continuing along the matching-edge and so on
until a cycle has been completed. A 2-cell (a disc) is glued along
every such cycle.
\item From the construction of $\Sigma_{\sigma}$, it is clear it is a surface,
with boundary $S^{1}\left(w_{1}\right)\sqcup\ldots\sqcup S^{1}\left(w_{\ell}\right)$
and with orientation prescribed from the boundary. Moreover, every
$2$-cell $D$ belongs to exactly one of the following categories: 
\begin{itemize}
\item Either there is an $o$-point at every component of $\partial D\cap\partial\Sigma_{\sigma}$,
in which case we call $D$ an $o$-disc\marginpar{$o$-disc}, 
\item or, $\partial D$ contains no $o$-points, in which case we call $D$
a $z$-disc\marginpar{$z$-disc}. In this case, there are some $x\in B$
with $\kappa_{x}\ge1$ and $j\in\left[\kappa_{x}-1\right]$ such that
the marked points in $\partial D$ are exactly of two types: $\left(x,j\right)$-points
and $\left(x,j+1\right)$-points. In this case we call the $z$-disc
$D$ also an $\left(x,j\right)$-disc\marginpar{$\left(x,j\right)$-disc}.
See Figure \ref{fig:surface from matchings}.
\end{itemize}
\item Let $\chi\left(\sigma\right)$\marginpar{$\chi\left(\sigma\right)$}
denote the Euler characteristic of this surface, namely $\chi\left(\sigma\right)\stackrel{\mathrm{def}}{=}\chi\left(\Sigma_{\sigma}\right)$.
\end{itemize}
\begin{figure}[t]
\centering{}\includegraphics[viewport=0bp 0bp 448bp 200bp,scale=0.85]{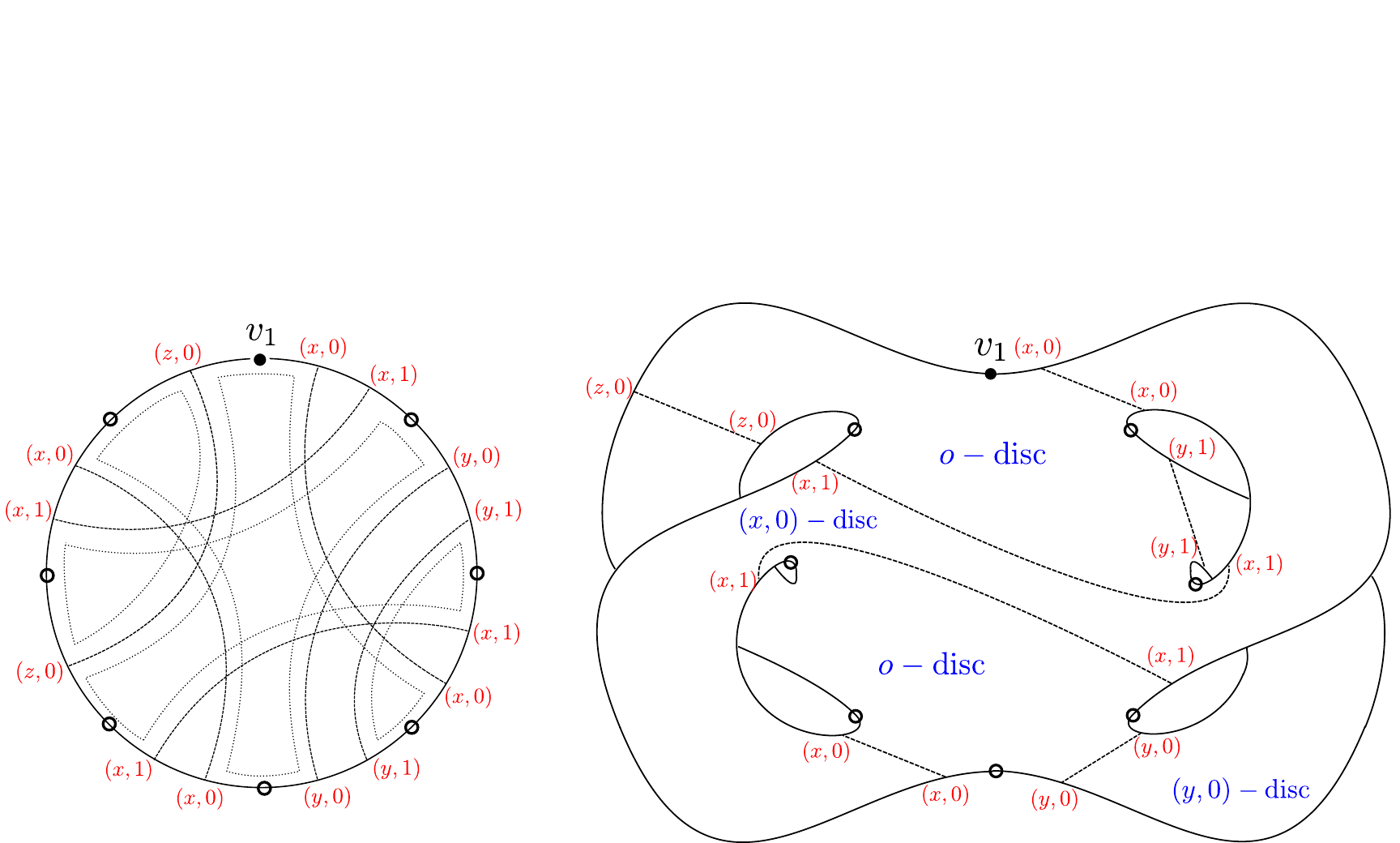}\caption{\label{fig:surface from matchings}On the left is the 1-skeleton of
$\Sigma_{\sigma}$ for $w=\left[x,y\right]\left[x,z\right]=x_{1}y_{2}X_{3}Y_{4}x_{5}z_{6}X_{7}Z_{8}$,
with $\kappa_{x}=\kappa_{y}=1$ and $\kappa_{z}=0$ and with the matchings
$\sigma_{x,0}=\left(x_{1}\protect\mapsto X_{3};~x_{5}\protect\mapsto X_{7}\right)$,
$\sigma_{x,1}=\left(x_{1}\protect\mapsto X_{7};~x_{5}\protect\mapsto X_{3}\right)$,
$\sigma_{y,0}=\sigma_{y,1}=\left(y_{2}\protect\mapsto Y_{4}\right)$
and $\sigma_{z,0}=\left(z_{6}\protect\mapsto Z_{8}\right)$. Dashed
lines are matching-edges. The dotted lines trace the boundaries of
the two $o$-discs to be glued in (see Definition \ref{def:surface from matchings}).
Two additional discs, a $\left(x,0\right)$-disc and a $\left(y,0\right)$-disc
are glued along the other types of cycles one can follow (unmarked).
The eight $o$-points are marked by $V_{1}$ and black circles. The
resulting surface $\Sigma_{\sigma}$ is on the right and is a genus-$2$
surface with one boundary component. In this case, $\chi\left(\sigma\right)=\chi\left(\Sigma_{\sigma}\right)=-3$.}
\end{figure}
\end{defn}

\subsection{A formula for $\protect\trwl\left(n\right)$ as a rational expression}

Our first formula for $\trwl\left(n\right)$ is a finite sum over
pairs of matchings for every letters, namely over elements in $\match^{\kappa}\left(\wl\right)$
with $\kappa_{x}=1$ for every $x\in B$. We denote this $\kappa$
by $\kappa\equiv1$. In particular, this formula proves Proposition
\ref{prop:rational expression}.
\begin{thm}[$\trwl\left(n\right)$ as finite sum]
\label{thm:trwl as finite sum} Let $\wl\in\F_{r}$ be a balanced
set of words. 
\begin{enumerate}
\item If\footnote{Interestingly, very similar constraints on $n$ appear in a formula
giving the expected trace of $w$ in $r$ uniform $n\times n$ \emph{permutation}
matrices as a rational expression in $n$ -- see \cite[Section 5]{Puder2014}.} $n\ge L_{x}$ for every $x\in B$, then
\begin{equation}
\trwl\left(n\right)=\sum_{\sigma\in\match^{\kappa\equiv1}}\left(\prod_{x\in B}\mathrm{Wg}_{L_{x}}\left(\sigma_{x,0}^{-1}\sigma_{x,1}\right)\right)\cdot n^{\#\left\{ o\mathrm{-discs~in~}\Sigma_{\sigma}\right\} }\label{eq:trwl as finite sum}
\end{equation}
(here $\sigma_{x,0}^{-1}\sigma_{x,1}$ is a permutation of the $x^{+1}$-letters
of $\wl$).
\item For $n\ge\max_{x\in B}L_{x}$, the function $\trwl\left(n\right)$
is a computable rational function in $n$.
\item For $\sigma\in\match^{\kappa\equiv1}\left(\wl\right)$, let $\sigma_{0}=\left(\sigma_{x,0}\right)_{x\in B}$
and $\sigma_{1}=\left(\sigma_{x,1}\right)_{x\in B}$ denote two matchings
of the ``positive'' letters of $\wl$ to the ``negative'' ones.
Then the summand in (\ref{eq:trwl as finite sum}) corresponding to
$\sigma$ is 
\begin{equation}
\moeb\left(\sigma_{0}^{-1}\sigma_{1}\right)\cdot n^{\chi\left(\sigma\right)}+O\left(n^{\chi\left(\sigma\right)-2}\right).\label{eq:leading term of contribution of sigma}
\end{equation}
\end{enumerate}
\end{thm}

\begin{proof}
Part $2$ follows from (\ref{eq:trwl as finite sum}) as every value
of the Weingarten function is computable and in $\mathbb{Q}\left(n\right)$.
We now prove part $\left(1\right)$, which we do by way of an example.
Let $w_{1}=xyx^{-2}y$ and $w_{2}=xy^{-2}$. Then,
\begin{align}
 & \tr_{w_{1},w_{2}}\left(n\right)=\mathbb{E}_{\left(A,B\right)\in{\cal U}(n)\times{\cal U}\left(n\right)}\left[\mathrm{tr}\left(ABA^{-2}B\right)\cdot\mathrm{tr}\left(AB^{-2}\right)\right]\nonumber \\
 & =\mathbb{E}_{\left(A,B\right)\in{\cal U}(n)\times{\cal U}\left(n\right)}\left[\left(\sum_{i,j,k,\ell,m\in\left[n\right]}A_{i,j}\cdot B_{j,k}\cdot A_{\,\,k,\ell}^{-1}\cdot A_{~\ell,m}^{-1}\cdot B_{m,i}\right)\left(\sum_{I,J,K\in\left[n\right]}A_{I,J}\cdot B_{~J,K}^{-1}\cdot B_{\,\,K,I}^{-1}\right)\right]\nonumber \\
 & =\sum_{i,j,k,\ell,m,I,J,K\in\left[n\right]}\mathbb{E}_{\left(A,B\right)\in{\cal U}(n)\times{\cal U}\left(n\right)}\left[A_{i,j}\cdot B_{j,k}\cdot\overline{A_{\ell,k}}\cdot\overline{A_{m,\ell}}\cdot B_{m,i}\cdot A_{I,J}\cdot\overline{B_{K,J}}\cdot\overline{B_{I,K}}\right]\nonumber \\
 & =\sum_{i,j,k,\ell,m,I,J,K\in\left[n\right]}\left(\mathbb{E}_{A\in{\cal U}(n)}\left[A_{i,j}\cdot A_{I,J}\cdot\overline{A_{\ell,k}}\cdot\overline{A_{m,\ell}}\right]\right)\cdot\left(\mathbb{E}_{B\in{\cal U}(n)}\left[B_{j,k}\cdot B_{m,i}\cdot\overline{B_{K,J}}\cdot\overline{B_{I,K}}\right]\right).\label{eq:step1}
\end{align}
Note that there is a clear correspondence between the $o$-points
in $S^{1}\left(w_{1}\right)$ and the indices $i,j,k,\ell,m$ and
between the $o$-points in $S^{1}\left(w_{2}\right)$ and the indices
$I,J,K$ (see Figure \ref{fig:surface for formula}). 
\begin{figure}
\centering{}\includegraphics[viewport=0bp 20bp 448bp 250bp,scale=0.85]{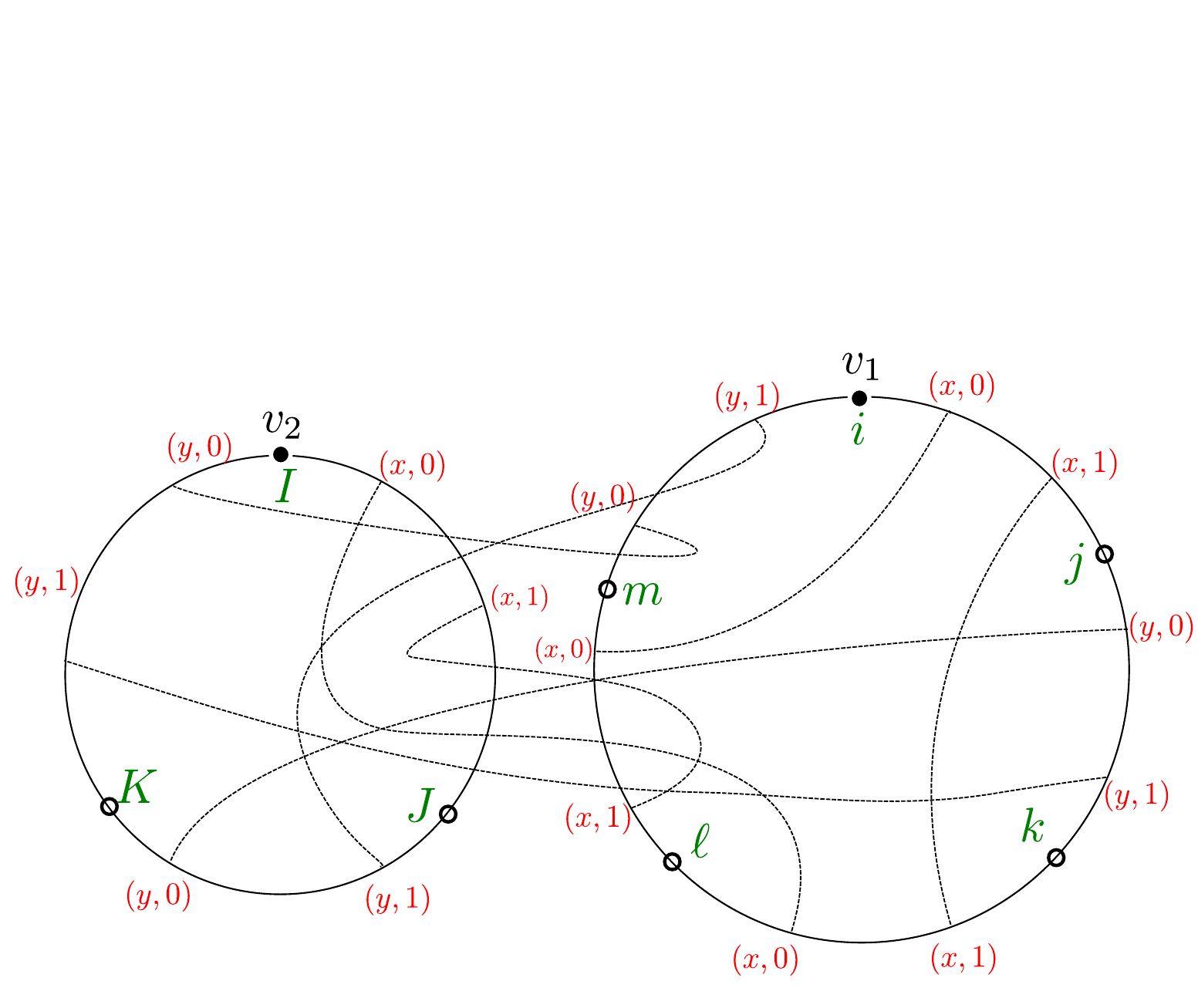}\caption{\label{fig:surface for formula}The 1-skeleton of the surface $\Sigma_{\sigma}$
for the tuple of matchings $\sigma\in\protect\match^{\kappa\equiv1}\left(xyx^{-2}y,xy^{-2}\right)$
specified in the proof of Theorem \ref{thm:trwl as finite sum}. The
$o$-points are identified with the indices $i,j,k,\ell,m$ and $I,J,K$
that appear in the computation of $\protect\tr_{xyx^{-2}y,xy^{-2}}\left(n\right)$.
The $o$-discs of $\Sigma_{\sigma}$ (two of these in the current
case) are in one-to-one correspondence with the blocks of indices
determined by $\sigma$, and for every $x\in B$, the $\left(x,0\right)$-discs
(one for each letter in the current case) are in one-to-one correspondence
with the cycles of the permutation $\sigma_{x,0}^{-1}\sigma_{x,1}$.}
\end{figure}

Now we use Theorem \ref{thm:collins-sniady} to replace each of the
two integrals inside the sum by a summation over pairs of matchings.
For the first integral we go over all bijections $\left\{ i,I\right\} \overset{\sim}{\to}\left\{ \ell,m\right\} $
and $\left\{ j,J\right\} \overset{\sim}{\to}\left\{ k,\ell\right\} $,
and we think of them as elements $\sigma_{x,0},\sigma_{x,1}\in\match_{x}\left(w_{1},w_{2}\right)$
by thinking of a matching between two $z$-points as a matching of
the adjacent $o$-points. For example, the $\left(x,0\right)$-point
in the first letter of $w_{1}$ is adjacent to the $o$-point $i$,
and the $\left(x,1\right)$-point in the same letter is adjacent to
the $o$-point $j$. Similarly, we go over all bijections $\sigma_{y,0}$
and $\sigma_{y,1}$ for the second integral. Changing the order of
summation, we sum first over $\sigma_{x,0}$, $\sigma_{x,1}$, $\sigma_{y,0}$
and $\sigma_{y,1}$, and only then over the indices $i,j,\ldots,K$.
This turns (\ref{eq:step1}) into a sum over $\match^{\kappa\equiv1}\left(w_{1},w_{2}\right)$. 

For every set of $\sigma\in\match^{\kappa\equiv1}\left(w_{1},w_{2}\right)$,
we only need to count the number of evaluations of $i,j,\ldots,L$
which ``agree'' with the permutations. For example, consider the
case where

\begin{equation}
\begin{gathered}\underline{\sigma_{x,0}}\\
i\mapsto m\\
I\mapsto\ell
\end{gathered}
\,\,\,\,\begin{gathered}\underline{\sigma_{x,1}}\\
j\mapsto k\\
J\mapsto\ell
\end{gathered}
\,\,\,\,\begin{gathered}\underline{\sigma_{y,0}}\\
j\mapsto K\\
m\mapsto I
\end{gathered}
\,\,\,\,\begin{gathered}\underline{\sigma_{y,1}}\\
k\mapsto K\\
i\mapsto J
\end{gathered}
\label{eq:example - sigma}
\end{equation}
(these are the matchings described in Figure \ref{fig:surface for formula}).
Note that in this case, both the permutation $\sigma_{x,0}^{-1}\sigma_{x,1}$
and the permutation $\sigma_{y,0}^{-1}\sigma_{y,1}$ are a $2$-cycle.
Hence, by Theorem \ref{thm:collins-sniady}, the summand corresponding
to these matchings is
\[
\wg_{2}\left(\left(1~2\right)\right)\cdot\wg_{2}\left(\left(1~2\right)\right)\cdot\sum_{i,j,k,\ell,m,I,J,K\in\left[n\right]}\delta_{im}\delta_{I\ell}\delta_{jk}\delta_{J\ell}\delta_{jK}\delta_{mI}\delta_{kK}\delta_{iJ},
\]
and the product inside the last sum is 1 (and not 0) if and only if
$i=m=I=\ell=J$ and $j=k=K$. Here, two indices must have the same
value if and only if they belong to the same $o$-disc in $\Sigma_{\sigma}$,
hence there are exactly $n^{\#o\mathrm{-discs~in~}\Sigma_{\sigma}}$
contributing values of the indices, each contributing 1 to the summation.
For $\sigma$ we defined in (\ref{eq:example - sigma}) this number
is $n^{2}$, and the total contribution of this $\sigma$ is, thus,
$\wg_{2}\left(\left(1~2\right)\right)^{2}\cdot n^{2}=\frac{1}{\left(n^{2}-1\right)^{2}}$.
The total summation over all the $16$ elements of $\match^{\kappa\equiv1}\left(w_{1},w_{2}\right)$
is $\frac{1}{n^{2}-1}$. Since the same argument works for every $\wl\in\F_{r}$,
this proves part $\left(1\right)$.

Recall that for $\pi\in S_{L}$, we have $\left\Vert \pi\right\Vert =L-\#\mathrm{cycles}\left(\pi\right)$.
The number of cycles in the permutation $\sigma_{x,0}^{-1}\sigma_{x,1}\in S_{L_{x}}$
is equal to the number of $\left(x,0\right)$-discs in $\Sigma_{\sigma}$.
Hence, by Theorem \ref{thm:properties of Wg}(\ref{enu:leading term of Wg}),
\begin{align*}
\prod_{x\in B}\mathrm{Wg}_{L_{x}}\left(\sigma_{x,0}^{-1}\sigma_{x,1}\right) & =\prod_{x\in B}\left[\frac{\moeb\left(\sigma_{x,0}^{-1}\sigma_{x,1}\right)}{n^{L_{x}+\left\Vert \sigma_{x,0}^{-1}\sigma_{x,1}\right\Vert }}+O\left(\frac{1}{n^{L_{x}+\left\Vert \sigma_{x,0}^{-1}\sigma_{x,1}\right\Vert +2}}\right)\right]\\
 & =\prod_{x\in B}\left[\frac{\moeb\left(\sigma_{x,0}^{-1}\sigma_{x,1}\right)}{n^{2L_{x}-\#\left\{ \left(x,0\right)\mathrm{-discs~in~\Sigma_{\sigma}}\right\} }}+O\left(\frac{1}{n^{2L_{x}-\#\left\{ \left(x,0\right)\mathrm{-discs~in~\Sigma_{\sigma}}\right\} +2}}\right)\right]\\
 & =\frac{\moeb\left(\sigma_{0}^{-1}\sigma_{1}\right)}{n^{2L-\#\left\{ z\mathrm{-discs~in~\Sigma_{\sigma}}\right\} }}+O\left(\frac{1}{n^{2L-\#\left\{ z\mathrm{-discs~in~\Sigma_{\sigma}}\right\} -2}}\right),
\end{align*}
where $L=\sum_{x\in B}L_{x}$ is the total number of positive letters
in $\wl$. We are done proving part $\left(3\right)$ as 
\[
\chi\left(\sigma\right)=\chi\left(\Sigma_{\sigma}\right)=-2L+\#\left\{ \mathrm{discs~in~\!}\Sigma_{\sigma}\right\} .
\]
\end{proof}

\subsection{A formula for $\protect\trwl\left(n\right)$ as Laurent expansion\label{subsec:Laurent expansion of trwl}}

We now give an alternative formula for $\trwl\left(n\right)$, which
also uses surfaces constructed from matchings of the letters of $\wl$.
The sum in (\ref{eq:trwl as finite sum}) is finite, it proves the
rationality of $\trwl\left(n\right)$, and allows a finite algorithm
to compute it. The alternative formula we introduce next has the disadvantage
that it is an infinite sum (unless $L_{x}\le1$ for every $x\in B$).
However, it has the advantage of simplifying greatly the contribution
of every surface involved in the computation, as well as being an
important step towards establishing Theorem \ref{thm:main}. This
formula is derived from (\ref{eq:trwl as finite sum}) together the
asymptotic expansion of the Weingarten function developed in \cite{collins2003moments}
and depicted in Theorem \ref{thm:properties of Wg}(\ref{enu:Asymptotic-expansion of Wg})
above\footnote{Novaes \cite{Novaes} has recently obtained a combinatorial formula
for the Weingarten function in terms of maps on surfaces; our approach
here is different and incorporates that we are integrating over independent
unitary matrices, which naturally leads to considerations about infinite
groups.}.

The formula uses a restricted set of tuples of matchings which, for
a given $\kappa\in\left(\mathbb{Z}_{\ge0}\right)^{B}$, we denote
by $\matchr^{\kappa}\left(\wl\right)$\marginpar{$\protect\matchr^{\kappa}$}:
this is the subset of $\match^{\kappa}\left(\wl\right)$ with the
restriction that no two adjacent matchings are identical, namely,
that $\sigma_{x,j}\ne\sigma_{x,j+1}$ for every $x\in B$ and $0\le j\le\kappa_{x}-1$.
We also denote by $\matchr^{*}\left(\wl\right)$\marginpar{$\protect\matchr^{*}$}
the union of restricted sets of matchings over all possible $\kappa$:
\[
\matchr^{*}\left(\wl\right)\stackrel{\mathrm{def}}{=}\coprod_{\kappa\in\left(\mathbb{Z}_{\ge0}\right)^{B}}\matchr^{\kappa}\left(\wl\right),
\]
and for $\sigma\in\matchr^{*}\left(\wl\right)$ denote by $\kappa\left(\sigma\right)$
and $\kappa_{x}\left(\sigma\right)$ the corresponding values of $\kappa$
and $\kappa_{x}$. Also, for $\kappa\in\left(\mathbb{Z}_{\ge0}\right)^{B}$
let $\left|\kappa\right|\stackrel{\mathrm{def}}{=}\sum_{x\in B}\kappa_{x}$\marginpar{$\left|\kappa\right|$}.
\begin{thm}[Laurent Combinatorial Formula for $\trwl\left(n\right)$]
\label{thm:combinatorial Laurent expansion of trwl}Let $\wl\in\F_{r}$
be a balanced set of words. If $n\ge L_{x}$ for every $x\in B$,
then
\begin{equation}
\trwl\left(n\right)=\sum_{\sigma\in\matchr^{*}\left(\wl\right)}\left(-1\right)^{\left|\kappa\left(\sigma\right)\right|}n^{\chi\left(\sigma\right)}.\label{eq:Laurent-formula}
\end{equation}
\end{thm}

\begin{proof}
This proof relies on grouping the summands in (\ref{eq:Laurent-formula})
according to the ``extreme'' bijections $\left\{ \sigma_{x,0},\sigma_{x,\kappa_{x}}\right\} _{x\in B}$
and show that the total contribution of the summands with extreme
bijections $\tau=\left\{ \tau_{x,0},\tau_{x,1}\right\} _{x}\in\match^{\kappa\equiv1}\left(\wl\right)$
is equal to $\left(\prod_{x\in B}\mathrm{Wg}_{L_{x}}\left(\tau_{x,0}^{-1}\circ\tau_{x,1}\right)\right)\cdot n^{\#\left\{ o\mathrm{-discs~in~}\Sigma_{\tau}\right\} }$.
This is enough by Theorem \ref{thm:trwl as finite sum}.

\selectlanguage{american}%
In (\ref{eq:asympt expansion of Wg}) above, substitute $\theta_{i}=\rho_{1}\cdots\rho_{i}$
to obtain
\[
\wg_{L}\left(\pi\right)=\frac{1}{n^{L}}\sum_{k\in\mathbb{Z}_{\ge0}}\sum_{\mathrm{id}=\theta_{0}\ne\theta_{1}\ne\ldots\ne\theta_{k-1}\ne\theta_{k}=\pi}\frac{\left(-1\right)^{k}}{n^{\left\Vert \theta_{0}^{-1}\theta{}_{1}\right\Vert +\left\Vert \theta_{1}^{-1}\theta_{2}\right\Vert +\ldots+\left\Vert \theta_{k-1}^{-1}\theta_{k}\right\Vert }}.
\]
Substituting $L=L_{x}$ and $\pi=\tau_{x,0}^{-1}\cdot\tau_{x,1}$,
we get 
\[
\wg_{L_{x}}\left(\tau_{x,0}^{-1}\cdot\tau_{x,1}\right)=\frac{1}{n^{L_{x}}}\sum_{k\in\mathbb{Z}_{\ge0}}\sum_{\begin{gathered}{\scriptstyle \theta_{0},\ldots,\theta_{k}\in\mathrm{Sym}\left(x^{+1}\mathrm{-letters~of}~\wl\right)}\\
{\scriptstyle \mathrm{s.t.~}\mathrm{id}=\theta_{0}\ne\theta_{1}\ne\ldots\ne\theta_{k-1}\ne\theta_{k}=\tau_{x,0}^{-1}\tau_{x,1}}
\end{gathered}
}\frac{\left(-1\right)^{k}}{n^{\left\Vert \theta_{0}^{-1}\theta{}_{1}\right\Vert +\left\Vert \theta_{1}^{-1}\theta_{2}\right\Vert +\ldots+\left\Vert \theta_{k-1}^{-1}\theta_{k}\right\Vert }}.
\]
Multiplying all permutations from the left by $\tau_{x,0}$ and substituting
$\sigma_{i}=\tau_{x,0}\theta_{i}$, one obtains:
\begin{equation}
\wg_{L_{x}}\left(\tau_{x,0}^{-1}\cdot\tau_{x,1}\right)=\frac{1}{n^{L_{x}}}\sum_{k\in\mathbb{Z}_{\ge0}}\sum_{\begin{gathered}{\scriptstyle \sigma_{0},\ldots,\sigma_{k}\in\match_{x}\left(\wl\right)}\\
{\scriptstyle \mathrm{s.t.~}\tau_{x,0}=\sigma_{0}\ne\sigma_{1}\ne\ldots\ne\sigma_{k-1}\ne\sigma_{k}=\tau_{x,1}}
\end{gathered}
}\frac{\left(-1\right)^{k}}{n^{\left\Vert \sigma_{0}^{-1}\sigma{}_{1}\right\Vert +\left\Vert \sigma_{1}^{-1}\sigma_{2}\right\Vert +\ldots+\left\Vert \sigma_{k-1}^{-1}\sigma_{k}\right\Vert }}.\label{eq:wg as infinite sum of bijections}
\end{equation}
Note that, by construction, the number of $o$-discs in $\Sigma_{\sigma}$
depends solely on the extreme bijections in $\sigma$, namely on $\left\{ \sigma_{x,0},\sigma_{x,\kappa_{x}}\right\} _{x\in B}$.
Thus, together with (\ref{eq:trwl as finite sum}) we obtain
\begin{align*}
 & \trwl\left(n\right)=\sum_{\tau\in\match^{\kappa\equiv1}\left(\wl\right)}\left(\prod_{x\in B}\mathrm{Wg}_{L_{x}}\left(\tau_{x,0}^{-1}\circ\tau_{x,1}\right)\right)\cdot n^{\#\left\{ o\mathrm{-discs~in~}\Sigma_{\tau}\right\} }\\
\\
 & =\sum_{\tau\in\match^{\kappa\equiv1}}\frac{n^{\#\left\{ o\mathrm{-discs~in}~\Sigma_{\tau}\right\} }}{n^{\sum_{x}L_{x}}}\sum_{\kappa\in\left(\mathbb{Z}_{\ge0}\right)^{B}}\sum_{\begin{gathered}{\scriptstyle \sigma\in\matchr^{\kappa}}~s.t.\\
{\scriptstyle \sigma_{x,0}=\tau_{x,0}~\&~\sigma_{x,\kappa_{x}}=\tau_{x,1}}
\end{gathered}
}\frac{\left(-1\right)^{\left|\kappa\right|}}{n^{\sum_{x}\sum_{j=0}^{\kappa_{x}-1}\left\Vert \sigma_{x,j}^{-1}\cdot\sigma_{x,j+1}\right\Vert }}\\
 & =\sum_{{\scriptscriptstyle \sigma\in\matchr^{*}\left(\wl\right)}}\left(-1\right)^{\left|\kappa\left(\sigma\right)\right|}n^{\#\left\{ o\mathrm{-discs~in}~\Sigma_{\sigma}\right\} -\sum_{x}\left(L_{x}+\sum_{j=0}^{\kappa_{x}\left(\sigma\right)-1}\left\Vert \sigma_{x,j}^{-1}\cdot\sigma_{x,j+1}\right\Vert \right)}.
\end{align*}
It is thus enough to explain why 
\[
\chi\left(\sigma\right)=\#\left\{ o\mathrm{-discs~in}~\Sigma_{\sigma}\right\} -\sum_{x}\left(L_{x}+\sum_{j=0}^{\kappa_{x}\left(\sigma\right)-1}\left\Vert \sigma_{x,j}^{-1}\cdot\sigma_{x,j+1}\right\Vert \right).
\]
\begin{itemize}
\item The number of vertices in $\Sigma_{\sigma}$ is $\sum_{x\in B}\left(2\kappa_{x}\left(\sigma\right)+4\right)L_{x}$
(consisting of $\kappa_{x}\left(\sigma\right)+1$ $z$-points and
a single $o$-point for each of the $2L_{x}$ $x^{\pm1}$-letters
in $\wl$). 
\item The number of $1$-cells in $\Sigma_{\sigma}$ is $\Sigma_{x\in B}\left(3\kappa_{x}\left(\sigma\right)+5\right)L_{x}$
(there are $\sum_{x\in B}\left(2\kappa_{x}\left(\sigma\right)+4\right)L_{x}$
$1$-cells along the boundary components, and additional $\sum_{x\in B}\left(\kappa_{x}\left(\sigma\right)+1\right)L_{x}$
bijection-edges).
\item Finally, it is easy to see from Definition \ref{def:surface from matchings}
that the number of cycles in the permutation $\sigma_{x,j}^{-1}\cdot\sigma_{x,j+1}$
is identical to the number of $\left(x,j\right)$-discs in $\Sigma_{\sigma}$,
so that 
\[
\left\Vert \sigma_{x,j}^{-1}\cdot\sigma_{x,j+1}\right\Vert =L_{x}-\#\left\{ \left(x,j\right)\mathrm{-discs~in}~\Sigma_{\sigma}\right\} .
\]
\end{itemize}
Therefore,
\begin{eqnarray}
\chi\left(\sigma\right) & = & \left[\sum_{x\in B}\left(2\kappa_{x}\left(\sigma\right)+4\right)L_{x}\right]-\left[\sum_{x\in B}\left(3\kappa_{x}\left(\sigma\right)+5\right)L_{x}\right]+\nonumber \\
 &  & +\left[\sum_{x\in B}\sum_{j=0}^{\kappa_{x}\left(\sigma\right)-1}\#\left\{ \left(x,j\right)\mathrm{-discs~in}~\Sigma_{\sigma}\right\} \right]+\#\left\{ o\mathrm{-discs~in}~\Sigma_{\sigma}\right\} \nonumber \\
 & = & \#\left\{ o\mathrm{-discs~in}~\Sigma_{\sigma}\right\} +\sum_{x\in B}\left[\left(-\kappa_{x}\left(\sigma\right)-1\right)L_{x}+\sum_{j=0}^{\kappa_{x}\left(\sigma\right)-1}\left(L_{x}-\left\Vert \sigma_{x,j}^{-1}\cdot\sigma_{x,j+1}\right\Vert \right)\right]\nonumber \\
 & = & \#\left\{ o\mathrm{-discs~in}~\Sigma_{\tilde{\sigma}}\right\} -\sum_{x\in B}\left[L_{x}+\sum_{j=0}^{\kappa_{x}\left(\sigma\right)-1}\left\Vert \sigma_{x,j}^{-1}\cdot\sigma_{x,j+1}\right\Vert \right].\label{eq:formula for chi(sigma)}
\end{eqnarray}

\end{proof}
It is implicit in Theorem \ref{thm:combinatorial Laurent expansion of trwl}
and its proof that there are only finitely many sets of sequences
of bijections $\sigma$ with contribution of a given order. Namely,
for every integer $\chi_{0}$ there are finitely many $\sigma$ in
the summation (\ref{eq:Laurent-formula}) with $\chi\left(\sigma\right)=\chi_{0}$.
This is true because the same property holds for the asymptotic expansion
of the Weingarten function in (\ref{eq:asympt expansion of Wg}).
However, for completeness, we give a direct proof for this fact:
\begin{claim}
\label{claim:new formula finite sum for every exponent}For every
$\chi_{0}\in\mathbb{Z}$ there are finitely many sets $\sigma$ in
the sum (\ref{eq:Laurent-formula}) with $\chi\left(\sigma\right)=\chi_{0}$.
\end{claim}

\begin{proof}
The number of $o$-discs in $\Sigma_{\sigma}$ is bounded by the number
of $o$-points in $S^{1}\left(w_{1}\right)\cup\ldots\cup S^{1}\left(w_{\ell}\right)$.
All sets $\sigma$ in the sum (\ref{eq:Laurent-formula}) satisfy
$\sigma_{x,j}\ne\sigma_{x,j+1}$ for all $x\in B$ and $0\le j\le\kappa_{x}-1$,
and so $\left\Vert \sigma_{x,j}^{-1}\cdot\sigma_{x,j+1}\right\Vert \ge1$.
Thus, from (\ref{eq:formula for chi(sigma)}) we obtain that if $\chi\left(\sigma\right)=\chi_{0}$
then
\[
\chi_{0}=\chi\left(\sigma\right)\le\#\left\{ o\mathrm{-points~in}~S^{1}\left(w_{1}\right)\cup\ldots\cup S^{1}\left(w_{\ell}\right)\right\} -\sum_{x\in B}\left[L_{x}+\kappa_{x}\left(\sigma\right)\right].
\]
Hence
\[
\sum_{x\in B}\kappa_{x}\left(\sigma\right)\le\#\left\{ o\mathrm{-points~in}~S^{1}\left(w_{1}\right)\cup\ldots\cup S^{1}\left(w_{\ell}\right)\right\} -\chi_{0}-\sum_{x\in B}L_{x}.
\]
Since the right hand side is independent of $\sigma$, the proof is
completed.
\end{proof}
The duality between the two types of formulas in Theorems \ref{thm:trwl as finite sum}
and \ref{thm:combinatorial Laurent expansion of trwl} will be manifested
also in the next section. Our main object of study will be the complex
$\T\left(\Sigma,f\right)$ of transverse maps which, similarly to
the sets in the infinite formula (\ref{eq:Laurent-formula}), consists
of sequences of arcs and curves of arbitrary lengths, but with the
single constraint that two consecutive objects in every sequence must
be different from each other (strict transverse maps -- see Definition
\ref{def:restrictions}). However, for one of the main results about
$\T\left(\Sigma,f\right)$, namely, its being contractible, we return
to the model of sequences of length two without the constraint of
two consecutive objects being different -- see Definition \ref{def:poset and complex of loose transverse maps}.

\subsection{Maps from the surfaces to the bouquet\label{subsec:Maps-from-the}}

In Definition \ref{def:surface from matchings} and in Theorems \ref{thm:trwl as finite sum}
and \ref{thm:combinatorial Laurent expansion of trwl} we introduced
surfaces associated with $\wl$ and tuples of matchings. The following
definition introduces a natural map from these surfaces to the bouquet
so that each surface and its associated map turns into an admissible
pair for $\wl$.
\selectlanguage{english}%
\begin{defn}
\label{def:f_sigma}Let $\wl\in\F_{r}$, let $\sigma\in\match^{\kappa}\left(\wl\right)$
for some $\kappa\in\left(\mathbb{Z}_{\ge0}\right)^{B}$ and let $\Sigma_{\sigma}$
be the surface constructed in Definition \ref{def:surface from matchings}.
Define \marginpar{$f_{\sigma}$}$f_{\sigma}\colon\Sigma_{\sigma}\to\wedger$
as follows: 
\begin{itemize}
\item For every $x\in B$, mark $\kappa_{x}+1$ distinct points on the circle
of the bouquet $\wedger$ corresponding to $x$, away from the wedge
point $o$, and label them $\left(x,0\right),\ldots,\left(x,\kappa_{x}\right)$
in the order of the orientation of the circle. See Figure \ref{fig:wedge-with-transverse-points}.
\item The preimage through $f_{\sigma}^{-1}$ of $\left(x,j\right)\in\wedger$
is exactly the bijection-edges corresponding to $\sigma_{x,j}$, which
contain the $\left(x,j\right)$-points of $\Sigma_{\sigma}$ as their
endpoints.
\item The $o$-points in $\Sigma_{\sigma}$ are mapped to $o\in\wedger$.
\item On $S^{1}\left(w_{i}\right)$, on each of the $\left|w\right|$ intervals,
if the interval $I$ corresponds to the letter $x^{\varepsilon}$,
$\varepsilon\in\left\{ \pm1\right\} $, $f_{\sigma}\Big|_{I}$ traces
the $x$-circle in $\wedger$ monotonically, with orientation prescribed\footnote{We mention this specifically because when $\kappa_{x}=0$ this does
not follow from the previous bullet points.} by $\varepsilon$.
\item Finally, for every open disc $D$ in the CW-complex $\Sigma_{\sigma}$,
the image of $f_{\sigma}$ along $\partial D$ is nullhomotopic, so
there is a unique way to extend it to $D$, up to homotopy, and we
extend it so that the image of $f$ on the interior of $D$ avoids
the marked points $\left\{ \left(x,j\right)\right\} _{x\in B,j\in\left[\kappa_{x}\right]}\subset\wedger$.
\end{itemize}
\end{defn}

\begin{figure}
\centering{}\includegraphics[viewport=230bp 300bp 300bp 460bp,scale=0.5]{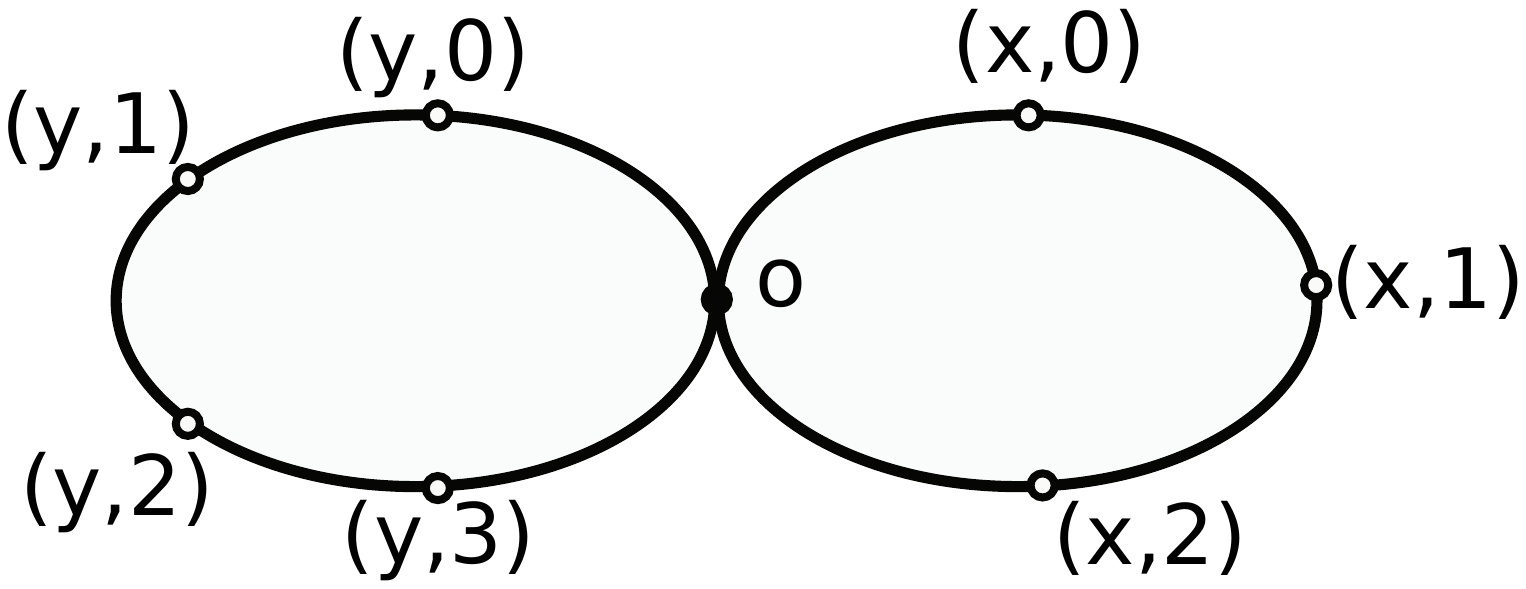}\caption{The wedge $\protect\wedger$ with transversion points. Here $r=2$,
$B=\left\{ x,y\right\} $, $\kappa_{x}=2$ and $\kappa_{y}=3$.\label{fig:wedge-with-transverse-points} }
\end{figure}

It is evident that $\left(\Sigma_{\sigma},f_{\sigma}\right)$ is admissible
for $\wl$ (with the appropriate $o$-point in $S^{1}\left(w_{i}\right)$
labeled also as $v_{i}$, for every $i=1,\ldots,\ell$). In particular,
\begin{cor}
\label{cor:surfaces not empty}If $w_{1}\cdots w_{\ell}\in\left[\F_{r},\F_{r}\right]$
then $\surfaces\left(\wl\right)\ne\emptyset$.
\end{cor}

Another important observation is that every incompressible pair $\left[\left(\Sigma,f\right)\right]\in\surfaces\left(\wl\right)$
has a representative in the form of $\left(\Sigma_{\sigma},f_{\sigma}\right)$
with only one matching per letter:
\begin{prop}
\label{prop:incompressible as transverse maps}Denote by $\match^{\kappa\equiv0}\left(\wl\right)$
the set of matchings corresponding to $\kappa_{x}=0$ for every $x\in B$.
Every incompressible pair $\left(\Sigma,f\right)$ which is admissible
for $\wl$ is equivalent to $\left(\Sigma_{\sigma},f_{\sigma}\right)$
for some $\sigma\in\match^{\kappa\equiv0}\left(\wl\right)$.
\end{prop}

\begin{proof}
The argument here imitates the one in \cite[Theorem 1.4]{CULLER}.
Assume $\left(\Sigma,f\right)$ is incompressible. Mark a point $\left(x,0\right)$
on the middle of the circle corresponding to $x$ in $\wedger$, and
perturb $f$ (relative to the points $v_{1},\ldots,v_{\ell}\in\Sigma$)
so that $f\left(\partial_{i}\Sigma\right)$ is monotone, namely, never
backtracking, for $i=1,\ldots,\ell$, and so that it becomes transverse
to $\left\{ \left(x,0\right)\right\} _{x\in B}\subset\wedger$ (see
Definition \ref{def:transverse map} below). As $f$ is transverse
to $\left(x,0\right)\in\wedger$, the preimage of $\left(x,0\right)$
consists of a collection of disjoint arcs and curves (in this paper
we use the notion ``curve'' as synonym for ``simple closed curve'').
Because $f\left(\partial_{i}\Sigma\right)$ is monotone, there are
exactly $L_{x}$ arcs, which determine an element $\sigma\in\match^{\kappa\equiv0}\left(\wl\right)$.
There are no curves in $f^{-1}\left(\left(x,0\right)\right)$ because
such curves would be null-curves, which is impossible with $f$ being
incompressible. Finally, $f$ being incompressible also guarantees
that the collection of arcs $\coprod_{x\in B}f^{-1}\left(\left(x,0\right)\right)$
cuts $\Sigma$ into discs. Thus $\left(\Sigma,f\right)\sim\left(\Sigma_{\sigma},f_{\sigma}\right)$.
\end{proof}
Since the set $\match^{\kappa\equiv0}\left(\wl\right)$ is finite,
we obtain: 
\begin{cor}
\label{cor:finitely many incompressible}There are finitely many classes
of incompressible $\left(\Sigma,f\right)$ in $\surfaces\left(\wl\right)$.
\end{cor}

In Section \ref{subsec:Classifying-all-incompressible} we address
the issue of how one can distinguish the different incompressible
classes in $\surfaces\left(\wl\right)$.

At this point we can also derive the asymptotic bounds we have for
$\trwl\left(n\right)$:
\begin{proof}[Proof of Corollary \ref{cor:chi and cl}]
Recall that we need to prove that $\trwl\left(n\right)=O\left(n^{\ch\left(\wl\right)}\right)$
where $\ch\left(\wl\right)$ is the maximal Euler characteristic of
a surface in $\surfaces\left(\wl\right)$. Theorem \ref{thm:trwl as finite sum}
says that $\trwl\left(n\right)$ is equal to a sum over $\sigma\in\match^{\kappa\equiv1}\left(\wl\right)$,
and the contribution of each $\sigma$ is $c\cdot n^{\chi\left(\sigma\right)}+O\left(n^{\chi\left(\sigma\right)-2}\right)$
where $c\in\mathbb{Z}\setminus\left\{ 0\right\} $. As $\left(\Sigma_{\sigma},f_{\sigma}\right)\in\surfaces\left(\wl\right)$,
then by definition $\chi\left(\sigma\right)\le\ch\left(\wl\right)$.
\end{proof}
\begin{rem}
\label{rem:lower bound}In fact, we get even more: every $\left[\left(\Sigma,f\right)\right]\in\surfaces\left(\wl\right)$
attaining\linebreak{}
$\ch\left(\wl\right)$ is incompressible, and therefore, by Proposition
\ref{prop:incompressible as transverse maps}, equivalent to $\left(\Sigma_{\sigma},f_{\sigma}\right)$
for some\linebreak{}
$\sigma\in\match^{\kappa\equiv0}\left(\wl\right)$. This $\sigma$
takes part in the expression for $\trwl\left(n\right)$ in Theorem
\ref{thm:combinatorial Laurent expansion of trwl} and thus one can
expect that $\trwl\left(n\right)=\Theta\left(n^{\ch\left(\wl\right)}\right)$.
As some of the examples from Table \ref{tab:examples} indicate, the
coefficient of $n^{\ch\left(\wl\right)}$ may vanish, but this only
happens if the different non-zero contributions cancel out. (One can
get to the same conclusion from the finite formula for $\trwl\left(n\right)$
in Theorem \ref{thm:trwl as finite sum}, by duplicating every matching
in $\sigma\in\match^{\kappa\equiv0}$ to obtain $\sigma'\in\match^{\kappa\equiv1}\left(\wl\right)$,
which then satisfies $\left(\Sigma_{\sigma},f_{\sigma}\right)\sim\left(\Sigma_{\sigma'},f_{\sigma'}\right)$.)
\end{rem}

\subsubsection*{Reduction of the main theorem }

Recall that Theorem \ref{thm:combinatorial Laurent expansion of trwl}
expresses $\trwl\left(n\right)$ as a sum over the (generally infinite)
set\linebreak{}
$\matchr^{*}\left(\wl\right)$. To prove our main result, Theorem
\ref{thm:main}, we group together all\linebreak{}
$\sigma\in\matchr^{*}\left(\wl\right)$ for which $\left(\Sigma_{\sigma},f_{\sigma}\right)$
belong to the same equivalence class, and show the total contribution
of these values of $\sigma$ to (\ref{eq:Laurent-formula}) is exactly
the one specified in Theorem \ref{thm:main}. Accordingly, for $\left[\left(\Sigma,f\right)\right]\in\surfaces$
we let $\matchr^{*}\left(\wl;\Sigma,f\right)$\marginpar{${\scriptstyle \protect\matchr^{*}\left(\ldots;\Sigma,f\right)}$}
be the $\sigma$'s yielding elements in the class of $\left(\Sigma,f\right)$
. So, recalling the notation ``$\sim$'' from Definition \ref{def:surfaces(w1,...,wl)},
\[
\matchr^{*}\left(\wl;\Sigma,f\right)\stackrel{\mathrm{def}}{=}\left\{ \sigma\in\matchr^{*}\left(\wl\right)\,\middle|\,\left(\Sigma_{\sigma},f_{\sigma}\right)\sim\left(\Sigma,f\right)\right\} .
\]
From Claim \ref{claim:new formula finite sum for every exponent}
it follows that $\matchr^{*}\left(\wl;\Sigma,f\right)$ is finite
for every\linebreak{}
$\left(\Sigma,f\right)\in\surfaces\left(\wl\right)$. Using Theorem
\ref{thm:combinatorial Laurent expansion of trwl}, Theorems \ref{thm:stabilizers have L2-EC}
and \ref{thm:main} now reduce to:
\begin{thm}
\label{thm:EC of a single (S,f)}Let $\left[\left(\Sigma,f\right)\right]\in\surfaces\left(\wl\right)$.
Then $\chi^{\left(2\right)}\left(\MCG\left(f\right)\right)$ is well
defined and given by
\begin{equation}
\chi^{\left(2\right)}\left(\MCG\left(f\right)\right)=\sum_{\sigma\in\matchr^{*}\left(\wl;\Sigma,f\right)}\left(-1\right)^{\left|\kappa\left(\sigma\right)\right|}.\label{eq:EC of a single (Sigma,f)}
\end{equation}
\end{thm}

In particular, if $\left[\left(\Sigma,f\right)\right]\in\surfaces\left(\wl\right)$
cannot be realized by any\linebreak{}
$\sigma\in\matchr^{*}\left(\wl\right)$, then $\chi^{(2)}\left(\MCG\left(f\right)\right)=0$.
Since there are only finitely many $\sigma$ with $\Sigma_{\sigma}$
of a given Euler characteristic (Claim \ref{claim:new formula finite sum for every exponent}),
it follows from Theorem \ref{thm:EC of a single (S,f)} that, indeed,
for any $\chi_{0}\in\mathbb{Z}$, there are only finitely many classes
$\left[\left(\Sigma,f\right)\right]\in\surfaces\left(\wl\right)$
with $\chi\left(\Sigma\right)=\chi_{0}$ and for which $\chi^{\left(2\right)}\left(\MCG\left(f\right)\right)\ne0$.

In the next two sections we describe the constructions and results
that lead to the proof of Theorem \ref{thm:EC of a single (S,f)}
and of Theorem \ref{thm:K(g,1) for incompressible} which strengthen
the result in the case $\left(\Sigma,f\right)$ is incompressible.
We hint that the special property of the incompressible case is that
the set $\matchr^{*}\left(\wl;\Sigma,f\right)$ gives rise to a natural
complex with one cell for every element $\sigma\in\matchr^{*}\left(\wl;\Sigma,f\right)$,
so that the Euler characteristic of this complex is exactly the right
hand side of (\ref{eq:EC of a single (Sigma,f)}). See Section \ref{subsec:Incompressible-maps}
for details.

\section{A complex of transverse maps\label{sec:A-Complex-of}}

The key ingredient in the proof of our main results is a complex of
transverse maps which we associate with a given pair $\left[\left(\Sigma,f\right)\right]\in\surfaces\left(\wl\right)$.
In the current section we define it, study important properties and
prove it is contractible. In the next section we study the action
of $\mcg\left(f\right)$ on this complex and prove our main results. 

\subsection{Transverse maps}

Recall that in this paper the term ``curve'' is short for a simple
closed curve.
\begin{defn}
\label{def:transverse map}Let $\Sigma$ be orientable. A map $f\colon\Sigma\to\wedger$
is said to be \emph{transverse} \marginpar{transverse}to a point
$p\in\wedger\setminus\left\{ o\right\} $ if the preimage of $p$
is a disjoint union of arcs and curves, and if in a small tubular
neighborhood $U$ of every curve or arc $\gamma$ in the preimage,
the two connected components of $U\setminus\gamma$ are mapped to
two different ``sides'' of $p$. 
\end{defn}

For example, the map $f_{\sigma}$ from Definition \ref{def:f_sigma}
is transverse to the points $\left\{ \left(x,j\right)\right\} _{x\in B,j\in\left[\kappa_{x}\right]}$
in $\wedger$. In this case, the preimage of each of these points
contains no curves but rather only arcs. We shall consider here different
realizations of the homotopy class $\left[f\right]$ of the same map
$f\colon\Sigma\to\wedger$, which are transverse to different collections
of points in $\wedger$. 

More formally, let $\Sigma$ be a surface and $f$ a map $f\colon\Sigma\to\wedger$
so that\linebreak{}
$\left[\left(\Sigma,f\right)\right]\in\surfaces\left(\wl\right)$.
By definition, $\Sigma$ has $\ell$ marked points: one point, labeled
$v_{i}$, in every boundary component $\partial_{i}\Sigma$, for $i=1,\ldots,\ell$,
and with $f\left(v_{i}\right)=o$. Note that $\wl$ are prescribed
from $\Sigma$ and $f$ by $w_{i}=f_{*}\left(\overrightarrow{\partial_{i}\Sigma},v_{i}\right)\in\pi_{1}\left(\wedger,o\right)$.
For every $i=1,\ldots,\ell$, we mark additional $\left|w_{i}\right|-1$
points on $\partial_{i}\Sigma$ inside $f^{-1}\left(o\right)$, so
that $f$ maps the intervals of $\partial_{i}\Sigma$ cut by these
points to the letters of $w_{i}$. We let \marginpar{$V_{o}$}$V_{o}\subset\Sigma$
denote the set of all marked points in $\Sigma$: a total of $\sum_{i=1}^{\ell}\left|w_{i}\right|$
marked points all at the boundary.
\begin{defn}
\label{def:transverse map for Sigma,f}Let $\kappa=\left\{ \kappa_{x}\right\} _{x\in B}\in\left(\mathbb{Z}_{\ge0}\right)^{B}$
be a set of non-negative integers. On the circle corresponding to
$x$ in $\wedger$ mark $\kappa_{x}+1$ disjoint points, $\left(x,0\right),\ldots,\left(x,\kappa_{x}\right)$,
arranged as in Definition \ref{def:f_sigma} and Figure \ref{fig:wedge-with-transverse-points}.
Let $\left[\left(\Sigma,f\right)\right]\in\surfaces\left(\wl\right)$
and $V_{o}\subset f^{-1}\left(o\right)\subseteq\Sigma$ be as above.
A map $g\colon\Sigma\to\wedger$ is \emph{a transverse map realizing
$\left(\Sigma,f\right)$ with parameters $\kappa$}, if it is homotopic
to $f$ relative to $V_{o}$ and transverse to the points $\left\{ \left(x,j\right)\right\} _{x\in B,j\in\left[\kappa_{x}\right]}\subset\wedger$.
Note that, in particular, $g\left(V_{o}\right)=\left\{ o\right\} $. 

An arc (curve, respectively) in the preimage of $\left(x,j\right)$
is called an $\left(x,j\right)$-arc ($\left(x,j\right)$-curve\marginpar{$\left(x,j\right)$-arc/curve}).
Let $U_{o}$ be the connected component of $o$ in $\wedger\setminus\left\{ \left(x,j\right)\right\} _{x\in B,j\in\left[\kappa_{x}\right]}$.
A connected component of $g^{-1}\left(U_{o}\right)$ is called an
\emph{$o$-zone}\marginpar{$o$-zone}. For $0\le j\le\kappa_{x}-1$,
let $I_{x,j}\subset\wedger$ be the interval on the $x$-circle cut
out by $\left(x,j\right)$ and $\left(x,j+1\right)$. We call a connected
component of $g^{-1}\left(I_{x,j}\right)$ an \emph{$\left(x,j\right)$-zone}\marginpar{$\left(x,j\right)$-zone},
or, if $x$ and $j$ are not relevant, also a $z$-zone\marginpar{$z$-zone}.
If all zones defined by $g$ are topological discs, we say that $g$
\emph{fills}\marginpar{fills} $\Sigma$, or that $g$ is \emph{filling}.

The isotopy\footnote{We call $\left[g\right]$ the isotopy class of $g$, rather than the
homotopy class, because if one thinks of $g$ as a collection of disjoint
colored arcs and curves embedded in $\Sigma$, then $\left[g\right]$
is indeed the isotopy class of this collection relative to $V_{o}$.} class of the transverse map $g$, denoted $\left[g\right]$\marginpar{$\left[g\right]$},
contains all transverse maps with the same parameters $\kappa$ which
are homotopic to $g$ relative to $V_{o}$ via a homotopy of transverse
maps with the same parameters. We stress that the marked points in
$\wedger$ are allowed to move inside $\wedger\setminus\left\{ o\right\} $
along the homotopy as long as they remain disjoint.
\end{defn}

Note that every $\left(x,j\right)$-arc/curve has a direction from
one side of the arc/curve to the other, induced by the orientation
of the circle in $\wedger$. Since we do not care about the location
of the $\left(x,j\right)$-points in $\wedger$, a transverse map
$g$ for $\left(\Sigma,f\right)$ can be identified with the collection
of disjoint ``directed'' and colored arcs and curves. The isotopy
class of $g$ can then be though of as the isotopy class of this collection
of arcs and curves relative to $V_{o}$. We illustrate such a collection
in Figure \ref{fig:A-transverse-map-a2-a^-2}.

Also note, by the definition above, that the boundary of an $\left(x,j\right)$-zone
of $g$ consists of pieces of $\partial\Sigma$, of $\left(x,j\right)$-arcs/curves
directed inward and of $\left(x,j+1\right)$-arcs/curves directed
outward. In contrast, the boundary of an $o$-zone of $g$ consists
of pieces of $\partial\Sigma$, of $\left(x,0\right)$-arcs/curves
directed outward and of $\left(x,\kappa_{x}\right)$-arcs/curves directed
inward, for various $x\in B$. Finally, every point in $V_{o}$ belongs
to some $o$-zone of $g$. 

\begin{figure}
\centering{}\includegraphics{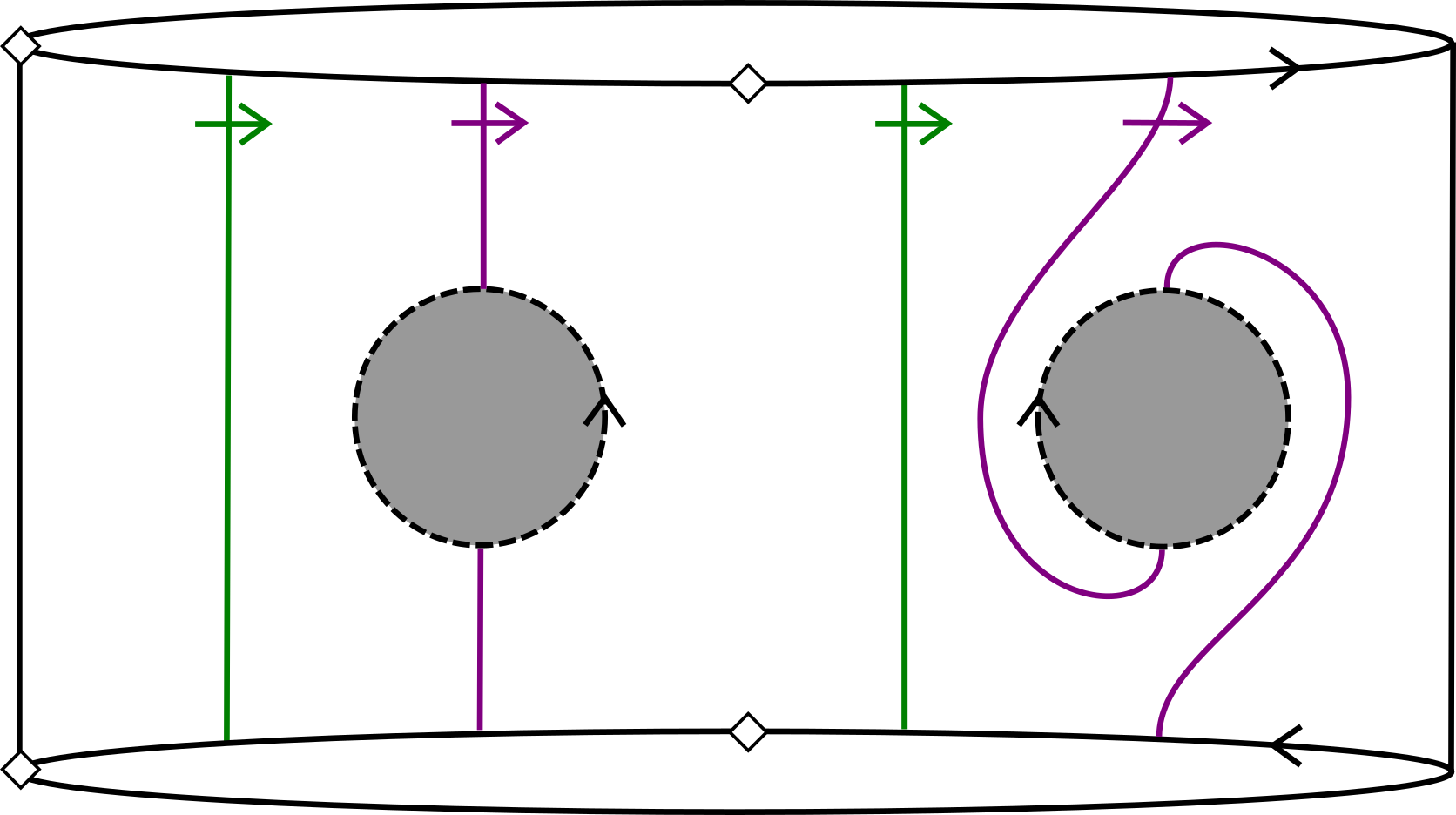}\caption{A collection of arcs corresponding to a transverse map realizing $\left(\Sigma,f\right)$,
where $\Sigma$ is a genus 1 surface with 2 boundary components, drawn
as an annulus with two discs cut out, with boundaries of those discs
identified. Here $r=1$ and there is one generator $x$. Green arcs
are $(x,0)$-arcs and purple are $(x,1)$-arcs. The marked points
$V_{o}$ are diamonds. There are no curves in this system and it is
filling. The words at the boundary are $x^{2},x^{-2}$.\label{fig:A-transverse-map-a2-a^-2} }
\end{figure}

Generally, we want to forbid certain trivial or redundant features
of transverse maps, as we elaborate in the following definition:
\begin{defn}
\label{def:restrictions}A transverse map for $\left(\Sigma,f\right)$
is called \emph{loose}\marginpar{loose transverse map} if it satisfies
\begin{itemize}
\item \textbf{Restriction$\:$1}. There are no $o$-zones nor $z$-zones
that contain no marked point from $V_{o}$ and whose boundary arcs
and curves have the same color $(x,j)$ and are all oriented pointing
inwards or all oriented outwards. Note this rules out the possibility
of a zone that is a disc bounded by a closed curve.
\item \textbf{Restriction$\:$2}. No segment of the boundary of $\Sigma$
that contains no marked point can be bounded by the end points of
two arcs that are equally-labeled and both directed inwards or both
outwards. Note that this is the boundary analog of \textbf{Restriction
1}.
\end{itemize}
A transverse map for $\left(\Sigma,f\right)$ is called \emph{strict}\marginpar{strict transverse map}\emph{
}if it satisfies, in addition, 
\begin{itemize}
\item \textbf{Restriction$\:$3}. For every $x\in B$ and $0\le j\le\kappa_{x}-1$,
the collection of $\left(x,j\right)$-arcs and curves is not isotopic
to the collection of $\left(x,j+1\right)$-arcs and curves. In other
words, there must be at least one $(x,j)$-zone which is neither a
rectangle nor an annulus\footnote{Here, a rectangle is a disc bounded by two arcs and two pieces of
$\partial\Sigma$, and an annulus is bounded by two curves. Restriction
3 should resonate the constraint on the set of matchings $\matchr^{\kappa}\left(\wl\right)$
from Section \ref{subsec:Laurent expansion of trwl}. In particular,
if $g\left(\Sigma\right)$ does not contain the circle in $\wedger$
associated with $x$, then necessarily $\kappa_{x}=0$.}. 
\end{itemize}
\end{defn}

\begin{rem}
\begin{itemize}
\item Note that if $g$ fills $\Sigma$ then there are no curves involved
in $g$, but only arcs. This is the case, for example, when $g=f_{\sigma}$
as in Definition \ref{def:f_sigma}.
\item Any transverse map for $\left(\Sigma,f\right)$ satisfying \textbf{Restriction
2} admits exactly $\kappa_{x}+1$ arcs touching every interval in
$\partial\Sigma$ corresponding to the letter $x$, one arc for every
$j\in\left[\kappa_{x}\right]$. Consequently, it admits exactly $L_{x}$
arcs labeled $\left(x,j\right)$ for every $x\in B$ and $j\in\left[\kappa_{x}\right]$.
In addition, if $O$ is an $o$-zone of such a map then every connected
component of $\overline{O}\cap\partial\Sigma$ contains exactly one
marked point from $V_{o}$.
\item If \textbf{Restriction 2} holds, then \textbf{Restriction 1} can only
fail at zones bounded by curves. 
\end{itemize}
\end{rem}

The following local surgery of a transverse map will be very useful
in the sequel:
\begin{defn}
\label{def:H-move}Let $g$ be a transverse map realizing $\left(\Sigma,f\right)$
with parameters $\kappa$. Let $\alpha_{1}$ and $\alpha_{2}$ be
each an $\left(x,j\right)$-arc or an $\left(x,j\right)$-curve in
the collection corresponding to $g$, where $\alpha_{1}$ and $\alpha_{2}$
have the same color and are not necessarily distinct. Assume further
there is an embedded arc-segment $\gamma$ inside the interior of
$\Sigma$, with one endpoint in $\alpha_{1}$ and the other in $\alpha_{2}$
such that the interior of $\gamma$ is disjoint from the arc-curve
collection of $g$, and such that both $\alpha_{1}$ and $\alpha_{2}$
are directed towards $\gamma$ or both directed away from $\gamma$.
We say that the transverse map $g'$ realizing $\left(\Sigma,f\right)$
with the same parameters is obtained from $g$ by an \emph{H-Move}\marginpar{H-Move}
along $\gamma$ if: 
\begin{itemize}
\item One takes a small collar neighborhood of $\gamma$ to obtain a rectangle
whose short sides are contained in $\alpha_{1}$ and in $\alpha_{2}$.
\item One deletes the short sides of the rectangle from the arc-curve collection
of $g$ and replaces them with the long sides to obtain a new collection
which defines $g'$. See Figure \ref{fig:h-move}.
\end{itemize}
\end{defn}

It is clear that $g'$ is homotopic to $g$ as a map but not isotopic
as a transverse map. 

\begin{figure}
\begin{centering}
\includegraphics[scale=0.8]{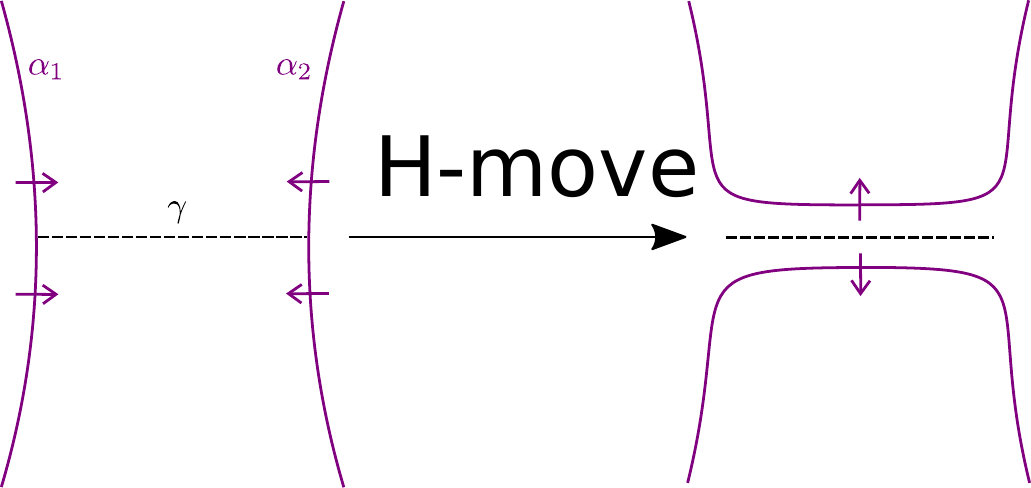}
\par\end{centering}
\caption{\label{fig:h-move}The dashed line is the arc-segment $\gamma$ along
which the H-move is performed. The two purple segments on the left
are parts of $\alpha_{1}$ and $\alpha_{2}$.}
\end{figure}

\subsection{The poset of strict transverse maps}

The complex of transverse maps will be defined as the geometric realization
of the \emph{poset }of transverse maps: 
\begin{defn}
\label{def:poset of transverse maps}Let $\Sigma$ and $f\colon\Sigma\to\wedger$
satisfy $\left[\left(\Sigma,f\right)\right]\in\surfaces\left(\wl\right)$.
Let $V_{o}\subset f^{-1}\left(o\right)\cap\partial\Sigma$ be defined
as above, so $V_{o}\cap\partial_{i}\Sigma$ cuts $\partial_{i}\Sigma$
into $\left|w_{i}\right|$ intervals, each of which is mapped to some
$x^{\pm1}$ with $x\in B$ by $f_{*}$. The \emph{poset of transverse
maps realizing $\left(\Sigma,f\right)$, }denoted $\T=\T\left(\Sigma,f\right)$\marginpar{${\scriptstyle \protect\T=\protect\T\left(\Sigma,f\right)}$}
or $\left(\T,\preceq\right)$, consists of the set of isotopy classes
relative to $V_{o}$ of \textbf{\emph{strict}} transverse maps realizing
$\left(\Sigma,f\right)$. The order is defined by ``forgetting points
of transversion''. Namely, whenever $g_{1}$ is a transverse map
realizing $\left(\Sigma,f\right)$ with parameters $\left\{ \kappa_{x}\right\} _{x\in B}$
and $g_{2}$ is identical to $g_{1}$ except we forget a proper (possibly
empty) subset of the transversion points $\left\{ \left(x,j\right)\right\} _{j\in\left[\kappa_{x}\right]}\subset\wedger$
for every $x\in B,$ then the isotopy classes $\left[g_{1}\right]$
and $\left[g_{2}\right]$ satisfy $\left[g_{2}\right]\preceq\left[g_{1}\right]$
in the poset $\T$. 
\end{defn}

Of course, the transversion points that remain in $g_{2}$ may need
relabeling. Note that we use here \emph{strict} transverse maps: maps
satisfying the three restrictions from Definition \ref{def:restrictions}.
The role of \emph{loose} transverse maps will be clarified in the
sequel of this section. Another important observation is that if the
transverse map $g$ is strict, then so are the maps obtained from
$g$ by forgetting transversion points:
\begin{lem}
\label{lem:X is downward closed}If $g_{1}$ is a strict transverse
map realizing $\left(\Sigma,f\right)$ and $g_{2}$ is obtained from
$g_{1}$ by forgetting transversion points, then $g_{2}$ is also
strict. In other words, if $\left[g_{1}\right]\in\T$ then $\left[g_{2}\right]\in\T$.
\end{lem}

\begin{proof}
It is enough to prove the statement of the lemma in the special case
where $g_{2}$ is obtained by forgetting a single point, say the point
$\left(x,j\right)$, for some $x\in B$ with $\kappa_{x}\ge1$ and
$j\in\left[\kappa_{x}\right]$. It is obvious that $g_{2}$ satisfies
\textbf{Restriction 3}. Let $I\subset\partial_{i}\Sigma\setminus V_{o}$
be an interval cut out by two adjacent marked points. As $g_{1}$
satisfies \textbf{Restriction 2}, all arcs touching $I$ are directed
in the same orientation w.r.t.~$I$ and this property remains true
for $g_{2}$, hence $g_{2}$ satisfies \textbf{Restriction 2}. 

As for \textbf{Restriction 1}, assume first that $j=\kappa_{x}$,
so every $o$-zone and $\left(x,\kappa_{x}-1\right)$-zone of $g_{1}$
which are neighbors belong to the same $o$-zone of $g_{2}$. Every
$z$-zone of $g_{2}$ is also a $z$-zone of $g_{1}$ with the same
boundary, so \textbf{Restriction 1 }is not violated there. Let $O\subset\Sigma$
be an $o$-zone of $g_{2}$ that violates \textbf{Restriction 1}.
Then $O$ is not an $o$-zone of $g_{1}$, and has to be a union of
$o$-zones and $\left(x,\kappa_{x}-1\right)$-zones of $g_{1}$, at
least one of each type. In addition, $O$ contains no marked points
and has only incoming $\left(x,\kappa_{x}-1\right)$-arcs/curves at
its boundary: this is because $g_{1}$ satisfies \textbf{Restriction
1}, every $\left(x,\kappa_{x}-1\right)$-zone has some incoming $\left(x,\kappa_{x}-1\right)$-arc/curve
at its boundary, and so $O$ also has some incoming $\left(x,\kappa_{x}-1\right)$-arc/curve
at its boundary. But then, every $o$-zone of $g_{1}$ contained in
$O$ has no marked points and only incoming $\left(x,\kappa_{x}\right)$-arcs/curves
at its boundary, a contradiction. 

The proof is analogous if $j=0$ and is similar but even simpler if
$1\le j\le\kappa_{x}-1$.
\end{proof}
Another important observation is that $\T=\T\left(\Sigma,f\right)$
is not empty.
\begin{lem}
\label{lem:X not empty}Let $\Sigma$, $f$ be as in Definition \ref{def:poset of transverse maps},
then the poset $\T=\T\left(\Sigma,f\right)$ is not empty.
\end{lem}

\begin{proof}
For every $x\in B$, mark a single point $\left(x,0\right)$ in $\wedger$
on the circle corresponding to $x$. Perturb $f$ to obtain $g$ that
is transverse to the points $\left\{ \left(x,0\right)\right\} _{x\in B}$
(without changing the image at $V_{o}$). The resulting map is a transverse
map realizing $\left(\Sigma,f\right)$ with parameters $\kappa_{x}=0$
for all $x$. \textbf{Restriction 3} is automatically satisfied when
$\kappa_{x}=0$ for all $x$. 

If $g$ violates \textbf{Restriction 2}, then there is a segment $I$
of the boundary cut out by two endpoints of $\left(x,0\right)$-arcs
for some $x\in B$, both directed, say, inwards, and without any marked
point. Let $\gamma$ be an arc parallel to $I$ slightly away from
$\partial\Sigma$ with endpoint at the two arcs cutting $I$. Perform
an $H$-move along $\gamma$, and delete the resulting $\left(x,0\right)$-arc
parallel to $I$ and $\gamma$. In that manner one can get rid of
all violations of \textbf{Restriction 2}. 

So assume now that $g$ does not violate \textbf{Restrictions 2 and
3}. Any violation of \textbf{Restriction 1} is\textbf{ }at zones bounded
by curves. But any such zone can be simply deleted by removing all
its bounding curves. To see that this procedure does not change the
homotopy type of the function, note that it can be achieved by a series
of H-moves: first perform H-moves along arcs connecting a curve and
itself to decrease the genus of the zone to 0. Then, use H-moves between
different bounding curves to eventually reduce the number of bounding
curves to one. The resulting zone is a disk bounded by a curve which
can easily be homotoped away. We can repeat this process until no
violations of \textbf{Restriction 3} remain. The resulting map is
a\textbf{ }strict transverse map realizing $\left(\Sigma,f\right)$.
\end{proof}

\subsection{The complex of transverse maps\label{subsec:The-complex-of-transverse maps}}

The complex of transverse maps is defined as a ``polysimplicial complex'',
meaning that its cells are products of simplices\emph{, or }\marginpar{polysimplex}\emph{polysimplices},
as in $\Delta_{k_{1}}\times\Delta_{k_{2}}\times\ldots\times\Delta_{k_{r}}$,
where $\Delta_{k}$ is the standard simplex of dimension $k$. Note
that the polysimplex $\Delta_{k_{1}}\times\ldots\times\Delta_{k_{r}}$
has dimension $k_{1}+\ldots+k_{r}$.
\begin{defn}
\label{def:the complex of transverse maps}\emph{The complex of transverse
maps realizing $\left(\Sigma,f\right)$, denoted} \marginpar{$\protect\tps$}$\tps=$\linebreak{}
$\left|\T\left(\Sigma,f\right)\right|_{\ps}$, is a polysimplicial
complex with a polysimplex \marginpar{$\mathrm{polysim}\left(\left[g\right]\right)$}$\mathrm{polysim}\left(\left[g\right]\right)\stackrel{\mathrm{def}}{=}\prod_{x}\Delta_{\kappa_{x}}$
for every element $\left[g\right]\in\T$ with parameters $\left\{ \kappa_{x}\right\} _{x\in B}$.
The faces of $\mathrm{polysim\left(\left[g\right]\right)}$ are exactly\linebreak{}
$\left\{ \mathrm{polysim}\left(\left[g'\right]\right)\,\middle|\,\left[g'\right]\preceq\left[g\right]\right\} $.
Then $\tps$ is the union of closed cells or disjoint union of open
cells:
\[
\tps\stackrel{\mathrm{def}}{=}\bigcup_{\left[g\right]\in\T}\overline{\mathrm{polysim}}\left(\left[g\right]\right)=\bigsqcup_{\left[g\right]\in\T}\mathrm{polysim^{o}\left(\left[g\right]\right).}
\]
The topology on $\tps$, as the topology on every (poly-)simplicial
complex in this paper, is defined by taking the Euclidean topology
on every (poly-)simplex $s$, and by letting a general set $A\subseteq\tps$
to be closed if and only if $A\cap s$ is closed in $s$ for every
(poly-)simplex $s$.
\end{defn}

We remark that \textbf{Restriction 3} plays an important role in this
definition: it guarantees that different vertices of the closed polysimplex
$\overline{\mathrm{polysim}}\left(\left[g\right]\right)$ correspond
to different (minimal) elements of $\T$, hence the closed polysimplices
are embedded in $\tps$.

There is an equivalent way to construct the complex of transverse
map (up to homeomorphism), as an ordinary simplicial complex: the\label{order complex |T|}
\emph{order complex} $\left|\T\right|$\marginpar{$\left|\protect\T\right|$}
of $\T$. This is a standard simplicial complex, with simplices corresponding
to \emph{chains} in $\T$: every chain $\left[g_{0}\right]\prec\left[g_{1}\right]\prec\ldots\prec\left[g_{m}\right]$
corresponds to an $m$-simplex, with the obvious faces. 
\begin{claim}
\label{claim:both realizations of T are homeomorphic}$\left|\T\right|$
is the barycentric subdivision of $\tps$. In particular, $\left|\T\right|\cong\tps$.
\end{claim}

To prove the claim we use the following well-known fact. Here, if
$\left(P,\le_{P}\right)$ and $\left(Q,\le_{Q}\right)$ are posets,
then $\left|P\right|$ is the order complex of $P$, and the direct
product $\left(P\times Q,\le_{P\times Q}\right)$ is defined by $\left(p_{1},q_{1}\right)\le_{P\times Q}\left(p_{2},q_{2}\right)$
if and only if $p_{1}\le_{P}p_{2}$ and $q_{1}\le_{Q}q_{2}$.
\begin{fact}[{e.g.~\cite[Theorem 3.2]{walker1988canonical}}]
\label{fact:product of posets} Let $P$ and $Q$ be posets. The
function $\left|P\times Q\right|\to\left|P\right|\times\left|Q\right|$
defined by 
\[
\sum\lambda_{i}\left(p_{i},q_{i}\right)\mapsto\left(\sum\lambda_{i}p_{i},\sum\lambda_{i}q_{i}\right)
\]
is an homeomorphism.
\end{fact}

\begin{proof}[Proof of Claim \ref{claim:both realizations of T are homeomorphic}]
 Let $\left[g\right]\in\T$ with parameters $\kappa\in\left(\mathbb{Z}_{\ge0}\right)^{B}$.
We show that the barycentric subdivision of $\overline{\mathrm{polysim}}\left(\left[g\right]\right)$
consists of the simplices corresponding to chains in $\T$ with top
element $\left[g'\right]$ satisfying $\left[g'\right]\preceq\left[g\right]$.
Indeed, this is certainly true in the single-letter case where $r=\left|B\right|=1$
and every polysimplex is merely a simplex. For the general case, let
$\left[\gamma^{x}\left(g\right)\right]$ denote the isotopy class
of the collection of arcs/curves corresponding to the letter $x$,
for $x\in B$. Let $P_{x}\left(g\right)$ denote the poset of all
isotopy classes of collections of arcs/curves obtained from $\left[\gamma^{x}\left(g\right)\right]$
by forgetting arcs/curves from a proper subset of the colors $\left[\kappa_{x}\right]$.
The single-letter case shows that $\left|P_{x}\left(g\right)\right|\cong\Delta_{\kappa_{x}}$.
Since the subposet of $\T$ given by $\T_{\preceq\left[g\right]}\stackrel{\mathrm{def}}{=}\left\{ \left[g'\right]\in\T\,\middle|\,\left[g'\right]\preceq\left[g\right]\right\} $
is exactly $\prod_{x\in B}P_{x}\left(g\right)$, Fact \ref{fact:product of posets}
yields that the order complex of $\T_{\preceq\left[g\right]}$ is
homeomorphic to $\overline{\mathrm{polysim}}\left(\left[g\right]\right)$.
\end{proof}
An important property of $\tps$ is that it is finite-dimensional.
This is an analog of Claim \ref{claim:new formula finite sum for every exponent}
and here, again, \textbf{Restriction 3} plays an important role:
\begin{lem}
\label{lem:X is finite dimensional}The complex $\tps$ is finite
dimensional with\footnote{Recall Remark \ref{rem:w_i ne 1} that we assume $w_{i}\ne1$ throughout
the proofs. If we do consider the case that some of the words are
trivial, then $\Sigma$ may contain components made of discs, and
the bound in Lemma \ref{lem:X is finite dimensional} needs to be
updated.} $\dim\left(\tps\right)\le\frac{\ell}{2}-\chi\left(\Sigma\right)$.
\end{lem}

\begin{proof}
We need to show that $\sum_{x}\kappa_{x}$ is bounded across $\left[g\right]\in\T$.
It is easy to see that 
\begin{equation}
\chi\left(\Sigma\right)=\sum_{\Sigma'}\left(\chi\left(\Sigma'\right)-\frac{1}{2}\#\left\{ \mathrm{arcs~at~}\partial\Sigma'\right\} \right),\label{eq:chi(Sigma) as sum over zones}
\end{equation}
where the sum is over all $o$-zones and $z$-zones of $g$ in $\Sigma$,
and an arc that bounds $\Sigma'$ from both its sides is counted twice
for $\Sigma'$. The contribution of $\Sigma'$ in (\ref{eq:chi(Sigma) as sum over zones})
is positive only if $\Sigma'$ is a topological disc with at most
one arc at its boundary. By \textbf{Restrictions 1} and \textbf{2}
this means that $\Sigma'$ is bounded by one arc and one interval
from $\partial\Sigma$ containing a marked point, and that its contribution
is $\frac{1}{2}$. Notice that in this case, the marked point must
be the special point $v_{i}\in\partial_{i}\Sigma$ marking ``the
beginning'' of $w_{i}$, and $w_{i}$ must be not cyclically reduced.
Hence the positive contributions on the right hand side of (\ref{eq:chi(Sigma) as sum over zones})
sum up to at most $\frac{\ell}{2}$, and come from $o$-zones only.

On the other hand, the only zones contributing zero to (\ref{eq:chi(Sigma) as sum over zones})
are discs with two arcs at their boundary, namely, rectangles, or
annuli bounded by two curves. Every other $z$-zone contributes at
most -1: this follows from the fact that every boundary component
of such a zone is either a curve or contains an even number of arcs.
Thus, \textbf{Restriction 3} guarantees that for every $x\in B$ and
$0\le j\le\kappa_{x}-1$, the total contribution of the $\left(x,j\right)$-zones
is at most $-1$. We obtain
\[
\chi\left(\Sigma\right)\le\frac{\ell}{2}-\sum_{x}\kappa_{x},
\]
so
\[
\sum_{x}\kappa_{x}\le\frac{\ell}{2}-\chi\left(\Sigma\right).
\]
\end{proof}
\begin{rem}
\label{rem:bound on dimension when words are cylically reduced}When
$\wl$ are cyclically reduced and none equal to $1$, the proof gives\linebreak{}
$\dim\left(\tps\right)\le-\chi\left(\Sigma\right)$.
\end{rem}

The following theorem is the main result of the current section. It
is established in Section \ref{subsec:Contractibility-of-|T|}.
\begin{thm}
\label{thm:X contractible}The complex of transverse maps $\tps$
is contractible.
\end{thm}

\subsection{A poset of loose transverse maps}

In order to show the contractibility of $\tps$ we introduce a poset
$\L=\L\left(\Sigma,f\right)$ of \emph{loose} transverse maps (see
Definition \ref{def:restrictions}) with exactly two transversion
points on every cycle of $\wedger$. This poset gives rise to a subdivision
of the polysimplicial complex $\tps$, which is well-adapted to the
surgeries we perform to prove contractibility. We are not able to
prove contractibility directly with the constructions $\tps$ or $\left|\T\right|$
from Section \ref{subsec:The-complex-of-transverse maps}. The relation
between $\L$ and $\T$ is analogous to the relation between the set
of matchings $\match^{\kappa\equiv1}$ appearing in Theorem \ref{thm:trwl as finite sum}
and the set of matchings $\matchr^{*}$ appearing in Theorem \ref{thm:combinatorial Laurent expansion of trwl}.
\begin{defn}
\label{def:poset and complex of loose transverse maps}Let $\Sigma$,
$f$ and $V_{o}$ be as in Definition \ref{def:poset of transverse maps}.
The \emph{poset of loose bi-transverse maps realizing $\left(\Sigma,f\right)$,
}denoted $\L=\L\left(\Sigma,f\right)$\marginpar{${\scriptstyle \protect\L=\protect\L\left(\Sigma,f\right)}$}
or $\left(\L,\preceq_{\L}\right)$, consists of the set of isotopy
classes relative to $V_{o}$ of \textbf{\emph{loose}} transverse maps
realizing $\left(\Sigma,f\right)$ with parameters $\kappa_{x}=1$
for all $x\in B$. The order is defined as follows: assume that $g$
is a transverse map realizing $\left(\Sigma,f\right)$ with $\kappa_{x}=\mathbf{3}$
for all $x\in B$, let $h_{1}$ be the transverse map obtained from
$g$ by forgetting the two exterior transversion points for every
$x\in B$, and let $h_{2}$ be the one obtained from $g$ by forgetting
the two interior points for every $x\in B$. If $h_{1}$ and $h_{2}$
are loose, then $\left[h_{1}\right]\preceq_{\L}\left[h_{2}\right]$
in $\L$. 

The geometric realization of $\L$, denoted $\left|\L\right|$\marginpar{$\left|\protect\L\right|$},
is the order complex of $\L$: the simplicial complex with vertices
corresponding to the elements of $\L$ and an $m$-simplex for every
chain $\left[h_{0}\right]\prec\ldots\prec\left[h_{m}\right]$ of length
$m+1$.
\end{defn}

In other words, $\left[h_{1}\right]\preceq_{\L}\left[h_{2}\right]$
whenever the $x$-arcs and curves of $\left[h_{1}\right]$ can be
arranged to be ``nested'' inside those of $h_{2}$, i.e.~to lie
inside the $\left(x,0\right)$-zones of $h_{2}$, for every letter
$x\in B$, so that the resulting map is a legal transverse map with
four transversion points for every $x$. Another way to put it is
that $\left[h_{1}\right]\preceq_{\L}\left[h_{2}\right]$ if and only
if there are representatives $h_{1}'$ and $h_{2}'$, respectively,
which are identical as maps, and are transverse to four points $\left(x,0\right),\ldots,\left(x,3\right)$
in every cycle of $\wedger$, such that the ``official'' transversion
points of $h'_{1}$ are $\left(x,1\right)$ and $\left(x,2\right)$,
while the ``official'' transversion points of $h'_{2}$ are $\left(x,0\right)$
and $\left(x,3\right)$.

Note that the relation $\preceq_{\L}$ is indeed a partial order:
as the transverse map $h$ with $\kappa_{x}=3$ for every $x\in B$
is allowed to be loose, i.e., to violate \textbf{Restriction 3}, we
get the desired reflexivity: $\left[h\right]\preceq_{\L}\left[h\right]$
for every $\left[h\right]\in\L$. For transitivity, assume $\left[h_{1}\right]\preceq_{\L}\left[h_{2}\right]\preceq_{\L}\left[h_{3}\right]$.
By definition, this means one can draw the $x$-arcs/curves of some
$h'_{2}\in\left[h_{2}\right]$ inside the $\left(x,0\right)$-zones
of $h_{3}$ to obtain a legal loose transverse map, and likewise to
draw the $x$-arcs/curves of some $h'_{1}\in\left[h_{1}\right]$ inside
the $\left(x,0\right)$-zones of $h_{2}'$ to obtain a legal loose
transverse map. The union of all three collections of arcs and curves
gives a loose transverse map $g$ with $\kappa_{x}=5$ for all $x$.
By forgetting $\left(x,1\right)$ and $\left(x,4\right)$ for every
$x$, we get a map that shows $\left[h_{1}\right]\preceq_{\L}\left[h_{3}\right]$.
Finally, if $\left[h_{1}\right]\preceq_{\L}\left[h_{2}\right]\preceq_{\L}\left[h_{1}\right]$,
we obtain in a similar fashion a map $g$ with $\kappa_{x}=5$ in
which the $\left(x,0\right)$-arcs/curves are isotopic to the $\left(x,2\right)$-arcs/curves.
This forces the $\left(x,1\right)$-arcs/curves to be isotopic to
$\left(x,0\right)$ and to $\left(x,2\right)$. Analogously, the $\left(x,4\right)$-collection
is isotopic to the $\left(x,3\right)$-collection and to the $\left(x,5\right)$-collection.
Thus $\left[h_{2}\right]=\left[h_{1}\right]$ and we have established
antisymmetry.
\begin{prop}
\label{prop:|L| cong |X|}The spaces $\tps$ and $\left|\L\right|$
are homeomorphic. Moreover, there exists an homeomorphism $\alpha:\left|\L\right|\stackrel{\cong}{\to}\tps$,
through which the simplices of $\left|\L\right|$ subdivide the polysimplices
of $\tps$.
\end{prop}

The proof relies on the following general lemma:
\begin{lem}
\label{lem:subdivision of a simplex by the poset of pairs}For a finite
chain (totally ordered set) $C$, let $\left(P_{C},\preceq\right)$
be the poset consisting of $\left\{ \left(i,j\right)\in C\times C\,\middle|\,i\le_{C}j\right\} $
with partial order given by $\left(i_{1},j_{1}\right)\preceq\left(i_{2},j_{2}\right)$
if and only if $i_{2}\le_{C}i_{1}\le_{C}j_{1}\le_{C}j_{2}$. Then
there is a canonical homeomorphism $f_{C}\colon\left|P_{C}\right|\to\overline{\Delta_{C}}$
from the order complex of $P_{C}$ to the closed $\left(\left|C\right|-1\right)$-simplex
with vertices the elements of $C$. Moreover, the family $\left\{ f_{C}\right\} _{C}$
of homeomorphisms respects subsets: for every subset $C'\subseteq C$,
$f_{C}\Big|_{C'}=f_{C'}$ and in particular $f_{C}\left(\left|P_{C'}\right|\right)=\overline{\Delta_{C'}}$.
\end{lem}

\begin{proof}
We write points in $\overline{\Delta_{C}}$ as $\sum_{c\in C}t_{c}\cdot c$
with $t_{c}\ge0$ and $\sum_{c}t_{c}=1$. If $\left(i_{0},j_{0}\right)\prec\ldots\prec\left(i_{m},j_{m}\right)$
is a chain in $P_{C}$, we write a point in the corresponding $m$-simplex
of $\left|P_{C}\right|$ as $t_{0}\cdot\left(i_{0},j_{0}\right)+\ldots+t_{m}\cdot\left(i_{m},j_{m}\right)$
with $t_{\ell}\ge0$ for all $\ell\in\left[m\right]$ and $\sum t_{\ell}=1$.
However, we recursively define $f_{C}$ on any linear combination
of $\left(i_{0},j_{0}\right),\ldots,\left(i_{m},j_{m}\right)$ with
image some linear combination of $\left\{ c\in C\right\} $. The definition
is the following:
\[
f_{C}\left(\sum_{s=0}^{m}t_{s}\cdot\left(i_{s},j_{s}\right)\right)\stackrel{\mathrm{def}}{=}\sum_{s=0}^{m}\left(\frac{t_{s}}{2}\cdot i_{s}+\frac{t_{s}}{2}\cdot j_{s}\right).
\]
We recursively define the converse map, $\phi_{C}\colon\overline{\Delta_{C}}\to\left|P_{C}\right|$,
again on any linear combination. For $c\in C$:
\[
\phi_{C}\left(t\cdot c\right)=t\cdot\left(c,c\right).
\]
If $i_{0}\le_{C}i_{1}$ then 
\[
\phi_{C}\left(t_{0}\cdot i_{0}+t_{1}\cdot i_{1}\right)=\begin{cases}
\left(t_{0}-t_{1}\right)\cdot\left(i_{0},i_{0}\right)+2t_{1}\cdot\left(i_{0},i_{1}\right) & t_{0}>t_{1}\\
2t_{0}\cdot\left(i_{0},i_{1}\right) & t_{0}=t_{1}\\
2t_{0}\cdot\left(i_{0},i_{1}\right)+\left(t_{1}-t_{0}\right)\cdot\left(i_{1},i_{1}\right) & t_{0}<t_{1}
\end{cases}.
\]
Finally, if $i_{0}\le_{C}\le i_{1}\le_{C}\ldots\le_{C}i_{m}$, then
\begin{multline*}
\phi_{C}\left(t_{0}\cdot i_{0}+\ldots+t_{m}\cdot i_{m}\right)=\\
=\begin{cases}
\phi_{C}\left(\left(t_{0}-t_{m}\right)\cdot i_{0}+t_{1}\cdot i_{1}+\ldots+t_{m-1}\cdot i_{m-1}\right)+2t_{m}\cdot\left(i_{0},i_{m}\right) & t_{0}>t_{m}\\
2t_{0}\cdot\left(i_{0},i_{m}\right)+\phi_{C}\left(t_{1}\cdot i_{1}+\ldots+t_{m-1}\cdot i_{m-1}\right) & t_{0}=t_{m}\\
2t_{0}\cdot\left(i_{0},i_{m}\right)+\phi_{C}\left(t_{1}\cdot i_{1}+\ldots+t_{m-1}\cdot i_{m-1}+\left(t_{m}-t_{0}\right)\cdot i_{m}\right) & t_{0}<t_{m}
\end{cases}.
\end{multline*}
It is easy to verify that $f_{C}$ and $\phi_{C}$ are inverse to
each other, that they are continuous and that, indeed, $f_{C}\Big|_{C'}=f_{C'}$
for every subset $C'\subseteq C$. In Figure \ref{fig:pair-poset}
we illustrate the resulting subdivision of $\overline{\Delta_{C}}$
when $C=\left\{ 0\prec1\prec2\right\} $.
\end{proof}
\begin{figure}
\centering{}\includegraphics[clip,scale=0.5]{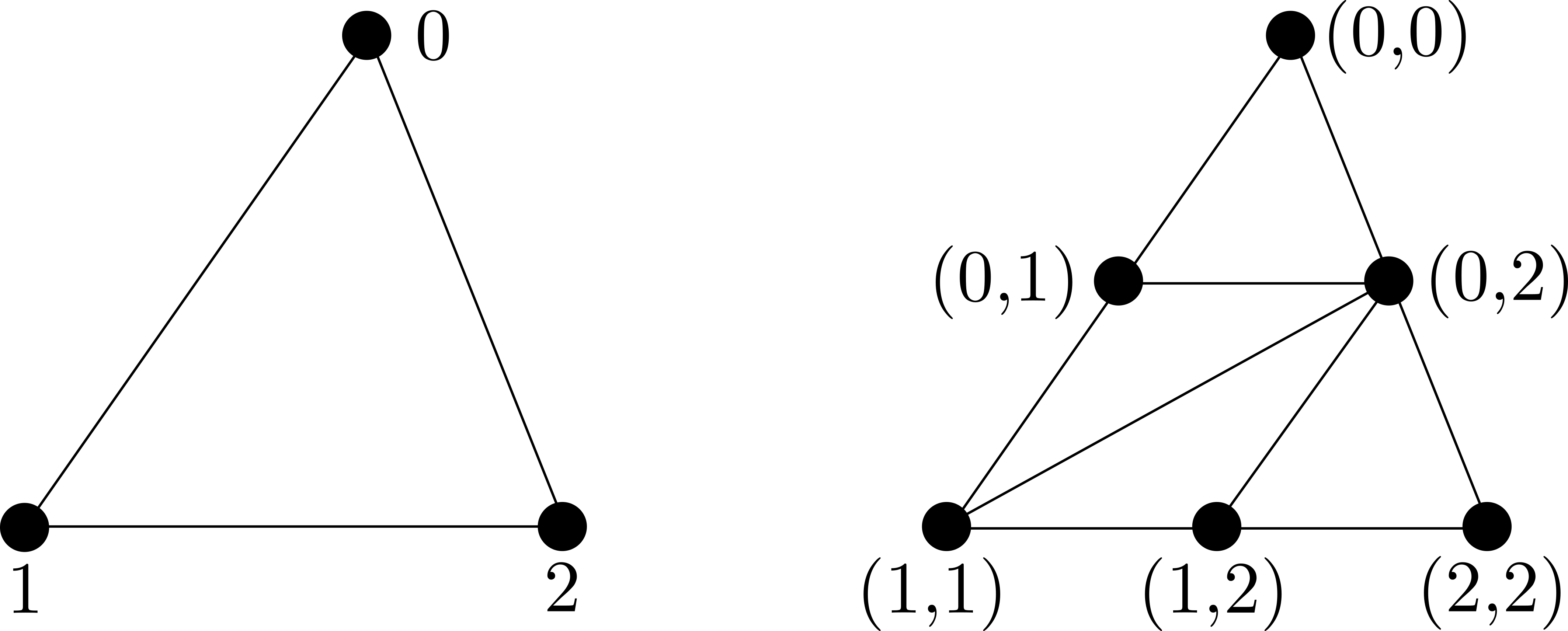}\caption{The subdivision (on the right) of the $2$-simplex $\overline{\Delta_{C}}$
(on the left) given by the order complex of the poset of pairs $P_{C}$,
where $C=\left\{ 0\prec1\prec2\right\} $ is the chain with three
elements.\label{fig:pair-poset} }
\end{figure}

\begin{proof}[Proof of Proposition \ref{prop:|L| cong |X|}]
 Let $c=\left\{ \left[h_{0}\right]\prec_{\L}\ldots\prec_{\L}\left[h_{m}\right]\right\} $
be a chain in $\L$. As above, find $h'_{0},\ldots,h'_{m}$ so that
$h'_{j}\in\left[h_{j}\right]$ and so that for every $0\le j\le m-1$,
the $x$-arcs/curves of $h'_{j}$ are located inside the $\left(x,0\right)$-zones
of $h'_{j+1}$ and together they yield a legal (loose) transverse
map with four transversion points for every letter. Let $g^{\mathrm{loose}}\left(c\right)$
be the loose transverse map with $2\left(m+1\right)$ transversion
points for all $x\in B$, obtained as the union of the collections
of arcs and curves of $h'_{0},\ldots,h'_{m}$. Let \marginpar{$g^{\mathrm{strict}}$}$g^{\mathrm{strict}}\left(c\right)$
be the strict transverse maps obtained from $g^{\mathrm{loose}}\left(c\right)$
by forgetting every transverse point $\left(x,j\right)$ such that
the collection of $\left(x,j\right)$-zones of $g^{\mathrm{loose}}\left(c\right)$
violates \textbf{Restriction 3}, namely, so that the collection of
$\left(x,j\right)$-arcs/curves is isotopic to the $\left(x,j+1\right)$-collection.
Note that $\left[g^{\mathrm{strict}}\left(c\right)\right]$ is a well-defined
element of $\T$, and we denote by $\left\{ \kappa_{x}\left(c\right)\right\} _{x\in B}$
its parameters. For a singleton $\left[h\right]\in\L$, we denote
also $g^{\mathrm{strict}}\left(h\right)$ the strict transverse map
corresponding to the single-element chain $\left\{ \left[h\right]\right\} $.

The sought-after homeomorphism $\alpha\colon\left|\L\right|\to\tps$
is defined per simplex, where the simplex corresponding to the chain
$c\subseteq\L$ is mapped into the polysimplex of $\tps$ corresponding
to $\left[g^{\mathrm{strict}}\left(c\right)\right]$. The exact definition
goes through the single-letter case, using Fact \ref{fact:product of posets}.
More concretely, for $\left[g\right]\in\T$, let $\L_{\le\left[g\right]}\stackrel{\mathrm{def}}{=}\left\{ \left[h\right]\in\L\,\middle|\,\left[g^{\mathrm{strict}}\left(h\right)\right]\preceq_{\T}\left[g\right]\right\} $.
While $\L$ is certainly \emph{not} a product of its projections on
the different letters $x\in B$, it is such a product locally inside
$\L_{\le\left[g\right]}$: for $x\in B$, let $\gamma^{x}\left(g\right)$
denote the collection of $x$-arcs and curves of $g$ (namely, the
union over $j\in\left[\kappa_{x}\right]$ of $\left(x,j\right)$-arcs/curves).
Let $P^{x}\left(g\right)$ denote the poset consisting of $\left\{ \gamma_{i,j}^{x}\left(g\right)\right\} _{0\le i\le j\le\kappa_{x}\left(g\right)}$,
with $\gamma_{i_{1},j_{1}}^{x}\left(g\right)\le_{P^{x}\left(g\right)}\gamma_{i_{2},j_{2}}^{x}\left(g\right)$
if and only if $i_{2}\le i_{1}\le j_{1}\le j_{2}$. Here, $\gamma_{i,j}^{x}\left(g\right)$
can be thought of as the union of $\left(x,i\right)$- and $\left(x,j\right)$-arcs/curves
inside $\gamma^{x}\left(g\right)$. It is easy to see that $\L_{\le\left[g\right]}$
is isomorphic as a poset to a direct product of posets given by 
\[
\L_{\le\left[g\right]}\cong\prod_{x\in B}P^{x}\left(g\right),
\]
where $\left[h\right]\in\L_{\le\left[g\right]}$ corresponds to $\prod_{x\in B}\gamma^{x}\left(g^{\mathrm{strict}}\left(h\right)\right)$.
Hence, 
\[
\left|\L_{\le\left[g\right]}\right|=\left|\prod_{x\in B}P^{x}\left(g\right)\right|\stackrel{\mathrm{Fact}~\ref{fact:product of posets}}{\cong}\prod_{x\in B}\left|P^{x}\left(g\right)\right|\stackrel{\mathrm{Lemma}~\ref{lem:subdivision of a simplex by the poset of pairs}}{\cong}\prod_{x\in B}\overline{\Delta_{\kappa_{x}\left(g\right)}}~=~\overline{\mathrm{polysim}}\left(\left[g\right]\right),
\]
and this homeomorphism defines $\alpha_{\left|\L_{\le\left[g\right]}\right|}$.
This definition expands to a well defined homeomorphism $\alpha\colon\left|\L\right|\to\tps$
because for $\left[g'\right]\preceq_{\T}\left[g\right]$, the restriction
of $\alpha_{_{\left|\L_{\le\left[g\right]}\right|}}$ to $\left|\L_{\le\left[g'\right]}\right|$
is exactly $\alpha_{\left|\L_{\le\left[g'\right]}\right|}$. This
shows that the image of the open simplices in $\left|\L\right|$ corresponding
to the chains $\left\{ c\,\middle|\,\left[g^{\mathrm{strict}}\left(c\right)\right]=\left[g\right]\right\} $
subdivides the open polysimplex $\mathrm{polysim}^{\circ}\left(\left[g\right]\right)$,
and the image of $\left|\L_{\le\left[g\right]}\right|$ subdivides
the closed simplex $\overline{\mathrm{polysim}}\left(\left[g\right]\right)$. 
\end{proof}

\subsection{Contractibility of the transverse map complex\label{subsec:Contractibility-of-|T|}}

To prove the contractibility of $\tps$, we use \emph{null-arcs}:
\begin{defn}
\begin{itemize}
\item A \emph{null-arc}\marginpar{null-arc} for $\left(\Sigma,f\right)$
is an arc $\omega$ in $\Sigma$ with endpoints in $V_{o}\subset\partial\Sigma$
and interior disjoint from $\partial\Sigma$, so that if $\omega$
is closed, it is not nullhomotopic, and such that $f_{*}\left(\omega\right)=1$.
The latter condition means, in other words, that the image of $\omega$
under $f$ is nullhomotopic in $\wedger$ relative the endpoints.
\item A \emph{system of null-arcs} for $\left(\Sigma,f\right)$ is a collection
of null-arcs that are disjoint away from their endpoints and such
that no two are isotopic relative to $V_{o}$.
\item If $\Omega$ is a system of null-arcs for $\left(\Sigma,f\right)$,
then $\T_{\Omega}=\T_{\Omega}\left(\Sigma,f\right)$ and $\L_{\Omega}=\L_{\Omega}\left(\Sigma,f\right)$\marginpar{$\protect\T_{\Omega},\protect\L_{\Omega}$}
are the subposets of $\T$ and $\L$, respectively, of isotopy classes
of transverse maps which map $\bigcup_{\omega\in\Omega}\omega$ to
$o\in\wedger$.
\end{itemize}
\end{defn}

Put differently, $\T_{\Omega}$ and $\L_{\Omega}$ consist of isotopy
classes of transverse maps with arcs/curves collections that can be
drawn away from $\Omega$, meaning that every $\omega\in\Omega$ is
entirely contained in some $o$-zone of the transverse map. 

Note that $\T_{\Omega}$ and $\L_{\Omega}$ are downward closed: if
$g'\preceq_{\T}g\in\T_{\Omega}$ then $g'\in\T_{\Omega}$ and likewise
for $\L_{\Omega}$. Hence $\left|\T_{\Omega}\right|_{\ps}$ and $\left|\L_{\Omega}\right|$
are subcomplexes of $\tps$ and $\left|\L\right|$, respectively.
Moreover:
\begin{claim}
\label{claim:X_Omega =00003D L_Omega}For any system of null-arcs
for $\left(\Sigma,f\right)$, the homeomorphism $\alpha\colon\left|\L\right|\to\tps$
from Proposition \ref{prop:|L| cong |X|} satisfies $\alpha\left(\left|\L_{\Omega}\right|\right)=\left|\T_{\Omega}\right|_{\ps}$.
\end{claim}

\begin{proof}
The homeomorphism $\alpha$ maps the simplex corresponding to the
chain $c$ in $\L$ into the polysimplex corresponding to $\left[g^{\mathrm{strict}}\left(c\right)\right]$.
But belonging to $\T_{\Omega}$ or to $\L_{\Omega}$ depends only
on the $o$-zones of the transverse map, and the $o$-zones of the
top element of $c$ are identical to those of $g^{\mathrm{strict}}\left(c\right)$.
Hence $c$ is contained in $\L_{\Omega}$ if and only if its top element
is in $\L_{\Omega}$, if and only if $\left[g^{\mathrm{strict}}\left(c\right)\right]\in\T_{\Omega}$.
\end{proof}
The following proposition is the main component of the proof of Theorem
\ref{thm:X contractible} concerning the contractibility of $\tps$.
\begin{prop}
\label{prop:deformation retract for null-arc system}Let $\Omega$
be a system of null-arcs for $\left(\Sigma,f\right)$, then there
is a deformation retract of $\tps$ to $\left|\T_{\Omega}\right|_{\ps}$.
In particular, $\T_{\Omega}$ is non-empty.
\end{prop}

\begin{proof}[Proof of Theorem \ref{thm:X contractible} assuming Proposition \ref{prop:deformation retract for null-arc system}]
Since the number of non-isotopic null-arcs that coexist for $\left(\Sigma,f\right)$
is bounded by Euler characteristic considerations, it is obvious there
exist \emph{maximal} systems of null-arcs: systems so that no further
null-arcs can be added to. Let $\Omega$ be a maximal system of null-arcs.
We claim that $\left|\T_{\Omega}\right|_{\ps}$ is a single vertex
of $\tps$. This is enough by Proposition \ref{prop:deformation retract for null-arc system}.

By Proposition \ref{prop:deformation retract for null-arc system},
$\T_{\Omega}$ is non-empty. Since $\T_{\Omega}$ is downward closed,
we can choose $g\in\T_{\Omega}$ with parameters $\kappa_{x}=0$ for
all $x$. Showing that $\left|\T_{\Omega}\right|_{\ps}$ is a single
vertex is equivalent to showing that $g$ is the only point in $\T_{\Omega}$.

To proceed, we claim that every connected component of $\Sigma\setminus\Omega$
has one of the following forms (and see Figure \ref{fig:maximal-system}):
\begin{figure}
\begin{centering}
\includegraphics[viewport=0bp 15bp 282bp 109bp,scale=1.6]{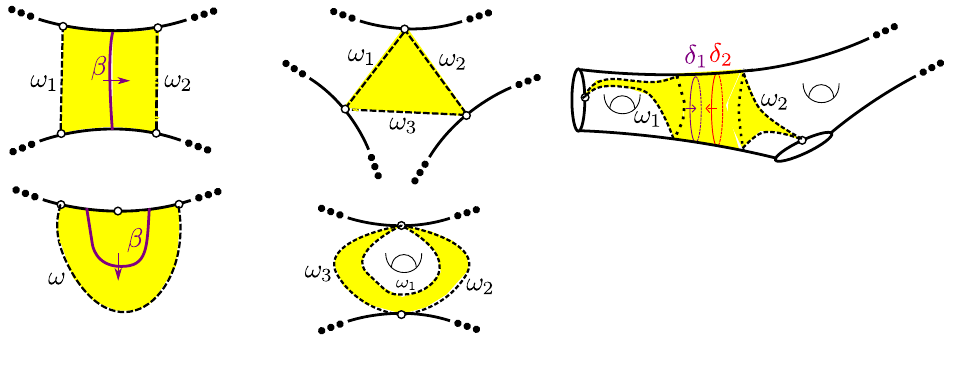}
\par\end{centering}
\caption{\label{fig:maximal-system}This figure shows the different types of
connected components of $\Sigma\setminus\Omega$, where $\Omega$
is a maximal system of null arcs. In the terminology of the proof
of Theorem \ref{thm:X contractible} on Page \pageref{component types for maximal system},
the two drawings on the left are pieces of type $\left(i\right)$,
namely, pieces containing a single arc $\beta$ of the unique transverse
map $g\in\protect\T_{\Omega}$. The two drawings in the middle are
pieces of type $\left(ii\right)$: triangles bounded by three null-arcs.
The drawing on the right is a piece of type $\left(iii\right)$: an
annulus cut out by two closed null-arcs and containing at least one
curve of $g$ (in the drawing: two curves, $\delta_{1}$ and $\delta_{2}$,
corresponding to two different basis elements).}
\end{figure}

\paragraph*{$\left(i\right)$ A rectangle around some arc $\beta$ of $g$\label{component types for maximal system}}

This usually means a rectangle cut out by two null-arcs which are
parallel to $\beta$ with endpoints at the points of $V_{o}$ neighboring
the endpoints of $\beta$. But we also refer here to a bigon cut out
by a single null-arc if $\beta$ connects two adjacent components
of $\partial_{i}\Sigma\setminus V_{o}$, which is possible when the
word $w_{i}$ is not cyclically reduced. 

\paragraph*{$\left(ii\right)$ A triangle bounded by three null-arcs}

\paragraph*{$\left(iii\right)$ An annulus cut out by two closed null-arcs }

\noindent In this case the annulus must contain at least one curve
of $g$ (non-nullhomotopic, evidently).

\medskip{}

Indeed, it is clear that for any arc $\beta$ of $g$, the arc that
is parallel to $\beta$ on either side with endpoints at the points
of $V_{o}$ neighboring the endpoints of $\beta$ is a null-arc and
therefore in $\Omega$ by maximality. So every connected component
of $\Sigma\setminus\Omega$ that contains an arc of $g$, contains
a single arc of $g$ and is of type $\left(i\right)$. This also shows
that components of type $\left(i\right)$ touch all of $\partial\Sigma$.
Any other component $\Sigma'$ of $\Sigma\setminus\Omega$ does not
contain any arc from $g$ and does not touch $\partial\Sigma\setminus V_{o}$.
If $\Sigma'$ contains no curves of $g$ neither, it can be triangulated
by null-arcs and therefore has to be a triangle as in $\left(iii\right)$
by the maximality of $\Omega$. 

Finally, assume that $\Sigma'$ contains a curve $\delta$ of $g$.
First, any component of $\partial\Sigma'$ is a chain of null-arcs,
and by maximality has to consist of a single closed null-arc. Recall
that $\Sigma'$ contains no arcs of $g$, and that any non-nullhomotopic
simple closed curve $c\subset\Sigma'$ disjoint from the curves of
$g$ is a null-curve (see Definition \ref{def:null-curve-incompressible}).
If $\Sigma'$ is not as described in item $\left(iii\right)$, then
one can add a null-arc to $\Omega$ inside $\Sigma'$ in one of the
following ways: if $\Sigma'$ has at least two boundary components,
draw a curve which leaves the marked point at one boundary component
$\omega_{1}$, takes some path to a different boundary component $\omega_{2}$,
goes around the $\omega_{2}$ and returns to $\omega_{1}$ along the
same way; If $\Sigma'$ has only one boundary component $\omega_{1}$,
there must be a pair of pants contained in $\Sigma'$ which is free
from curves of $g$, and one can draw a new null-curve by going from
the marked point of $\omega_{1}$, entering the pair of pants through
one sleeve, circling another sleeve and going back. This is a contradiction
to maximality. Hence $\Sigma'$ is necessarily of type $\left(iii\right)$.

We can now finish the argument showing that $\left[g\right]$ is the
only element in $\T_{\Omega}$. Let $g'$ be a transverse map for
$\left(\Sigma,f\right)$ with $\left[g'\right]\in\T_{\Omega}$. Obviously,
there are no arcs/curves of $g'$ in components of $\Sigma\setminus\Omega$
of type $\left(ii\right)$. Any $z$-zone of $g'$ is contained in
some component $\Sigma'$ of type $\left(i\right)$ or $\left(iii\right)$.
But the structure of these components guarantees that any such $z$-zone
is either a rectangle or an annulus. Thus $\kappa_{x}\left(g\right)=0$
for all $x\in B$, for otherwise $g'$ violates \textbf{Restriction
3}. We can now see that $\left[g'\right]=\left[g\right]$: its clear
that their arcs are isotopic by the structure of type-$\left(i\right)$
components. Their curves are also isotopic because for every $\Sigma'$
of type $\left(iii\right)$, consider an arc $\alpha$ connecting
the two distinct marked points from $V_{o}$ touching $\Sigma'$.
The image of $\alpha$ under $f$ completely prescribes the curves
of $g'$ inside $\Sigma'$ (here we use also \textbf{Restriction 1}).
\renewcommand{\labelenumi}{\arabic{enumi}.} 
\end{proof}

\subsubsection{Proof of Proposition \ref{prop:deformation retract for null-arc system}}

Now we come to prove Proposition \ref{prop:deformation retract for null-arc system}
and show that $\tps$ deformation retracts to $\left|\T_{\Omega}\right|_{\ps}$.
Using Proposition \ref{prop:|L| cong |X|} and Claim \ref{claim:X_Omega =00003D L_Omega},
we actually prove the equivalent statement that $\left|\L\right|$
deformation retracts to $\left|\L_{\Omega}\right|$. The general strategy
to prove Proposition \ref{prop:deformation retract for null-arc system}
is to perform local surgeries to gradually simplify transverse maps
by removing intersections of their arcs and curves with the null-arcs
in $\Omega$. The complexity of a given transverse map in $\L$ is
measured in terms of ``depth of words along null-arcs'':

\subsubsection*{Depth of words along null-arcs}

Fix an arbitrary orientation along every null-arc in $\Omega$. For
every element $\left[h'\right]\in\L$, pick a loose transverse map
$h\in\left[h'\right]$ so that the arcs and curves are in minimal
position with respect to $\Omega$, meaning there are no bigons cut
out by $\Omega$ and the arcs/curves of $h$. Every null-arc $\omega\in\Omega$
may cross arcs and curves of $h$, and we record these crossings as
a word $u_{\omega}\left(h\right)$\marginpar{$u_{\omega}\left(h\right)$},
writing 
\begin{align*}
P_{x} & \text{ if the arc/curve has color \ensuremath{(x,0)}},\\
Q_{x} & \text{ if the arc/curve has color \ensuremath{(x,1)}.}
\end{align*}
Put differently, we consider the path $h\left(\omega\right)$ in $\wedger$,
and write $P_{x}$ whenever it crosses $\left(x,0\right)$ and $Q_{x}$
whenever it crosses $\left(x,1\right)$. Note that $h\left(\omega\right)$
begins and ends at $o$, and as $\omega$ is a null-arc and $h$ homotopic
to $f$, we get that $h\left(w\right)$ is nullhomotopic relative
to its endpoints. This means that the word $u_{\omega}\left(h\right)$
can be reduced to the empty word by repeatedly deleting consecutive
pairs of the form $P_{x}P_{x}$ or $Q_{x}Q_{x}$. 

For a general word in the alphabet $\left\{ P_{x},Q_{x}\right\} _{x\in B}$,
we define its length as the length of its reduced form (it is standard
the the reduced form does not depend on the choice of series of reduction
steps). We define \emph{the depth} of a word as the maximal length
of a prefix. For example, in the word below, which reduces to the
empty word, the superscripts denote the length of each prefix:

\[
^{\,0\,}P_{x}{}^{\,1\,}Q_{x}{}^{\,2\,}P_{y}{}^{\,3\,}P_{y}{}^{\,2\,}Q_{z}{}^{\,3\,}Q_{z}{}^{\,2\,}Q_{t}{}^{\,3\,}P_{t}{}^{\,4\,}Q_{t}{}^{\,5\,}Q_{t}{}^{\,4\,}P_{t}{}^{\,3\,}Q_{t}{}^{\,2\,}Q_{x}{}^{\,1\,}P_{x}{}^{\,0\,}.
\]
Hence the depth of this word is $5$. We denote the depth of the word
$u_{\omega}\left(h\right)$ by $\mathrm{depth}\left(u_{\omega}\left(h\right)\right)$\marginpar{$\mathrm{depth}\left(u_{\omega}\left(h\right)\right)$}.

Notice that $\mathrm{depth}\left(u_{\omega}\left(h\right)\right)=0$
if and only if $\omega$ does not intersect any arcs or curves of
$h$, namely, if and only if $\omega$ is contained inside some $o$-zone
of $h$. Thus $\left[h\right]\in\L_{\Omega}$ if and only if $\mathrm{depth}\left(u_{\omega}\left(h\right)\right)=0$
for all $\omega\in\Omega$.

We use the depth to filter $\L$: for $n\in\mathbb{Z}_{\ge0}$ we
let
\[
{\cal P}_{n}\stackrel{\mathrm{def}}{=}\left\{ \left[h\right]\in\L\,\middle|\,\mathrm{depth}\left(u_{\omega}\left(h\right)\right)\le n~\mathrm{for~all}~\omega\in\Omega\right\} .
\]
Then 
\[
\L_{\Omega}={\cal P}_{0}\subseteq\P_{1}\subseteq\ldots\subseteq\P_{n}\subseteq\ldots\subseteq\L
\]
is a countable filtration of $\L$ and 
\[
\bigcup_{n=0}^{\infty}\P_{n}=\L.
\]

\subsubsection*{A deformation retract $\left|\protect\P_{n}\right|\to\left|\protect\P_{n-1}\right|$}

Let $h$ with $\left[h\right]\in\L$ and $\omega\in\Omega$ satisfy
that $\mathrm{depth}\left(u_{\omega}\left(h\right)\right)=n$, and
consider the prefixes of length $n$ in $u_{\omega}\left(h\right)$.
If the last letter of such a prefix is, say, $P_{x}$, then so is
the following letter. Each of these two letters correspond to a point
where $\omega$ crosses an $\left(x,0\right)$-arc/curve of $h$.
We call the segment of $\omega$ cut out by these two crossing points
a \emph{depth-$n$ leaf}\marginpar{depth-$n$ leaf} of $h$ in $\Omega$.
The deformation retract we shall construct ``prunes'' all depth-$n$
leaves of the elements of $\P_{n}$. 

\paragraph*{Parity assumption}

A crucial observation here is that for every null arc $\omega$ and
every $h$, if we cut $\omega$ to segments using the crossing points
with the arcs and curves of $h$, then the segments alternate between
belonging to $o$-zones of $h$ and belonging to $z$-zones of $h$,
with the first segment always in an $o$-zone. So if $n$ is even,
every depth-$n$ leaf is contained in some $o$-zone, while if $n$
is odd, every depth-$n$ leaf is contained in some $z$-zone. In what
follows we assume that $n$ is even and so all depth-$n$ leaves are
contained in $o$-zones. The other case is very similar, and we shall
point out steps of the proof where there is an important difference
between the two cases.\\

The deformation retract $\left|\P_{n}\right|\to\left|\P_{n-1}\right|$
is based on a map $r_{n}\colon\P_{n}\to\P_{n-1}$ between the underlying
posets.
\begin{defn}
\label{def:r_n}For $\left[h\right]\in\P_{n}$ assume that $h$ is
in minimal position with respect to $\Omega$. Define $r_{n}\left(\left[h\right]\right)$\marginpar{$r_{n}$}
by the following two steps:\\
(i) ~Perform an $H$-move (see Definition \ref{def:H-move}) along
every depth-$n$ leaf of $h$ in $\Omega$ to obtain $h'$, a transverse
map for $\left(\Sigma,f\right)$.\\
(ii) If $n$ is even (respectively, odd) consider all $o$-zones (respectively,
$z$-zones) in $h'$ which violate \textbf{Restriction 1} and remove
them\footnote{As we explained in the proof of Lemma \ref{lem:X not empty}, in the
current scenario, a zone violating \textbf{Restriction 1} is necessarily
a zone bounded by curves all of which are of the same color. By removing
the zone we mean removing all bounding curves to obtain a new transverse
map, and this procedure does not change the homotopy type of the map
relative to $V_{o}$.} to obtain $h''$, a transverse map for $\left(\Sigma,f\right)$.
Then set $r_{n}\left(\left[h\right]\right)\stackrel{\mathrm{def}}{=}\left[h''\right]$.
\end{defn}

Recall that all null-arcs in $\Omega$ are disjoint away from their
endpoints, so all depth-$n$ leaves of $h$ are disjoint, and so the
different $H$-moves in step $\left(i\right)$ do not interact with
each other and can be performed simultaneously. Also note that $r_{n}\left(\left[h\right]\right)$
does not depend on the representative $h$ of $\left[h\right]$. See
Figure \ref{fig:This-figure-shows-r_n} for an illustration of how
the $r_{n}$ act on transverse maps. 

\begin{figure}

\begin{centering}
\includegraphics{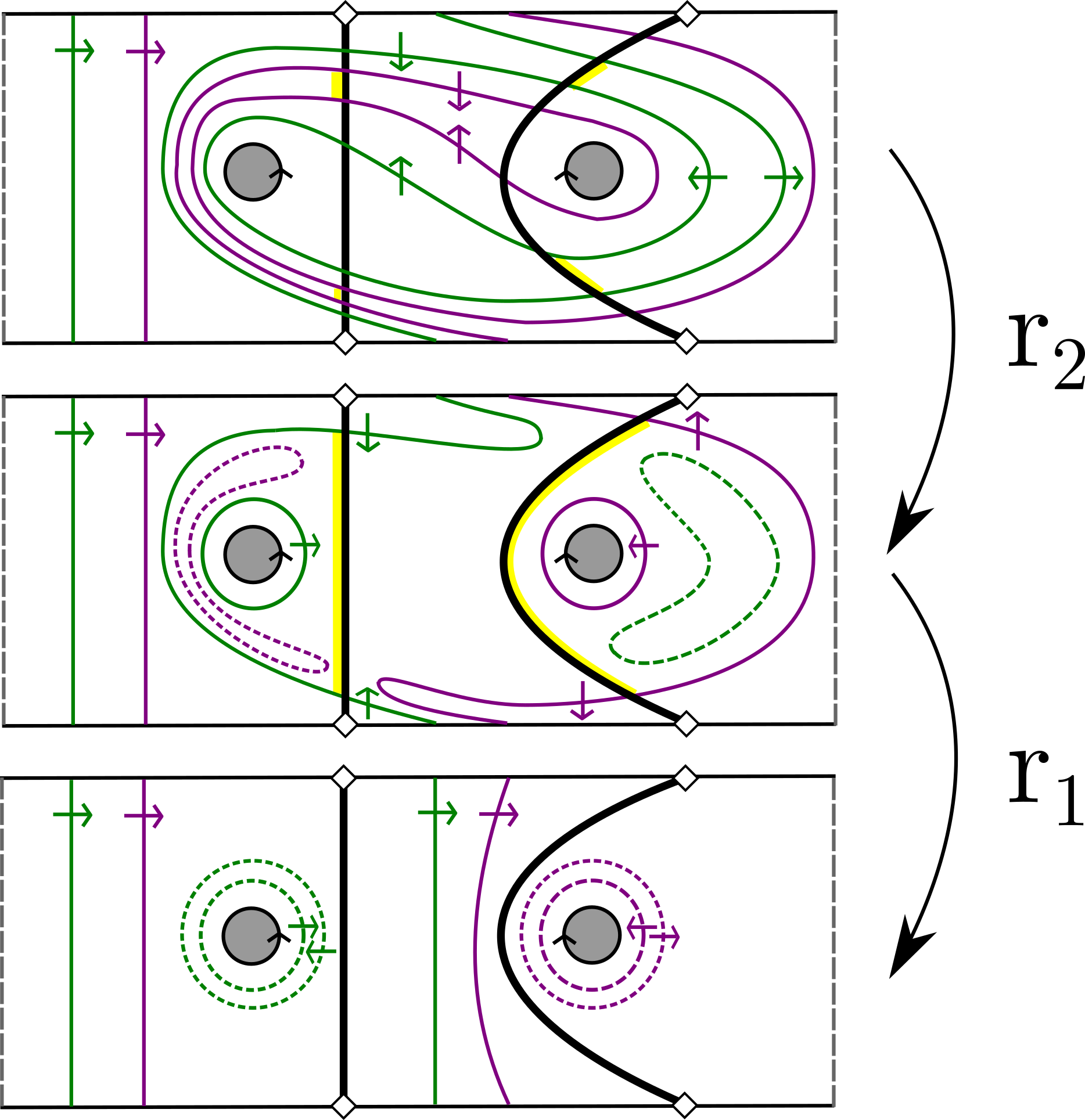}\caption{\label{fig:This-figure-shows-r_n}This figure shows the effects of
$r_{2}$ and $r_{1}$ on a transverse map on a genus $1$ surface
with 2 boundary components. The surface is depicted as a rectangle
with 2 holes (shaded) whose boundaries are identified according to
the labeled orientations, and with the two dashed vertical sides of
the rectangle also identified. Green corresponds to $(x,0)$ and purple
corresponds to $(x,1)$. The two null-arcs in the system are the thick
black arcs. Yellow shading indicates depth-$n$ leaves of the guide
arcs. Dashed curves are those to be removed by Step $\left(ii\right)$
of $r_{i}$ -- see Definition \ref{def:r_n}.}
\par\end{centering}
\end{figure}

We still need to explain why $r_{n}\left(\left[h\right]\right)\in\P_{n-1}$.
We do this through the following series of claims:
\begin{claim}
\label{claim:k_x(h'')=00003D1}$\kappa_{x}\left(h''\right)=1$ for
all $x\in B$ .
\end{claim}

\begin{proof}
It is clear that step $\left(i\right)$ of Definition \ref{def:r_n}
does not alter $\kappa_{x}$, so $\kappa_{x}\left(h'\right)=1$. It
remains to show that for every $x\in B$ and $j\in\left[1\right]$,
some $\left(x,j\right)$-arc/curve survives step $\left(ii\right)$.
We remark that this is clear if there is some $\left(x,j\right)$-arc
in $h$, because $r_{n}$ does not modify $h$ near $\partial\Sigma$.
It is less clear, however, when there are only $\left(x,j\right)$-curves. 

Let $\beta$ be some $\left(x,j\right)$-arc or $\left(x,j\right)$-curve
of $h$, and consider $O_{1}$, the $o$-zone of $h$ touching $\beta$.
All depth-$n$ leaves of $h$ are contained inside $o$-zones (recall
our ongoing assumption in the proofs that $n$ is even), and the leaves
inside $O_{1}$ cut it in step $\left(i\right)$ to smaller $o$-zones
of $h'$, separated by ``$z$-tunnels'' along the leaves of depth
n. Let ${\cal O}_{\left(x,j\right)}$ denote the collection of $o$-zones
of $h'$ which are contained in $O_{1}$ and which are removed in
step $\left(ii\right)$ because they contain no marked points and
have only $\left(x,j\right)$-arcs/curves along their boundary. If
${\cal O}_{\left(x,j\right)}$ is empty, we are done, as the $\left(x,j\right)$-arcs/curves
which are the traces of $\beta$ survive in $h''$. So assume ${\cal O}_{\left(x,j\right)}$
is non-empty. It cannot include all the $o$-zones of $h'$ contained
in $O_{1}$, because this would mean that $O_{1}$ itself is redundant.
Thus, there must be some $o$-zone $O_{2}\in{\cal O}_{\left(x,j\right)}$
which borders, through a depth-$n$ leaf, some $o$-zone $O_{3}\notin{\cal O}_{\left(x,j\right)}$
of $h'$ which is contained in $O_{1}$. Since the leaf separating
$O_{2}$ and $O_{3}$ has $\left(x,j\right)$-arcs/curves on both
sides (in $h'$), $O_{3}$ has some bounding $\left(x,j\right)$-arc/curve,
which survives step $\left(ii\right)$.
\end{proof}
We have not shown yet that $r_{n}\left(\left[h\right]\right)\in\L$:
it remains to prove that $h''$ is loose, but the following claim
is the analog of saying that $\left[h\right]\preceq_{\L}\left[h''\right]$:
\begin{claim}
\label{claim:h le h''}There is a transverse map $g$ for $\left(\Sigma,f\right)$
with $\kappa_{x}=3$ for all $x$, so that forgetting $\left(x,0\right)$
and $\left(x,3\right)$ for all $x$ yields a map in $\left[h\right]$
and forgetting $\left(x,1\right)$ and $\left(x,2\right)$ for all
$x$ yields a map in\footnote{If $n$ is odd, the parallel claim is the analog of $\left[h''\right]\preceq_{\L}\left[h\right]$.}
$\left[h''\right]$.
\end{claim}

\begin{proof}
First, construct a transverse map $g$ with $\kappa_{x}=3$ by duplicating
the arcs and curves of $h$, so that the $\left(x,0\right)$-arcs/curves
are isotopic to the $\left(x,1\right)$-arcs/curves, and likewise
with $\left(x,2\right)$ isotopic to $\left(x,3\right)$. Since the
$H$-moves in step $\left(i\right)$ of Definition \ref{def:r_n}
are performed in $o$-zones (recall our assumption that $n$ is even),
we can perform them for $g$, in which they involve only arcs/curves
with color from $\bigcup_{x}\left\{ \left(x,0\right),\left(x,3\right)\right\} $
and occur inside $o$-zones. The resulting map, call it $g'$, shows
the analog of $h\preceq_{\L}h'$. Finally, the $o$-zones of $g'$
are identical (up to homotopy) to those of $h'$, so step $\left(ii\right)$
can be performed in $g'$ by removing all redundant $o$-zones of
$g'$. The resulting map, $g''$, is still transverse with parameters
$\kappa_{x}=3$ for all $x$ by Claim \ref{claim:k_x(h'')=00003D1},
and is the map we need to establish the claim.
\end{proof}
\begin{lem}
\label{lem:h'' loose}$r_{n}\left(\left[h\right]\right)\in\L$.
\end{lem}

\begin{proof}
We need to show that $h''$ is loose, namely that it abides to \textbf{Restrictions
1} and \textbf{2}. Neither step $\left(i\right)$ nor step $\left(ii\right)$
from Definition \ref{def:r_n} change $h$ near $\partial\Sigma$,
so $h''$ abides to \textbf{Restriction 2} because so does $h$. It
remains to show there are no ``redundant'' zones in $h''$, namely,
no zones which violate \textbf{Restriction 1}. Note that the removal
of redundant $o$-zones of $h'$ in step $\left(ii\right)$ enlarges
$z$-zones and possibly merges several $z$-zones into one, but it
does not create new $o$-zones nor does it affect other existing $o$-zones.
So the remaining $o$-zones are not redundant. 

As for $z$-zones, we use the map $g''$ from Claim \ref{claim:h le h''}.
We claim that $g''$ has no redundant $z$-zones. Clearly, $g''$
has no $\left(x,1\right)$-redundant zone, because these are exactly
the $\left(x,0\right)$-zones of $h$, which abide to \textbf{Restriction
1}. Note that every $\left(x,0\right)$-arc/curve of $g''$ is parallel,
at least in some segments, to $\left(x,1\right)$-arcs/curves (by
the nature of $H$-moves). Hence, every $\left(x,0\right)$-zone of
$g''$ must have some bounding $\left(x,1\right)$-arc/curve. Therefore,
a redundant $\left(x,0\right)$-zone in $g''$ has only $\left(x,1\right)$-arcs/curves
at its boundary, and is thus a redundant $o$-zone of $h$, a contradiction.
That there are no redundant $\left(x,2\right)$-zones in $g''$ is
analogous to the $\left(x,0\right)$ case. 

Now, let $Z$ be an arbitrary $z$-zone of $h''$. Without loss of
generality, there is some $\left(x,0\right)$-arc/curve of $h''$
at $\partial Z$. This $\left(x,0\right)$-arc/curve is at $\partial Z_{0}$
for some $\left(x,0\right)$-zone $Z_{0}$ of $g$ contained in $Z$.
By the claim on $g$, this $Z_{0}$ borders some $\left(x,1\right)$-zone
$Z_{1}\subset Z$ of $g$, which borders some $\left(x,2\right)$-zone
$Z_{2}\subset Z$ of $g$. But $Z_{2}$ has some $\left(x,3\right)$-arc/curve
of $g$ at its boundary, which is necessarily a $\left(x,1\right)$-arc/curve
of $h''$ at the boundary of $Z$. Hence $Z$ does not violate \textbf{Restriction
1}.
\end{proof}
\begin{cor}
\label{cor:r_n in P_n-1  and increasing}$r_{n}\left(\left[h\right]\right)\in\P_{n-1}$
and\footnote{For $n$ odd, $r_{n}\left(\left[h\right]\right)\preceq_{\L}\left[h\right]$.}
$\left[h\right]\preceq_{\L}r_{n}\left(\left[h\right]\right)$.
\end{cor}

\begin{proof}
It remains to show that $\mathrm{depth}\left(u_{\omega}\left(h''\right)\right)\le n-1$
for all $\omega\in\Omega$. The $H$-moves of step $\left(i\right)$
in the definition of $r_{n}$ remove all the crossings between $\omega$
and arcs/curves of $h$ which cut out depth-$n$ leaves. It is thus
clear that $\mathrm{depth}\left(u_{\omega}\left(h'\right)\right)\le n-1$.
But whenever $\omega$ enters a redundant zone of $h'$, it has to
leave it through an arc/curve of the same color. So the effect of
removing a redundant zone on the words $u_{w}\left(h'\right)$ is
performing reduction steps (omitting consecutive pairs of the type
$P_{x}P_{x}$ or $Q_{x}Q_{x}$). Reduction moves cannot increase the
depth of the word.
\end{proof}
After establishing that $r_{n}\colon\P_{n}\to\P_{n-1}$, our next
goal is to use $r_{n}$ to obtain the sought-after deformation retract.
We do this using the following general technique concerning posets:

A map $\varphi\colon P\to Q$ between posets which is order-preserving,
in the sense that $p_{1}\le_{P}p_{2}\Longrightarrow\varphi\left(p_{1}\right)\le_{Q}\varphi\left(p_{2}\right)$,
maps a chain $p_{0}<_{P}\ldots<_{P}p_{m}$ in $P$ to a, possibly
``stuttering'', chain $\varphi\left(p_{0}\right)\le_{Q}\ldots\le_{Q}\varphi\left(p_{m}\right)$
in $Q$, so the set $\left\{ \varphi\left(p_{0}\right),\ldots,\varphi\left(p_{m}\right)\right\} $
defines a simplex in the order complex $\left|Q\right|$. This allows
the following natural induced map $\left|\varphi\right|\colon\left|P\right|\to\left|Q\right|$
between the order complexes:
\begin{equation}
\left|\varphi\right|\left(\sum\lambda_{i}p_{i}\right)=\sum\lambda_{i}\varphi\left(p_{i}\right).\label{eq:induced-map-on-order-complexes}
\end{equation}

\begin{lem}
\label{lem:defomration retract from poset map}Let $P$ be a subposet
of the poset $Q$. Assume that $\varphi\colon Q\to P$ satisfies the
following three conditions:
\begin{itemize}
\item $\varphi$ is order-preserving
\item $\varphi$ is a retract, i.e.~$f\Big|_{P}\equiv\mathrm{id}$
\item $\varphi\left(q\right)\le q$ for all $q\in Q$, or $\varphi\left(q\right)\ge q$
for all $q\in Q$
\end{itemize}
Then $\left|\varphi\right|\colon\left|Q\right|\to\left|P\right|$
is a strong deformation retract.
\end{lem}

By a strong deformation retract we mean that there is a homotopy of
$\left|\varphi\right|$ with the identity on $\left|Q\right|$ which
fixes $\left|P\right|$ pointwise throughout the homotopy.
\begin{proof}
Recall that a map $\psi$ between posets is called a poset-morphism
if it is order preserving. If $\psi\colon P\to Q$ is a poset morphism,
we let $\left|\psi\right|$ denote the induced map $\left|\psi\right|\colon\left|P\right|\to\left|Q\right|$
defined as in (\ref{eq:induced-map-on-order-complexes}). If $P$
and $Q$ are posets, $\psi_{0},\psi_{1}\colon P\to Q$ are poset morphisms,
and $\psi_{0}\left(p\right)\le\psi_{1}\left(p\right)$ for every $p\in P$,
then $\left|\psi_{0}\right|$ and $\left|\psi_{1}\right|$ are homotopic.
Indeed, let $\left\{ 0\le1\right\} $ denote the poset with two comparable
elements $0$ and 1. Define a map $\left(\psi_{0},\psi_{1}\right)\colon P\times\left\{ 0\le1\right\} \to Q$
by $\left(p,0\right)\mapsto\psi_{0}\left(p\right)$ and $\left(p,1\right)\mapsto\psi_{1}\left(p\right)$.
This is clearly a poset-morphism by the assumptions, so it induces
a continuous map 
\[
\left|\left(\psi_{0},\psi_{1}\right)\right|\colon\left|P\times\left\{ 0\le1\right\} \right|\to\left|Q\right|.
\]
By Fact \ref{fact:product of posets}, there is an homeomorphism 
\[
\left|P\times\left\{ 0\le1\right\} \right|\overset{\cong}{\to}\left|P\right|\times\left|\left\{ 0\le1\right\} \right|=\left|P\right|\times\left[0,1\right],
\]
so we get that $\left|\left(\psi_{0},\psi_{1}\right)\right|$ is a
continuous map $\left|P\right|\times\left[0,1\right]\to\left|Q\right|$.
Because $\left|\left(\psi_{0},\psi_{1}\right)\right|\Big|_{\left|P\times\left\{ 0\right\} \right|}\equiv\left|\psi_{0}\right|$
and $\left|\left(\psi_{0},\psi_{1}\right)\right|\Big|_{\left|P\times\left\{ 1\right\} \right|}\equiv\left|\psi_{1}\right|$,
the map $\left|\left(\psi_{0},\psi_{1}\right)\right|$ is the sought-after
homotopy. (This result appears in \cite[Section 1.3]{QUILLEN}.)

Note that the map $\varphi\colon Q\to Q$ in the statement of the
lemma and the identity $\id\colon Q\to Q$ satisfy the conditions
regarding $\psi_{0}$ and $\psi_{1}$ above. Hence $\left|\varphi\right|$
is homotopic to the identity. The fact that the homotopy fixes $\left|P\right|$
pointwise follows from the fact that the homotopy above does not move
the points where $\psi_{0}$ and $\psi_{1}$ agree. Namely, if $P_{0}\subseteq P$
is the subposet where $\psi_{0}\left(p\right)=\psi_{1}\left(p\right)$,
then $\left|\left(\psi_{0},\psi_{1}\right)\right|\left(x,t\right)=\psi_{0}\left(x\right)=\psi_{1}\left(x\right)$
for every $x\in\left|P_{0}\right|$ and $t\in\left[0,1\right]$.
\end{proof}
\begin{prop}
The map $r_{n}\colon\P_{n}\to\P_{n-1}$ satisfies the conditions of
Lemma \ref{lem:defomration retract from poset map} and so defines
a strong deformation retract 
\[
\left|r_{n}\right|\colon\left|\P_{n}\right|\to\left|\P_{n-1}\right|.
\]
\end{prop}

\begin{proof}
We already proved above that for $n$ even, $\left[h\right]\preceq_{\L}r_{n}\left(\left[h\right]\right)$
which is the third assumption of Lemma \ref{lem:defomration retract from poset map}.
The second assumption is also clear: if $\left[h\right]\in\P_{n-1}$,
then $h$ admits no depth-$n$ leaves in $\Omega$, and therefore
in Definition \ref{def:r_n}, $h=h'=h''$. It remains to show that
$r_{n}$ is order-preserving.

Let $h_{1}$ and $h_{2}$ be transverse maps so that $\left[h_{1}\right],\left[h_{2}\right]\in\P_{n}$,
with $\left[h_{1}\right]\preceq_{\L}\left[h_{2}\right]$, and assume
that $g$ is a transverse map $g$ with $\kappa_{x}=3$ for all $x$,
so that $h_{1}$ and $h_{2}$ are obtained by forgetting the exterior
and interior, respectively, two transversion points for every letter
$x$. We also assume $g$ is in minimal position with respect to $\Omega$.
For $\omega\in\Omega$, the words $u_{\omega}\left(h_{1}\right)$
and $u_{\omega}\left(h_{2}\right)$ are very much dependent: they
can be constructed simultaneously by following the path $g\left(\omega\right)$
in $\wedger$, and adding a letter to $u_{\omega}\left(h_{1}\right)$
whenever $g\left(\omega\right)$ crosses some $\left(x,1\right)$
or $\left(x,2\right)$ point, and a letter to $u_{\omega}\left(h_{2}\right)$
whenever $g\left(\omega\right)$ crosses some $\left(x,0\right)$
or $\left(x,3\right)$ point. This description shows that whenever
$\omega$ visits an $o$-zone or an $\left(x,1\right)$-zone of $g$,
the prefix of the two words until that point has the same reduced
form and, in particular, the same length. 

Consider a depth-$n$ leaf $\gamma$ of $h_{1}$ in $\omega$. The
beginning of $\gamma$ is at a crossing point of $\omega$ with some
$\left(x,j\right)$-arc/curve of $g$ with $x\in B$ and $j\in\left\{ 1,2\right\} $,
in which $\omega$ leaves an $\left(x,1\right)$-zone of $g$ and
enters some $\left(x,0\right)$- or $\left(x,2\right)$-zone. Without
loss of generality, assume that $\omega$ crosses some $\left(y,1\right)$-arc/curve
with $y\in B$. The image $g\left(\gamma\right)$ is a closed path
in $\wedger$, based at $\left(y,1\right)$, which avoids the segments
$\left[\left(x,1\right),\left(x,2\right)\right]$ for every $x\in B$
-- see Figure \ref{fig:wedge-with-4-points-on-every-letter}. 

\begin{figure}
\centering{}\includegraphics[viewport=230bp 300bp 300bp 460bp,scale=0.5]{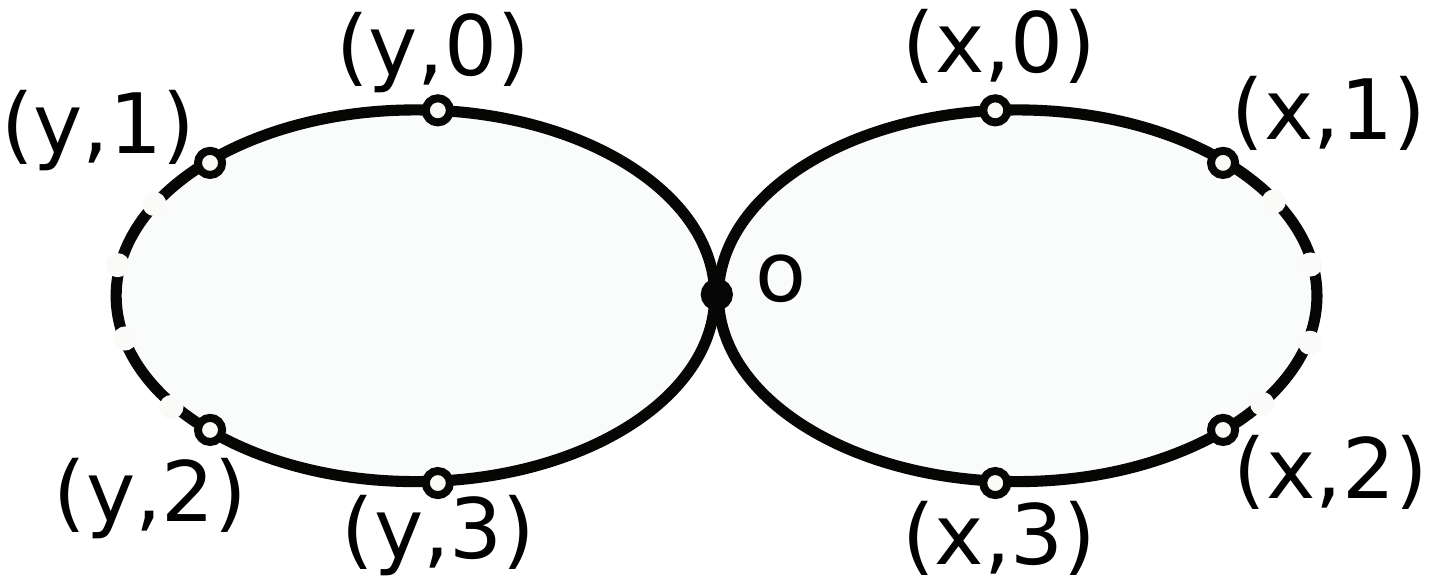}\caption{The wedge $\protect\wedger$ when $r=2$, $B=\left\{ x,y\right\} $.
The marked points are the transversion points of a map $g$, with
$\kappa_{x}\left(g\right)=\kappa_{y}\left(g\right)=3$. If $h_{1}$
is the transverse map obtained from $g$ by forgetting $\left(x,0\right),\left(x,3\right),\left(y,0\right),\left(y,3\right)$
and if $n$ is even, then a depth-$n$ leaf $\gamma$ of $h_{1}$
at some null-arc $\omega$ is contained in some $o$-zone of $h_{1}$,
meaning that $g\left(\gamma\right)$ lives outside the broken segments
in the figure.\label{fig:wedge-with-4-points-on-every-letter} }
\end{figure}

When one follows the prefix of the word $u_{\omega}\left(h_{2}\right)$
along $\gamma$, it is clear, therefore, that at the beginning of
$\gamma$ it has length $n-1$. If it then crosses $\left(y,0\right)$,
it has the same length as the prefix of $u_{\omega}\left(h_{1}\right)$
which is $n$. Then, it could seemingly cross, e.g., $\left(x,3\right)$
for some $x\in B$, but this would increase the length of the prefix
of $u_{\omega}\left(h_{2}\right)$ to $n+1$, which is impossible
as $\left[h_{2}\right]\in\P_{n}$. Hence $g\left(\gamma\right)$ can
only cross the point $\left(y,0\right)$ back and forth. Every two
consecutive such crossings define a depth-$n$ leaf of $h_{2}$ at
$\omega$. 

Therefore, in the notation of Definition \ref{def:r_n}, step $\left(i\right)$
can be performed in two phases: first, perform step $\left(i\right)$
for $h_{2}$, where the depth-$n$ leaves never cross any arcs/curves
of $h_{1}$. Second, perform step $\left(i\right)$ for $h_{1}$:
although a depth-$n$ leaf of $h_{1}$ may cross arcs/curves of $h_{2}$,
the previous paragraph explains why it never crosses arcs/curves of
$h_{2}'$. The resulting $h_{1}'$ and $h_{2}'$ are compatible together
in the sense there is $g'$ with $\kappa_{x}=3$ as in the definition
of the order on $\L$ (although, of course, $h_{1}'$ and $h_{2}'$
may not be in $\L$). See Figure \ref{fig:r_n order preserving}.

\begin{figure}
\centering{}\includegraphics[clip]{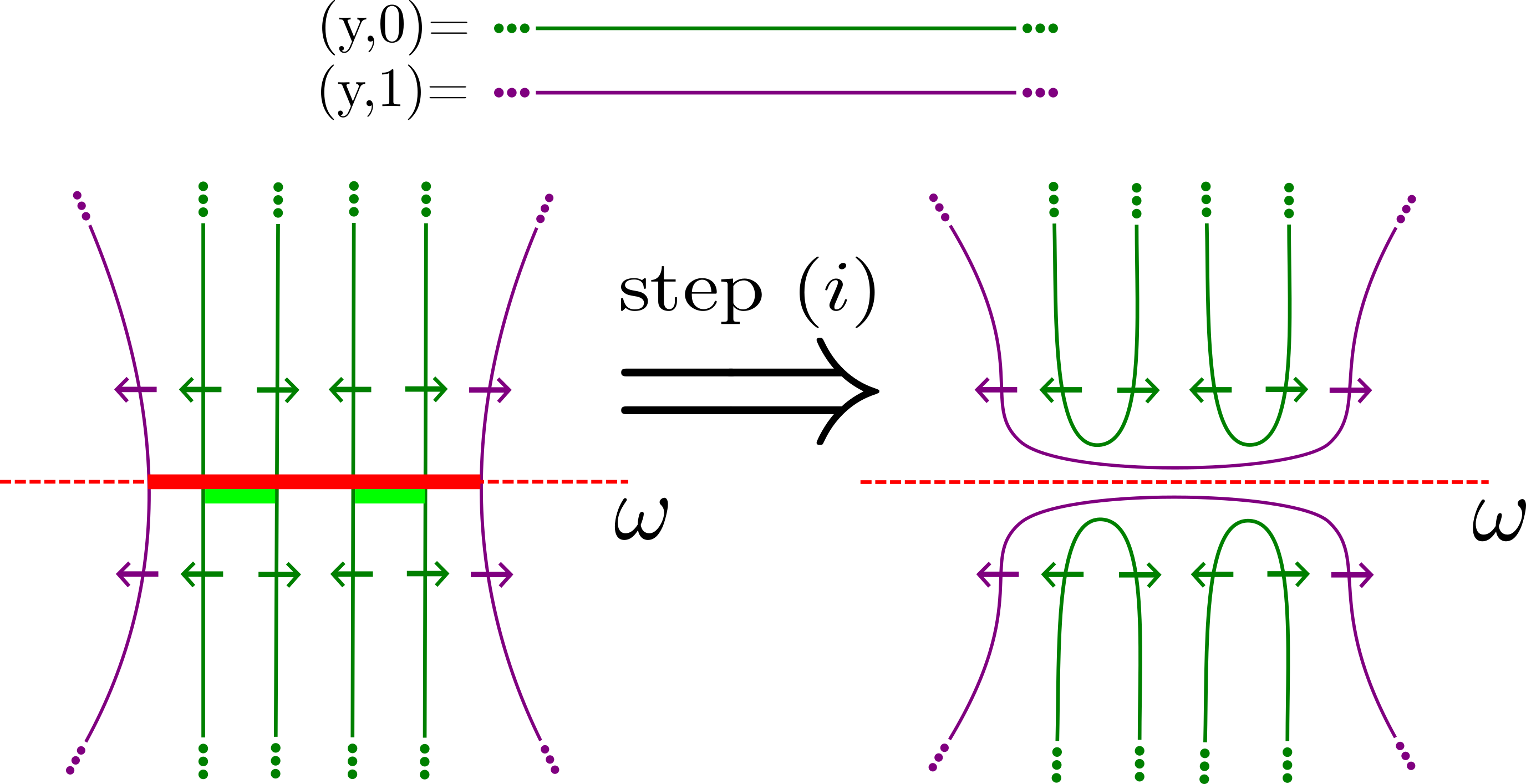}\caption{On the left, a piece of a null-arc $\omega$ crosses some arcs/curves
of $g$, the transverse map with $\kappa_{x}=3$ for all $x$ showing
that $h_{1}\preceq_{\protect\L}h_{2}$, both inside $\protect\P_{n}$.
The thick red part of $\omega$ is a depth-$n$ leaf $\gamma$ of
$h_{1}$, which in $g$ is a segment of $\omega$ between two crossing-points
with $\left(y,1\right)$-arcs/curves. Inside $\gamma$ there are two
depth-$n$ leaves of $h_{2}$, which, in $g$, are two segments cut
out by $\left(y,0\right)$-arcs/curves. The right hand side shows
the result of step $\left(i\right)$ of Definition \ref{def:r_n}
on this local picture, where performing the single $H$-move for $h_{1}$
after the two $H$-moves of $h_{2}$ causes no collisions.\label{fig:r_n order preserving} }
\end{figure}

In step $\left(ii\right)$ of Definition \ref{def:r_n} we now remove
redundant $o$-zones of $h_{1}'$ and of $h_{2}'$. Since the $o$-zones
of $h_{2}'$ and those of $g'$ coincide, removing redundant $o$-zones
of $h_{2}'$ is equivalent to removing redundant $o$-zones of $g'$
and keeps the structure of $g'$ as a legal transverse map with $\kappa_{x}=3$
for all $x$. Denote the resulting map by $\overline{g'}$. However,
we still need to show that removing redundant $o$-zones of $h_{1}'$
does not cause a problem, namely, that any redundant $o$-zone of
$h_{1}'$ does not contain any curves of $h_{2}''$, which are the
same as $\left(x,0\right)$- or $\left(x,3\right)$-curves of $\overline{g'}$
for any $x\in B$.

Indeed, let $O$ be a redundant $o$-zone of $h_{1}'$. Without loss
of generality it is bounded by outgoing $\left(y,1\right)$-curves
of $\overline{g'}$. Assume there is some curve of $h_{2}''$ inside
$O$ which is not a $\left(y,0\right)$-curve of $\overline{g'}$,
say, an $\left(x,3\right)$-curve of $\overline{g'}$. The negative
side of this $\left(x,3\right)$-curve cannot be a redundant $\left(x,2\right)$-zone
of $\overline{g'}$, because then it would be a redundant $z$-zone
of $h_{2}''$ which is impossible by Lemma \ref{lem:h'' loose}. Thus,
this $\left(x,2\right)$-zone of $\overline{g'}$ must have some $\left(x,2\right)$-arc/curve
at its boundary, a contradiction to the assumption that $O$ is redundant.
We conclude that $O$ may only contain $\left(y,0\right)$-curves
of $\overline{g'}$. But then, on their negative side, these curves
must bound a redundant zone (there cannot be marked points from $V_{o}$
inside as $O$ is redundant), and thus should have been removed in
step $\left(ii\right)$ for $h_{2}$. Therefore, step $\left(ii\right)$
for $h_{1}'$ can be performed on $\overline{g'}$ without violating
any rule, and the resulting map, $g''$, shows that $r_{n}\left(h_{1}\right)=\left[h_{1}''\right]\preceq_{\L}\left[h_{2}''\right]=r_{n}\left(\left[h_{2}\right]\right)$.
\end{proof}
\begin{proof}[Proof of Proposition \ref{prop:deformation retract for null-arc system}]
To get a deformation retract of $\left|\L\right|$ to $\left|\L_{\Omega}\right|$
we perform $\left|r_{n}\right|$ at time $\left[\frac{1}{2^{n}},\frac{1}{2^{n-1}}\right]$.
We remark that the fact that $\left|r_{n}\right|$ is a \emph{strong}
deformation retract, namely, keeps $\left|\P_{n-1}\right|$ fixed
pointwise, guarantees that the total deformation retract on $\left|\L\right|$
is well defined.
\end{proof}

\section{The action of $\protect\MCG(f)$ on the complex of transverse maps\label{sec:The-Action-of MCG(f) on T}}

In this section we prove our main results: Theorems \ref{thm:stabilizers have L2-EC},
\ref{thm:K(g,1) for incompressible} and \ref{thm:main}. We begin
with some background on $L^{2}$-Euler characteristics. 

\subsection{$L^{2}$-Betti numbers and $L^{2}$-Euler characteristics\label{subsec:L2-Betti-numbers-and-L2-EC}}

We now define the $L^{2}$-invariants of groups that appear in our
main theorem, although, for the sake of the proofs, one can use Theorem
\ref{thm:ec-as-alt-sum}, Lemma \ref{lem:contractible-with-good-isotropy}
and Theorem \ref{thm:infinite-normal-amenable-subgroup-in-B_infty}
as black boxes.

The following definitions and properties are all found in the book
of Lück \cite{L}; many of the ideas we discuss originate from the
paper of Cheeger and Gromov \cite{CG}. Throughout this subsection,
$G$ is a discrete group. 
\begin{defn}[{\cite[Def. 1.25]{L}}]
\label{def:G-CW-complex}A $G$-\emph{CW}-complex is a \emph{CW-}complex
with a cellular action of $G$ such that if an element of $G$ fixes
an open cell, it acts as the identity on that open cell.
\end{defn}

Following \cite[Def. 1.1]{L}, the \emph{group von Neumann algebra
$\N(G)$ }is defined to be the space of $G$-equivariant bounded operators
from $\ell^{2}(G)$ to itself. Here $\ell^{2}(G)$ is given the standard
Hermitian inner product making it a Hilbert space. Now suppose $X$
is a $G$-\emph{CW-}complex. Denote by $C_{*}^{\sing}\left(X\right)$
the singular chain complex of $X$. This is a complex of left $\mathbb{Z}G$-modules.
Giving $\N(G)$ the structure of an $\left(\N(G),\mathbb{Z}G\right)$-bimodule,
we can form a chain complex
\[
\ldots\xrightarrow{d_{p+1}}\N(G)\otimes_{\mathbb{Z}G}C_{p}^{\sing}(X)\xrightarrow{d_{p}}\N(G)\otimes_{\mathbb{\mathbb{Z}}G}C_{p-1}^{\sing}(X)\xrightarrow{d_{p-1}}\ldots
\]
of $\N(G)$-modules. This is a Hilbert chain complex in the terminology
of \cite[Def. 1.15]{L}. In particular, each piece $\N(G)\otimes_{\mathbb{Z}G}C_{p}^{\sing}(X)$
is a Hilbert module for $\N(G)$ as defined in \cite[Def. 1.5]{L},
$\N(G)\otimes_{\mathbb{Z}G}C_{p}^{\sing}(X)$ is a Hilbert space,
and the boundary maps are bounded $G$-equivariant operators. The
$L^{2}$-homology of the pair $\left(X,G\right)$ we denote by $H_{*}^{(2)}\left(X;G\right)$
and define by 
\[
H_{p}^{(2)}\left(X;G\right)\overset{\mathrm{def}}{=}\frac{\ker(d_{p})}{\mathrm{closure}(\image(d_{p+1}))},
\]
cf. \cite[Def. 6.50, Def. 1.16]{L}. Each of these homology groups
are themselves Hilbert $\N(G)$-modules. Any $\N(G)$-module $M$
has an associated dimension in $[0,\infty]$ called the \emph{von
Neumann dimension} and denoted by $\dim_{\N(G)}(M)$ \cite[Def 6.20]{L}.
The \emph{$L^{2}$-Betti numbers} of the pair $\left(X,G\right)$
are defined by
\[
b_{p}^{(2)}\left(X,G\right)\overset{\mathrm{def}}{=}\dim_{\N(G)}H_{p}^{(2)}\left(X;G\right)\in[0,\infty].
\]
If 
\begin{equation}
\sum_{p\in\mathbb{Z}_{\ge0}}b_{p}^{(2)}\left(X,G\right)<\infty\label{eq:finite-sum-of-betti-numbers}
\end{equation}
then we can also define the \emph{$L^{2}$-Euler characteristic} of
the pair $\left(X,G\right)$ to be
\[
\chi^{(2)}\left(X,G\right)=\sum_{p\in\mathbb{Z}_{\ge0}}(-1)^{p}\cdot b_{p}^{(2)}\left(X,G\right)\in\mathbb{R}.
\]
If $EG$ is a contractible $G$-\emph{CW-}complex with a free action
of $G$ then we define
\[
b_{p}^{(2)}\left(G\right)\overset{\mathrm{def}}{=}b_{p}^{(2)}\left(EG,G\right)
\]
and if moreover (\ref{eq:finite-sum-of-betti-numbers}) holds for
$X=EG$, then we also define as in \cite[Def. 6.79]{L} the \emph{$L^{2}$-Euler
characteristic of $G$ }to be
\[
\chi^{(2)}\left(G\right)\overset{\mathrm{def}}{=}\chi^{(2)}\left(EG,G\right).
\]
Since $EG$ is unique up to $G$-equivariant homotopy equivalence,
it follows for example from \cite[Theorem 6.54]{L} that the quantities
$b_{p}^{(2)}\left(G\right),\chi^{(2)}\left(G\right)$ only depend
on $G$. The existence and $G$-homotopy uniqueness of $EG$ is discussed
in \cite[pg. 33]{L} with references therein to \cite{tD1,tD2}.

Assume $X$ is an arbitrary $G$-\emph{CW-}complex. If $c$ is a cell
of $X$ write $G_{c}$ for the isotropy group (stabilizer) of $c$
in $G$. As in \cite[\S 6.6.1]{L}, we consider the quantities
\[
\left|G_{c}\right|^{-1}
\]
\textbf{where we set $\left|G_{c}\right|^{-1}=0$ if $G_{c}$ is infinite.
}We define following \cite[Def 6.79]{L}
\[
m\left(X,G\right):=\sum_{[c]\in G\backslash X}\left|G_{c}\right|^{-1}\in[0,\infty].
\]

\begin{thm}[{\cite[Thm. 6.80(1)]{L}}]
\label{thm:ec-as-alt-sum}If $m\left(X,G\right)$ is finite then
the sum of $b_{p}^{(2)}\left(X,G\right)$ is finite and, moreover,
\begin{equation}
\chi^{(2)}\left(X,G\right)=\sum_{[c]\in G\backslash X}\left(-1\right)^{\dim c}\left|G_{c}\right|^{-1}.\label{eq:chil2}
\end{equation}
\end{thm}

Following \cite[Def. 7.1]{L} let $\B_{\infty}$\marginpar{$\protect\B_{\infty}$}
denote the class of groups $G$ for which $b_{p}\left(G\right)=0$
for all $p\in\mathbb{Z}_{\ge0}$. 
\begin{lem}
\label{lem:contractible-with-good-isotropy}If $X$ is a contractible
$G$-\emph{CW-}complex, and for all cells $c$ of $X$ the isotropy
group $G_{c}$ is either finite or in $\B_{\infty}$, then 
\[
b_{p}^{(2)}\left(X,G\right)=b_{p}^{(2)}\left(G\right),\quad p\in\mathbb{Z}_{\ge0}.
\]
Hence if also $m\left(X,G\right)$ is finite then $\chi^{(2)}\left(X,G\right)=\chi^{(2)}\left(G\right)$.
\end{lem}

\begin{proof}
This is \cite[Exercise 6.20]{L}. It can be proved by combining \cite[Thm 6.54 (2) and (3)]{L},
and referring to Theorem \ref{thm:ec-as-alt-sum} for the statement
about Euler characteristics.
\end{proof}
To use Lemma \ref{lem:contractible-with-good-isotropy} we need to
have a source of groups lying in $\B_{\infty}$. The following theorem
is essentially due to Cheeger and Gromov (cf. \cite[Corollary 0.6]{CG}).
The precise statement we need can be deduced from \cite[Theorem 7.2, items (1) and (2)]{L}.
Recall that a discrete group is called \emph{amenable} if it has a
finitely additive left invariant probability measure.
\begin{thm}[Cheeger-Gromov]
\label{thm:infinite-normal-amenable-subgroup-in-B_infty}If $G$
is a discrete group containing a normal infinite amenable subgroup
then $G\in\B_{\infty}$.
\end{thm}

\subsection{The complex of transverse maps as a $\protect\mcg\left(f\right)$-\emph{CW}-complex}

The stabilizer $\mcg\left(f\right)$ of $f$ in $\mcg\left(\Sigma\right)$
acts on the poset $\T=\T\left(\Sigma,f\right)$ by precomposition:
if $\left[\rho\right]\in\mcg\left(f\right)$ and $\left[g\right]\in\T$
with parameters $\kappa$, then $\left[\rho\right].\left[g\right]=\left[g\circ\rho^{-1}\right]$
is an element of $\T$ with the same $\kappa$: indeed, $g\circ\rho^{-1}$
is a transverse map realizing $f$ with the exact same transversion
points as $g$. This action is obviously an order preserving action:
if $\left[g_{1}\right]\preceq\left[g_{2}\right]$ then $\left[\rho\right].\left[g_{1}\right]\preceq\left[\rho\right].\left[g_{2}\right]$. 

We now show that this action on $\T$ turns its geometric realization
into a $\mcg\left(f\right)$-$CW$-complex, as in Definition \ref{def:G-CW-complex}.
The properties mentioned above of the action of $\mcg\left(f\right)$
on the poset $\T$ guarantee that this is the case for $\left|\T\right|$,
the order complex of $\T$ (see Page \pageref{order complex |T|}
for the definition of $\left|\T\right|$). We claim this is also the
case for the polysimplicial complex $\tps$:
\begin{lem}
\label{lem:|T|-is-a-G-CW-complex}Let $[(\Sigma,f)]\in\surfaces(\wl)$.
Let $\Gamma=\MCG(f)$. The action of $\Gamma$ on $\T=\T(\Sigma,f)$
makes $\tps$ into a $\Gamma$-CW\emph{-}complex.
\end{lem}

\begin{proof}
If $\left[\rho\right]\in\Gamma$ fixes $[g]\in\T$ we need to show
$\left[\rho\right]$ cannot permute the faces of $\mathrm{polysim\left(\left[g\right]\right)}$.
But $\left[g\circ\rho^{-1}\right]=\left[g\right]$ means there is
an isotopy of transverse maps between $g$ and $g\circ\rho^{-1}$.
In such an isotopy, the $\sum_{x\in B}\left(\kappa_{x}\left(g\right)+1\right)$
points of transversion in $\wedger$ may move around, but away from
the wedge point $o$, and without collisions. This means that their
order on each circle of $\wedger$ is preserved. In particular, for
every $x\in B$ and $j\in\left[\kappa_{x}\left(g\right)\right]$,
the isotopy takes the $\left(x,j\right)$ point of $g$ to the $\left(x,j\right)$
point of $g\circ\rho^{-1}$, and the collection of $\left(x,j\right)$-arcs/curves
of $g$ to the collection of $\left(x,j\right)$-arcs/curves of $g\circ\rho^{-1}$.
Thus, $\left[\rho\right]$ necessarily preserves every face of $\mathrm{polysim}\left(\left[g\right]\right)$.
\end{proof}
\begin{defn}
\label{def:T_infty}We define $\T_{\infty}=\T_{\infty}(\Sigma,f)$
\marginpar{$\protect\T_{\infty}(\Sigma,f)$} to be the subposet of
$\T=\T\left(\Sigma,f\right)$ consisting of classes of transverse
maps $\left[g\right]$ in $\T$ that do \emph{not} fill $\Sigma$.
\end{defn}

Recall that $\left[g\right]$ fills $\Sigma$ if its $o$-zones and
$z$-zones are all topological discs. This means, in particular, that
the preimage of every transversion point contains only arcs (and no
curves). 

Our notation $\T_{\infty}$ is in analogy to Harer's use of $A_{\infty}$
in \cite{HARERVCD,HARERSTABILITY} for the subcomplex of the arc complex
consisting of arc systems that do not cut the surface into discs;
Harer used this complex in \cite{HARERVCD} to construct a Borel-Serre
type bordification of Teichmüller space. This had previously been
done by Harvey \cite{HARV} using the complex of curves. These bordifications
are closely related to the Deligne-Mumford compactification \cite{DM}
of the moduli space of curves (see \cite[Remark 2.5]{MOND}).

We define $\tips$ to be the polysimplicial subcomplex of $\tps$
consisting of polysimplices in $\T_{\infty}$. It is clear that $\tips$
is indeed a subcomplex of $\tps$ since if the arcs of $\left[g\right]$
do not cut $\Sigma$ into discs then neither do the arcs of $\left[g'\right]$
obtained from $g$ by forgetting points of transversion. 
\begin{lem}
\label{lem:-acts-freely}$\MCG(f)$ acts freely on $\T\setminus\T_{\infty}$.
\end{lem}

\begin{proof}
If $\left[\rho\right]$ in $\MCG(f)$ fixes an isotopy class $\left[g\right]$
of filling transverse maps then we can assume $\rho$ fixes all the
arcs of $g$, so restricts to mapping classes on each of the zones
of $g$, which are all discs. The Alexander Lemma \cite[Lemma 2.1]{FM}
implies these mapping classes must be trivial, so $\rho$ is homotopic
to the identity on each zone of $g$, hence overall.
\end{proof}
So the isotropy groups $\mcg\left(f\right)_{\left[g\right]}$ are
trivial for $\left[g\right]\in\T\setminus\T_{\infty}$. The following
lemma shows that for any other element of $\T$, the isotropy groups
are not only infinite, but also have vanishing $L^{2}$-Betti numbers:
\begin{lem}
\label{lem:isoptropy-of-X-infinity}Let $\Gamma=\MCG(f)$. If $\left[g\right]\in\T_{\infty}$
then the isotropy group $\Gamma_{\left[g\right]}$ of $\left[g\right]$
are in $\B_{\infty}$.
\end{lem}

\begin{proof}
Fix a representative transverse map $g$ for $\left[g\right]$. Let
$\C$ denote a set of disjoint simple closed curves, where for every
zone of $g$ we add a simple closed curve parallel to every boundary
component of that zone to $\C$, and we think of the curves as drawn
inside the zones of $g$ they come from. If $g$ contains curves in
the preimages of points of transversion, then this process can add
to $\C$ multiple copies of isotopy classes of simple closed curves,
but this does not matter. 

Because $\C$ is drawn in the zones of $g$, then a Dehn twist in
any element of $\C$ belongs to $\Gamma$. Let $N$ be the subgroup
of $\Gamma$ generated by Dehn twists in elements of $\C$. This group
is isomorphic to $\Z^{r}$ for some $r\geq0$ because the curves in
${\cal C}$ are disjoint. In fact, $r\ge1$ since by assumption, some
zone of $g$ is not a topological disc, and hence has a boundary component
which does not bound a disc, so gives rise to a non-trivial Dehn twist.
To see that $N$ is normal in $\Gamma_{\left[g\right]}$, note that
any mapping class in $\Gamma_{\left[g\right]}$ can be taken to permute
the zones of $g$. Hence $\Gamma_{\left[g\right]}$ permutes the isotopy
classes of curves in $\C$. Therefore the conjugation by $\left[\rho\right]\in\Gamma_{\left[g\right]}$
of any Dehn twist in an element of $\C$ is another Dehn twist in
an element of $\C$. 

It was proved by von Neumann \cite{VN} that $\mathbb{Z}^{r}$ is
amenable, hence $\Gamma_{\left[g\right]}$ contains a normal infinite
amenable subgroup. The statement of the lemma now follows from Theorem
\ref{thm:infinite-normal-amenable-subgroup-in-B_infty}.
\end{proof}

\subsection{Proof of Theorem \ref{thm:EC of a single (S,f)}\label{subsec:Proof-of-Theorem which yields the main one}}

Fix $[(\Sigma,f)]\in\surfaces(\wl)$ and let $\Gamma=\mcg\left(f\right)$.
Recall that Theorem \ref{thm:EC of a single (S,f)} states that $\chi^{\left(2\right)}\left(\Gamma\right)$
is well-defined and is given by a finite alternating sum over the
set\linebreak{}
$\matchr^{*}\left(\wl;\Sigma,f\right)$ of matchings of the letters
of $\wl$. There is a natural map from elements of $\T\setminus\T_{\infty}$
to $\matchr^{*}\left(\wl;\Sigma,f\right)$:
\begin{defn}
\label{def:match map}Define a map \marginpar{$\widetilde{\protect\matchmap}$}
\[
\widetilde{\matchmap}\colon\T\setminus\T_{\infty}\to\matchr^{*}\left(\wl\right)
\]
as follows. The $\left(x,j\right)$-arcs of $\left[g\right]\in\T\setminus\T_{\infty}$
define a matching $\sigma_{x,j}$ between the instances of $x^{+1}$
in $\wl$ and the instances of $x^{-1}$. Define $\widetilde{\matchmap}\left(\left[g\right]\right)$
to be the element\linebreak{}
$\sigma\in\matchr^{\kappa\left(g\right)}\left(\wl\right)$ consisting
of the matchings $\left\{ \sigma_{x,j}\right\} _{x\in B,j\in\left[\kappa_{x}\right]}$. 
\end{defn}

We remark that if $\sigma=\widetilde{\matchmap}\left(\left[g\right]\right)$
then indeed $\sigma_{x,j}\ne\sigma_{x,j+1}$ for $x\in B,j<\kappa_{x}$:
this is guaranteed by \textbf{Restriction 3} and the fact that the
arcs of $g$ cut $\Sigma$ into discs.
\begin{lem}
\label{lem:bijs:props}The map $\widetilde{\matchmap}$ descends to
a bijection
\begin{equation}
\matchmap\colon~~\MCG(f)\backslash\left(\T\setminus\T_{\infty}\right)\stackrel{\cong}{\longrightarrow}\matchr^{*}\left(\wl;\Sigma,f\right).\label{eq:matchmap on quotient}
\end{equation}
\end{lem}

\begin{proof}
It is obvious that $\widetilde{\matchmap}$ is invariant under the
action of $\mcg\left(f\right)$, hence $\matchmap$ is well defined.
Since every element $\left[g\right]\in\T\setminus\T_{\infty}$ fills
$\Sigma$ (its arcs cut $\Sigma$ into discs), it is clear that $\left(\Sigma,f\right)\sim\left(\Sigma_{\sigma},f_{\sigma}\right)$
where $\sigma=\widetilde{\matchmap}\left(\left[g\right]\right)$,
using a homeomorphism $\rho\colon\Sigma\to\Sigma_{\sigma}$ taking
$g$ to the transverse map $f_{\sigma}$ (so $g\circ\rho^{-1}$ and
$f_{\sigma}$ are isotopic as transverse maps - recall Definitions
\ref{def:surface from matchings} and \ref{def:f_sigma}). This shows
$\left(i\right)$ that $\widetilde{\matchmap}\left(\left[g\right]\right)$
indeed belongs to $\matchr^{*}\left(\wl;\Sigma,f\right)$ (and not
only to $\matchr^{*}\left(\wl\right)$), and $\left(ii\right)$ that
if $\widetilde{\matchmap}\left(\left[g_{1}\right]\right)=\widetilde{\matchmap}\left(\left[g_{2}\right]\right)$
then $\left[g_{1}\right]$ and $\left[g_{2}\right]$ are in the same
$\mcg\left(f\right)$-orbit of $\T\setminus\T_{\infty}$, hence (\ref{eq:matchmap on quotient})
is injective. 

Finally, to see (\ref{eq:matchmap on quotient}) is surjective, notice
that for every $\sigma\in\matchr^{*}\left(\wl;\Sigma,f\right)$, if
$\rho$ is the homeomorphism showing the equivalence of $\left(\Sigma,f\right)\sim\left(\Sigma_{\sigma},f_{\sigma}\right)$
as above, then $\left[f_{\sigma}\circ\rho\right]\in\T\setminus\T_{\infty}$,
and its image through $\widetilde{\matchmap}$ is $\sigma$.
\end{proof}
It follows from Claim \ref{claim:new formula finite sum for every exponent}
that $\matchr^{*}\left(\wl;\Sigma,f\right)$ is finite, hence:
\begin{cor}
\label{cor:finitely many orbits of filling cells}There are finitely
many $\MCG(f)$-orbits in $\T\setminus\T_{\infty}$.
\end{cor}

We can now prove Theorem \ref{thm:EC of a single (S,f)}.
\begin{proof}[Proof of Theorem \ref{thm:EC of a single (S,f)}]
The polysimplicial complex $\tps$ is a $\Gamma$-$CW$-complex for\linebreak{}
$\Gamma=\mcg\left(f\right)$ by Lemma \ref{lem:|T|-is-a-G-CW-complex}.
The isotropy groups of $\Gamma$ in its action on $\tps$ are either
trivial if $\left[g\right]\in\T\setminus\T_{\infty}$ (Lemma \ref{lem:-acts-freely})
or infinite if $\left[g\right]\in\T_{\infty}$ (Lemma \ref{lem:isoptropy-of-X-infinity}).
Since $\Gamma\backslash\left(\T\setminus\T_{\infty}\right)$ is finite
(Corollary \ref{cor:finitely many orbits of filling cells}), we have
that 
\[
m\left(\tps,\Gamma\right)=\sum_{\left[g\right]\in\Gamma\backslash\T}\left|\Gamma_{[g]}\right|^{-1}=\sum_{\left[g\right]\in\Gamma\backslash\left(\T\setminus\T_{\infty}\right)}\left|\Gamma_{[g]}\right|^{-1}
\]
is finite. From Theorem \ref{thm:ec-as-alt-sum} we deduce that $\chi^{\left(2\right)}\left(\tps,\Gamma\right)$
is well defined and given by
\begin{eqnarray*}
\chi^{\left(2\right)}\left(\tps,\Gamma\right) & = & \sum_{\left[g\right]\in\Gamma\backslash\T}\left(-1\right)^{\dim\left(\mathrm{polysim}\left[g\right]\right)}\left|\Gamma_{[g]}\right|^{-1}=\sum_{\left[g\right]\in\Gamma\backslash\left(\T\setminus\T_{\infty}\right)}\left(-1\right)^{\left|\kappa\left(g\right)\right|}\left|\Gamma_{[g]}\right|^{-1}\\
 & = & \sum_{\left[g\right]\in\Gamma\backslash\left(\T\setminus\T_{\infty}\right)}\left(-1\right)^{\left|\kappa\left(g\right)\right|}=\sum_{\sigma\in\matchr^{*}\left(\wl;\Sigma,f\right)}\left(-1\right)^{\left|\kappa\left(\sigma\right)\right|},
\end{eqnarray*}
where the last equality follows from Lemma \ref{lem:bijs:props},
as the bijection maps the orbit of $\left[g\right]\in\T\setminus\T_{\infty}$
to a set of matchings $\sigma$ with $\kappa\left(\sigma\right)=\kappa\left(g\right)$.

Finally, Theorem \ref{thm:X contractible} and Lemmas \ref{lem:-acts-freely}
and \ref{lem:isoptropy-of-X-infinity} show that the assumptions of
Lemma \ref{lem:contractible-with-good-isotropy} hold for the action
of $\Gamma=\mcg\left(f\right)$ on $\tps$. As $m\left(\tps,\Gamma\right)$
is finite, we conclude that 
\[
\chi^{\left(2\right)}\left(\Gamma\right)=\chi^{\left(2\right)}\left(\tps,\Gamma\right)=\sum_{\sigma\in\matchr^{*}\left(\wl;\Sigma,f\right)}\left(-1\right)^{\left|\kappa\left(\sigma\right)\right|}.
\]
\end{proof}
This completes the proof of Theorem \ref{thm:EC of a single (S,f)},
and hence of our main Theorem \ref{thm:main} and of Theorem \ref{thm:stabilizers have L2-EC}.

\subsection{Incompressible maps and the proof of Theorem \ref{thm:K(g,1) for incompressible}\label{subsec:Incompressible-maps}}
\begin{defn}[{\cite[Page 247]{BROWN}}]
\label{def:chi(G)}If $G$ is a discrete group and $X$ is a $G$-\emph{CW}-complex
such that $G$ acts freely on $X$, $X$ is contractible, and $G\backslash X$
is a finite \emph{CW-}complex, then one defines the \emph{Euler characteristic
of $G$ }to be
\[
\chi(G)\stackrel{\mathrm{def}}{=}\chi(G\backslash X)
\]
where the right hand side is the topological Euler characteristic.
Since $G\backslash X$ is a $K(G,1)$-space for $G$, hence unique
up to weak homotopy equivalence, this definition does not depend on
$X$.
\end{defn}

Recall from Definition \ref{def:null-curve-incompressible} that $[(\Sigma,f)]\in\surfaces(\wl)$
is called incompressible if it admits no null-curves.
\begin{lem}
\label{lem:incompressible iff T_infty empty}$\left[\left(\Sigma,f\right)\right]$
is incompressible if and only if $\T_{\infty}\left(\Sigma,f\right)$
is empty.
\end{lem}

\begin{proof}
If $\left(\Sigma,f\right)$ admits a null-curve $\gamma$, one can
start with an arbitrary element $\left[g\right]$ of $\T$ and surger
$g$ using $H$-moves to remove its intersections with $\gamma$,
similarly to the proof of Proposition \ref{prop:deformation retract for null-arc system}
with $\gamma$ playing the role of the null-arc. It is easy to check
the resulting $\left[g'\right]$ is in $\T_{\infty}$. In the other
direction, the arcs and curves of any element $\left[g\right]\in\T_{\infty}$
are disjoint from some essential simple closed curve, which is thus
a null-curve of $\left(\Sigma,f\right)$.
\end{proof}
Recall that Theorem \ref{thm:K(g,1) for incompressible} says that
an incompressible $\left(\Sigma,f\right)$ admits a finite complex
as a $K\left(\Gamma,1\right)$-space for $\Gamma=\mcg\left(f\right)$,
and that $\chi\left(\Gamma\right)=\chi^{\left(2\right)}\left(\Gamma\right)$.
\begin{proof}[Proof of Theorem \ref{thm:K(g,1) for incompressible}]
By Lemmas \ref{lem:|T|-is-a-G-CW-complex}, \ref{lem:-acts-freely}
and \ref{lem:incompressible iff T_infty empty}, $\Gamma$ acts freely
on the $\Gamma$-$CW$-complex $\tps$, and by Corollary \ref{cor:finitely many orbits of filling cells}
the quotient $\Gamma\backslash\tps$ is finite. As $\tps$ is contractible
(Theorem \ref{thm:X contractible}) we obtain that $\Gamma\backslash\tps$
is the sought-after $K\left(\Gamma,1\right)$ -complex. Hence $\chi\left(\Gamma\right)=\chi\left(\Gamma\backslash\tps\right)$
is well defined. Moreover, the proof of the Theorem \ref{thm:EC of a single (S,f)}
in Section \ref{subsec:Proof-of-Theorem which yields the main one}
shows that $\chi\left(\Gamma\right)=\chi^{\left(2\right)}\left(\Gamma\right)$.
\end{proof}
\begin{rem}
Note that the $K\left(\Gamma,1\right)$-complex we obtained as a quotient
in the last proof can also be constructed directly as a cell complex
with a cell for every $\sigma\in\matchr^{*}\left(\wl,\Sigma,f\right)$,
in an analogous way to Definition \ref{def:the complex of transverse maps}
of the complex of transverse maps. The example of the single incompressible
map for $w=\left[x,y\right]\left[x,z\right]$, where there are two
vertices connected by two parallel edges, illustrates that this is
not always a polysimplicial complex. However, the set $\matchr^{*}\left(\wl,\Sigma,f\right)$
also has a natural partial order defined by forgetting proper subsets
of the matchings for every $x\in B$. We can thus realize this $K\left(\Gamma,1\right)$
also as the order complex $\left|\matchr^{*}\left(\wl,\Sigma,f\right)\right|$,
which is a genuine simplicial complex. 
\end{rem}

We end this subsection with a bound on the dimension of the $K\left(\Gamma,1\right)$-complex
we constructed:
\begin{cor}
If $\wl$ are all cyclically reduced and different than $1$, then
the $K\left(\Gamma,1\right)$-space we constructed has dimension at
most $-\chi\left(\Sigma\right)$.
\end{cor}

\begin{proof}
The $K\left(\Gamma,1\right)$-space is a quotient of $\T\left(\Sigma,f\right)$,
and therefore has the same dimension, which is bounded by $-\chi\left(\Sigma\right)$
-- see Lemma \ref{lem:X is finite dimensional} and Remark \ref{rem:bound on dimension when words are cylically reduced}.
\end{proof}

\subsection{Non-finiteness of $\protect\MCG(f)\backslash\protect\T$: why (the
proof of) Theorem \ref{thm:K(g,1) for incompressible} fails for compressible
maps\label{subsec:Nonfiniteness}}

When $\left(\Sigma,f\right)$ is compressible, the subposet $\T_{\infty}$
is non-empty (Lemma \ref{lem:incompressible iff T_infty empty}),
hence the action of $\Gamma=\mcg\left(f\right)$ on $\tps$ is not
free, and the quotient is not a $K\left(\Gamma,1\right)$. Still,
the ordinary Euler characteristic of a group is defined in much more
general cases then the one based on a finite $K\left(\Gamma,1\right)$-space
as in Definition \ref{def:chi(G)} -- see \cite[Chapter IX]{BROWN}.
For example, one could hope to use the following:
\begin{thm*}[{\cite[Proposition IX.7.3(e')]{BROWN}}]
Let $G$ be a discrete group, and let $X$ be a contractible $G$-$CW$-complex
such that $G\backslash X$ has finitely many cells and such that the
isotropy group $G_{c}$ of every cell $c$ ``has finite homological
type'' (see \cite[Page 246]{BROWN}). Then $\chi\left(G\right)$
is defined and satisfies 
\[
\chi\left(G\right)=\sum_{\left[c\right]\in G\backslash X}\left(-1\right)^{\dim c}\chi\left(G_{c}\right).
\]
\end{thm*}
It is not too difficult to show that when $G=\Gamma=\mcg\left(f\right)$
and $X=\tps$, all the assumptions in this theorem hold, except for
the assumption that $\Gamma\backslash\tps$ has finitely many cells.
It turns out, perhaps counter-intuitively, that indeed this latter
assumption often fails:

\subsubsection*{Abelian neck phenomenon}

\begin{figure}
\begin{centering}
\includegraphics{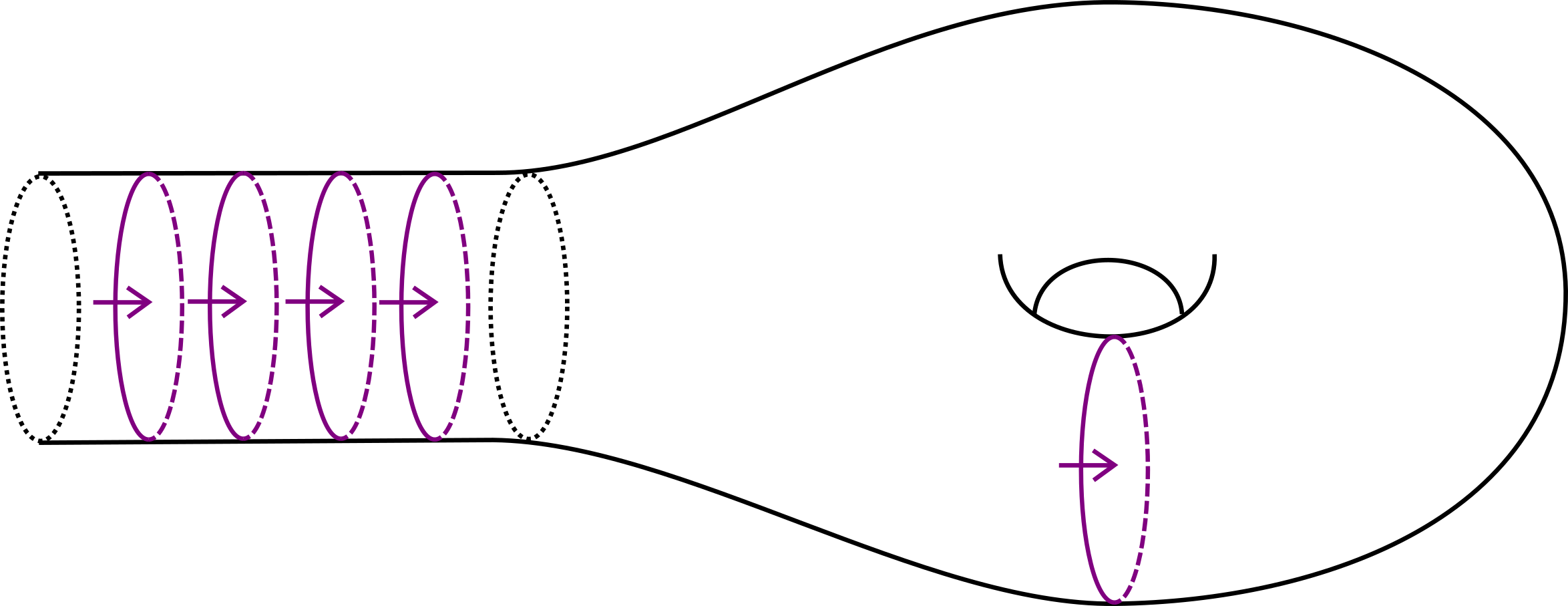}
\par\end{centering}
\caption{\label{fig:neck-figure}This figure shows part of a transverse map,
the surface extends to the left where there may be other arcs and
curves making up the map. Let $x\in B$. The transverse map shown
has $\kappa_{x}=0$ and purple curves are $\left(x,0\right)$-curves.}
\end{figure}

Let $g$ be the transverse map which is partially depicted in Figure
\ref{fig:neck-figure}, and $(\Sigma,f)$ be such that $\left[g\right]\in\T\left(\Sigma,f\right)$.
The key feature of $g$ is that there is a null-curve (e.g., one of
the dotted black lines) that \emph{separate}s a subsurface that is
mapped by $f$ at the level of $\pi_{1}$ to the cyclic group $\langle x\rangle$.
In terms of our picture, this can be seen as the curves that appear
in this subsurface are associated to only one generator $x$. Note
also in our picture we have illustrated a `neck' region bounded by
two black dotted curves that contains 4 parallel and codirected $(x,0)$-curves.
Assume for the sake of clarity that any curve that could be drawn
in the neck region is indeed drawn there. One could modify this transverse
map by changing the number of the repeated curves in the neck region. 
\begin{lem}
\label{lem:neck-extension-same-homotopy-type}No matter how many parallel
codirected $(x,0)$-curves are placed in the neck region of $g$,
the resulting transverse map still realizes $(\Sigma,f)$ and thus
represents an element of $\T\left(\Sigma,f\right)$.
\end{lem}

\begin{proof}
Call the rightmost $\left(x,0\right)$-arc in the neck $\alpha$ and
the $\left(x,0\right)$-curve in the right part of Figure \ref{fig:neck-figure}
$\beta$. Consider an $H$-move along a piece of arc connecting $\alpha$
to $\beta$ (and arrives to both from their positive side). This results
in a strict transverse map in $\T\left(\Sigma,f\right)$ which is
the same as $g$ except that $\alpha$ is omitted. This shows that
alternating the number of $\left(x,0\right)$-curve in the neck in
Figure \ref{fig:neck-figure} does not take us out of $\T\left(\Sigma,f\right)$.
\end{proof}
\begin{cor}
\label{cor:X doesn't have finite quotient}With $\T=\T(\Sigma,f)$
as in Lemma \ref{lem:neck-extension-same-homotopy-type}, obtained
from Figure \ref{fig:neck-figure}, $\MCG(f)\backslash\T$ is not
finite. 
\end{cor}

\begin{proof}
We can create elements in $\T$ with an unbounded number of curves,
and the number of curves is a $\MCG(f)$-invariant.
\end{proof}
This is not the most general version of this phenomenon: the neck
could for example be replaced by a collection of disjoint annuli that
cut from $\Sigma$ some subsurface with $\pi_{1}$ mapped by $f$
to a non-trivial cyclic subgroup of $\F_{r}$. Our aim here is to
give an illustrative example.

\section{Further applications and consequences\label{sec:Further-Applications}}

We specify here three interesting applications of our results and
techniques, regarding the stable commutator length of a word, the
complete classification of all incompressible solutions in $\surfaces\left(\wl\right)$,
and the cohomological dimension of $\mcg\left(f\right)$. Let us also
mention that our construction of a finite $K\left(\Gamma,1\right)$-space
for $\Gamma=\mcg\left(f\right)$ when $\left[\left(\Sigma,f\right)\right]$
is incompressible also enables one to write explicit finite presentations
for $\Gamma$: consult \cite[Pages 57-59]{MP}.

\subsection{Stable commutator length\label{subsec:Stable-commutator-length}}

Recall that Corollary \ref{cor:can hear scl} states that the $w$-measures
on $\left\{ \U\left(n\right)\right\} _{n\in\mathbb{N}}$ determine
$\mathrm{scl}\left(w\right)$, the stable commutator length of $w\in\F_{r}$,
defined in (\ref{eq:scl}). In this subsection we explain how this
result follows from Theorem \ref{thm:main} and from Calegari's rationality
theorem. 

Calegari's theorem, which is the main result of \cite{CALRATIONAL},
says that $\mathrm{scl}\left(w\right)$ is rational for every $w\in\left[\F_{r},\F_{r}\right]$.
First, it is shown that $\mathrm{scl\left(w\right)}$ is equal to
the infimum of $\frac{-\chi\left(\Sigma\right)}{2\left|j_{1}+\ldots+j_{\ell}\right|}$
over all possible $j_{1},\ldots,j_{\ell}\in\mathbb{Z}$ and $\left(\Sigma,f\right)$
admissible for $w^{j_{1}},\ldots,w^{j_{\ell}}$ \cite[Lemma 2.6]{CALRATIONAL}.
The proof goes through showing the existence of ``extremal surfaces''
for $w$: a surface attaining the infimum. Moreover, by \cite[Lemma 2.7]{CALRATIONAL},
this extremal surface can be taken to be admissible for $w^{j_{1}},\ldots,w^{j_{\ell}}$
with $j_{1},\ldots,j_{\ell}>0$. By definition of extremal surface,
$\Sigma$ has maximal Euler characteristic for $w^{j_{1}},\ldots,w^{j_{\ell}}$,
namely, $\chi\left(\Sigma\right)=\ch\left(w^{j_{1}},\ldots,w^{j_{\ell}}\right)$.
In fact, every surface which is admissible for $w^{j_{1}},\ldots,w^{j_{\ell}}$
with Euler characteristic $\ch\left(w^{j_{1}},\ldots,w^{j_{\ell}}\right)$
is extremal. By \cite[Lemma 2.9]{CALRATIONAL}, the maps associated
with extremal surfaces are $\pi_{1}$-injective, namely, if $\gamma\subset\Sigma$
is a non-nullhomotopic closed curve, then $f\left(\gamma\right)$
is not nullhomotopic. Note that this condition is stronger than incompressibility,
which only deals with \emph{simple} closed curves. The crux of the
matter is the following lemma:
\begin{lem}
\label{lem:pi1-injective means trivial stab}If $\left(\Sigma,f\right)$
is $\pi_{1}$-injective, then $\mcg\left(f\right)$ is trivial. 
\end{lem}

\begin{proof}
The outline of the argument here is that if $\left[\rho\right]\in\mcg\left(f\right)$
then $\left[\rho\right]_{*}\in\mathrm{Aut}\left(\pi_{1}\left(\Sigma\right)\right)$
fixes $f_{*}$, and since $f_{*}$ is injective, this means that $\left[\rho\right]_{*}$
must be the identity. By a variation of the Dehn-Nielsen-Baer Theorem,
it follows that $\left[\rho\right]$ is the identity.

In more detail, assume that $\Sigma$ is connected (the general cases
easily follows). Recall that $v_{1}\in\partial_{1}\Sigma$ is one
of the $\ell$ marked points at $\partial\Sigma$, and let $G=\pi_{1}\left(\Sigma,v_{1}\right)$.
If $\ell=1$, the Dehn-Nielsen-Baer Theorem (see Page \pageref{Dehn-Nielsen-Baer}
and \cite[Thm 2.4]{MP}) yields what we need. If $\ell\ge2$, consider
an arc $\gamma\subset\Sigma$ connecting $v_{1}$ and $v_{\ell}$.
Because $\left[\rho\right]\in\mcg\left(\Sigma\right)$ fixes the marked
points, we must have that $\rho\left(\gamma\right)$ is homotopic
relative to $\left\{ v_{1},v_{\ell}\right\} $ to $\beta*\gamma$,
where $\beta$ is a closed, not necessarily simple, curve based at
$v_{1}$ and ``$*$'' stands for concatenation. Inside $\F_{r}$
we have 
\[
f_{*}\left[\gamma\right]=f_{*}\left[\rho\left(\gamma\right)\right]=f_{*}\left[\beta*\gamma\right]=f_{*}\left[\beta\right]\cdot f_{*}\left[\gamma\right]
\]
hence $f_{*}\left[\beta\right]=1$ which means that $\beta$ is nullhomotopic
by $\pi_{1}$-injectivity. Hence we can assume without loss of generality
that $\rho$ fixes $\gamma$, and we can analyze $\rho$ on $\Sigma'$,
the surface obtained from $\Sigma$ by cutting along $\gamma$. Since
$\Sigma'$ has only $\ell-1$ boundary components, we are done by
induction.
\end{proof}

\begin{proof}[Proof of Corollary \ref{cor:can hear scl}]
By Lemma \ref{lem:pi1-injective means trivial stab} and the discussion
preceding it, if one of the extremal surfaces of $w$ is admissible
for $w^{j_{1}},\ldots,w^{j_{\ell}}$ with $j_{1},\ldots,j_{\ell}>0$,
then Theorem \ref{thm:main} translates in this case to 
\begin{equation}
\tr_{w^{j_{1}},\ldots,w^{j_{\ell}}}\left(n\right)=n^{\ch\left(w^{j_{1}},\ldots,w^{j_{\ell}}\right)}\cdot K+O\left(n^{\ch\left(w^{j_{1}},\ldots,w^{j_{\ell}}\right)-2}\right),\label{eq:trwl for extremal}
\end{equation}
where $K$ is the number of highest-Euler-characteristic surfaces
in $\surfaces\left(w^{j_{1}},\ldots,w^{j_{\ell}}\right)$. Note that
(\ref{eq:trwl for extremal}) is strictly positive for large enough
$n$. Hence,
\[
\frac{-\lim_{n\to\infty}\log_{n}\left|\tr_{w^{j_{1}},\ldots,w^{j_{\ell}}}\left(n\right)\right|}{2\left(j_{1}+\ldots+j_{\ell}\right)}=\frac{-\ch\left(w^{j_{1}},\ldots,w^{j_{\ell}}\right)}{2\left(j_{1}+\ldots+j_{\ell}\right)}=\mathrm{scl}\left(w\right).
\]
On the other hand, for an arbitrary $\ell>0$ and $j_{1},\ldots,j_{\ell}>0$
we have 
\[
\frac{-\lim_{n\to\infty}\log_{n}\left|\tr_{w^{j_{1}},\ldots,w^{j_{\ell}}}\left(n\right)\right|}{2\left(j_{1}+\ldots+j_{\ell}\right)}\ge\frac{-\ch\left(w^{j_{1}},\ldots,w^{j_{\ell}}\right)}{2\left(j_{1}+\ldots+j_{\ell}\right)}\ge\mathrm{scl}\left(w\right).
\]
This proves (\ref{eq:read scl}).
\end{proof}
\medskip{}

\begin{cor}
If $\mathrm{scl}\left(w_{1}\right)\ne\mathrm{scl}\left(w_{2}\right)$
then for every large enough $n$, the $w_{1}$-measure on ${\cal U}\left(n\right)$
is different from the $w_{2}$-measure on ${\cal U}\left(n\right)$.
In particular, if $w_{1}\in\left[\F_{r},\F_{r}\right]$ and $w_{2}\notin\left[\F_{r},\F_{r}\right]$
then they induce different measures on ${\cal U}\left(n\right)$ for
almost all $n$.
\end{cor}

\begin{proof}
Assume without loss of generality that $\mathrm{scl}\left(w_{1}\right)<\mathrm{scl}\left(w_{2}\right)$,
and let $j_{1},\ldots,j_{\ell}>0$ be so that $w_{1}^{j_{1}},\ldots,w_{1}^{j_{\ell}}$
admit an extremal surface. Then by the above discussion, $\tr_{w_{1}^{j_{1}},\ldots,w_{1}^{j_{\ell}}}\left(n\right)$
is strictly larger than $\tr_{w_{2}^{j_{1}},\ldots,w_{2}^{j_{\ell}}}\left(n\right)$
for any large enough $n$. In particular, if $w_{2}$ is not balanced,
i.e.~$w_{2}\notin\left[\F_{r},\F_{r}\right]$ and $\mathrm{scl}\left(w_{2}\right)=\infty$,
then nor is the set $w_{2}^{j_{1}},\ldots,w_{2}^{j_{\ell}}$ balanced
as we assume $j_{1},\ldots,j_{\ell}>0$. By Claim \ref{claim: tr=00003D0 for non-balanced words},
$\tr_{w_{2}^{j_{1}},\ldots,w_{2}^{j_{\ell}}}\left(n\right)\equiv0$
for every $n$. 
\end{proof}

\subsection{Classifying all incompressible solutions to generalized commutator
equation\label{subsec:Classifying-all-incompressible}}

Since the late 1970's there are known algorithms to determine the
commutator length of a given word $w\in\left[\F_{r},\F_{r}\right]$
\cite{Edmunds1975,Goldstein1979,CULLER} and also to find at least
one representative from every equivalence class of solutions to $\left[u_{1},v_{1}\right]\cdots\left[u_{g},v_{g}\right]=w$
with $g=\cl\left(w\right)$ \cite[Section 4.2]{CULLER}. In fact,
the algorithm in \cite{CULLER} uses matchings of letters of $w$
as in Proposition \ref{prop:incompressible as transverse maps}. Our
analysis and techniques expand Culler's algorithm to yield a clear
description of the set of classes of solutions and, in particular,
a direct way to distinguish them from each other.

Consider the poset $P=\matchr^{\left|\kappa\right|\le1}\left(\wl\right)$
consisting of sets of matchings for $\wl$ as in Section \ref{sec:A-Formula-for trwl},
where $\left|\kappa\right|\stackrel{\mathrm{def}}{=}\sum_{x\in B}\kappa_{x}\le1$
and $\sigma_{x,0}\ne\sigma_{x,1}$ whenever $\kappa_{x}=1$, and with
partial order $\sigma_{0}\prec\sigma_{1}$ whenever $\left|\kappa\left(\sigma_{0}\right)\right|=0$,
$\left|\kappa\left(\sigma_{1}\right)\right|=1$ and $\sigma_{0}$
is obtained from $\sigma_{1}$ by deleting one of the two $x$-matchings
for the $x\in B$ with $\kappa_{x}\left(\sigma_{1}\right)=1$. Recall
the definition of $\chi\left(\sigma\right)$ from Definition \ref{def:surface from matchings}.
Construct a graph $G\left(\wl\right)$ with vertices the elements
of $P$ and an edge $\left(\sigma_{0},\sigma_{1}\right)$ whenever
$\sigma_{0}\prec\sigma_{1}$ \emph{and} $\chi\left(\sigma_{1}\right)=\chi\left(\sigma_{2}\right)$.
We say a component $C$ of $G\left(\wl\right)$ is \emph{downward-closed
}if every vertex $\sigma_{1}$ of $C$ with $\kappa\left(\sigma_{1}\right)=1$
has two neighbors: the two elements of $P$ that are strictly smaller.
Recall the notation $\Sigma_{\sigma}$ and $f_{\sigma}$ from Definitions
\ref{def:surface from matchings} and \ref{def:f_sigma}.
\begin{prop}
The map $\varphi\colon P=\matchr^{\left|\kappa\right|\le1}\to\surfaces\left(\wl\right)$
given by 
\[
\sigma\mapsto\left[\left(\Sigma_{\sigma},f_{\sigma}\right)\right]
\]
induces a bijection between the downward-closed components of $G\left(\wl\right)$
and the incompressible pairs in $\surfaces\left(\wl\right)$.
\end{prop}

\begin{proof}
First, $\varphi$ is constant on connected components of $G\left(\wl\right)$:
indeed, assume that $\sigma_{0}\prec\sigma_{1}$ \emph{with} $\chi\left(\sigma_{0}\right)=\chi\left(\sigma_{1}\right)$
and, say, $\sigma_{0}$ is obtained from $\sigma_{1}$ by forgetting
the matching $\left(\sigma_{1}\right)_{x,1}$. Then the condition
$\chi\left(\sigma_{0}\right)=\chi\left(\sigma_{1}\right)$ shows forgetting
the $\left(x,1\right)$ transversion point of the transverse map $f_{\sigma_{1}}$
results in a transverse map which is still filling, and thus equal
to $f_{\sigma_{0}}$. Hence we can define $\hat{\varphi}$ to be a
map from the downward-closed components of $G\left(\wl\right)$ to
$\surfaces\left(\wl\right)$.

Second, the image of $\hat{\varphi}$ consists of incompressible elements.
To see this, let $C$ be a downward-closed component of $P$. Let
$\sigma\in C$ have $\left|\kappa\left(\sigma\right)\right|=0$. Assume
to the contrary that $\varphi\left(\sigma\right)$ is compressible.
Then it admits a null-curve $\gamma$ which is not disjoint from the
matching-edges in $\Sigma_{\sigma}$ (recall that the matching-edges
cut $\Sigma_{\sigma}$ to discs). One can start performing $H$-moves
along this null-curve. In an $H$-move between two $x$-matching-edges
$e_{1}$ and $e_{2}$ along a piece of $\gamma$, one first creates
a transverse map $g_{1}$ with $\left|\kappa\left(g_{1}\right)\right|=1$
(with $f_{\sigma}\prec g_{1}$ in $\T=\T\left(\Sigma_{\sigma},f_{\sigma}\right)$)
and then obtains $g_{0}\prec_{\T}g_{1}$ with $\left|\kappa\left(g_{0}\right)\right|=0$
which has two fewer intersection points with $\gamma$. Because the
matching-edges cut $\Sigma_{\sigma}$ to disks, $e_{1}$ and $e_{2}$
must be distinct, and thus $g_{1}$ has only arcs and $\left[g_{1}\right]=\left[f_{\sigma_{1}}\right]$
for some $\sigma_{1}\in C$. As $C$ is downward-closed, there is
some $\sigma_{0}\in C$ with $\left[g_{0}\right]=\left[f_{\sigma_{0}}\right]$.
We can continue in the same manner until $\gamma$ intersects no matching-edges,
which is a contradiction. 

Third, $\hat{\varphi}$ is the sought-after bijection. Indeed, every
incompressible\linebreak{}
$\left[\left(\Sigma,f\right)\right]\in\surfaces\left(\wl\right)$
is the $\varphi$-image of some component of $G\left(\wl\right)$
by Proposition \ref{prop:incompressible as transverse maps}. If $\sigma\in P$
satisfies $\varphi\left(\sigma\right)=\left[\left(\Sigma,f\right)\right]$
and $C$ is the connected component of $\sigma$ then $C$ is a component
of the $1$-skeleton of the $K\left(\Gamma,1\right)$-complex we constructed
in the proof of Theorem \ref{thm:K(g,1) for incompressible} in Section
\ref{subsec:Incompressible-maps}. In particular, this complex is
connected (because its universal cover $\tps$ is connected), hence
so its $1$-skeleton is connected. This shows that $C$ is the only
component mapping to $\left[\left(\Sigma_{\sigma},f_{\sigma}\right)\right]$
and that it is downward-closed. 
\end{proof}
Alternatively, one could use here a direct argument imitating some
ingredients from the proof of Theorem \ref{thm:X contractible}, as
follows. For $\left[\left(\Sigma,f\right)\right]$ incompressible,
show that $\varphi^{-1}\left(\left[\left(\Sigma,f\right)\right]\right)$
is a downward-closed connected component of $G\left(\wl\right)$,
by taking a maximal system of null-arcs, showing there is a single
$\sigma_{0}\in\varphi^{-1}\left(\left[\left(\Sigma,f\right)\right]\right)$
with matching edges disjoint from these null-arcs, and showing every
other element in the preimage can be connected to $\sigma_{0}$ by
$H$-moves that never leave the same connected component of $G\left(\wl\right)$. 

\subsection{Finiteness of the cohomological dimension of the stabilizer $\protect\mcg\left(f\right)$\label{subsec:Cohomological-dimension-of}}

Recall that the cohomological dimension, $\cd(\Gamma)$, of a torsion-free
group $\Gamma$ is the minimal length of a projective resolution of
$\Z$ over $\Z\Gamma$ if one exists, and $\infty$ otherwise. If
a group $\Gamma$ is virtually torsion-free then the virtual cohomological
dimension, $\vcd(\Gamma)$, is defined to be $\cd(\Gamma')$ where
$\Gamma'$ is a finite index torsion-free subgroup of $\Gamma$; it
is a theorem of Serre \cite{SERRE71} that the resulting dimension
does not depend on the chosen finite index subgroup. As the the following
result is not needed for the main results of this paper, we only sketch
its proof.
\begin{prop}
\label{prop:finitevcd}Let $\Sigma$ be a compact orientable surface
with no closed connected components and let $f\colon\Sigma\to\wedger$
be a map. Then $\cd(\MCG(f))<\infty$.
\end{prop}

\begin{proof}[Sketch of proof]
Let $\T=\T(\Sigma,f)$ and $\Gamma=\MCG(f)$. Note that $\Gamma$
is torsion-free since $\Sigma$ has no closed components. We use a
result that is attributed to Quillen by Serre in \cite[Prop. 11(a)]{SERRE71}.
As $\tps$ is contractible (Theorem \ref{thm:X contractible}), Quillen's
result says that

\[
\cd(\Gamma)\leq\sup_{[g]\in\Gamma\backslash\T}\left(\dim(\mathrm{polysim}\left(\left[g\right]\right))+\cd(\Stab_{\Gamma}(g)\right).
\]
Therefore, as $\tps$ is finite dimensional (Lemma \ref{lem:X is finite dimensional}),
it suffices\textbf{ }to prove there is an upper bound depending only
on the pair $(\Sigma,f)$ for $\cd(\Stab_{\Gamma}(g))$ given an arbitrary
element $g$ in $\T$.

We now give a quick analysis of these stabilizers. Fix a transverse
map $g$ with $\left[g\right]\in\T$. Let $\left\{ \Sigma_{i}\right\} _{i\in I}$
denote the zones of $g$ which are not annuli bounded by two curves
of $g$. By Euler characteristic argument, $I$ is finite and bounded
independently of $g$. Form $\Sigma_{i}^{*}$ by contracting each
end of $\Sigma_{i}$ bounded by a curve of $g$ to a point, and mark
the new points $W_{i}\subset\Sigma_{i}^{*}$ on their respective surfaces.
We denote by $\MCG(\Sigma_{i}^{*},W_{i})$ the mapping class group
of $\Sigma_{i}^{*}$ that fixes each individual element of $W_{i}$. 

The subgroup $\Gamma_{0}\le\mathrm{Stab}_{\Gamma}\left(g\right)$
that fixes all the curves in $g$ and their orientations has finite
index in $\mathrm{Stab}_{\Gamma}\left(g\right)$, and there is a short
exact sequence obtained by restricting mapping classes in $\Gamma_{0}$
to the zones $\Sigma_{i}$:
\begin{equation}
1\to N\to\Gamma_{0}\to H\overset{\mathrm{def}}{=}\prod_{i\in I}\MCG(\Sigma_{i}^{*},W_{i})\to1,\label{eq:SES-vcd}
\end{equation}
where $N$ is a free abelian group generated by Dehn twists in the
curves of $g$. The reason one obtains the whole of each $\MCG(\Sigma_{i}^{*},W_{i})$
as a factor is because $g$ maps each $\Sigma_{i}$ to a contractible
piece of $\wedger$, and any lift of any element of $\MCG(\Sigma_{i}^{*},W_{i})$
to $\MCG(\Sigma)$ can be taken to be the identity outside $\Sigma_{i}$,
and therefore preserves the homotopy class of $f$. Although $\MCG(\Sigma_{i}^{*},W_{i})$
could contain torsion, it is virtually torsion-free (see either \cite[Theorem 6.8.A]{IVANOV}
or \cite[\S 4]{HARERVCD}).

Harer proved in \cite{HARERVCD} that for any surface $\Sigma$ and
collection of interior marked points $W$,
\[
\vcd(\MCG(\Sigma,W))\leq4g(\Sigma)+2|\pi_{0}(\partial\Sigma)|+|W|-3,
\]
where $g\left(\Sigma\right)$ is the genus of $\Sigma$. Therefore
using \cite[Prop. VIII.2.4.b]{BROWN} together with an argument as
in \cite[Proof of Prop. IX.7.3.d]{BROWN} to pass between $\vcd$
and $\cd$, one obtains
\[
\vcd(H)\leq\sum_{i\in I}\left(4g(\Sigma_{i}^{*})+2|\pi_{0}(\partial\Sigma_{i}^{*})|+|W_{i}|-3\right)\le F_{1}\left(\Sigma\right),
\]
where $F_{1}\left(\Sigma\right)$ is a bound in terms of $\Sigma$
which is independent of $g$. Since $H$ is virtually torsion-free,
and $\Gamma_{0}$ has no torsion, we can find torsion-free finite
index subgroups $\Gamma'_{0},H'$ in $\Gamma_{0}$ and $H$ respectively
that form a short exact sequence $1\to N\to\Gamma'_{0}\to H'\to1$.
Then Serre's Theorem \cite{SERRE71} gives $\cd(\Stab_{\Gamma}(g))=\cd(\Gamma_{0})=\cd(\Gamma'_{0})$
and $\cd(H')=\vcd(H)$. We also have $\cd\left(N\right)\leq F_{2}(\Sigma)$
where $F_{2}(\Sigma)$ is the maximal number of pairwise non-isotopic
disjoint simple closed curves on $\Sigma$. Now applying \cite[Prop. VIII.2.4.b]{BROWN}
to the short exact sequence for $N,\Gamma_{0}',H'$ we get 
\[
\cd(\Stab_{\Gamma}(g))=\cd(\Gamma_{0}')\leq\cd(N)+\cd(H')=\cd(N)+\vcd(H)<F_{1}\left(\Sigma\right)+F_{2}\left(\Sigma\right).
\]
\end{proof}

\section{Open problems\label{sec:Open-Questions}}

We mention some open problems that naturally arise from the discussion
in this paper.
\begin{enumerate}
\item \label{enu:primitivity conjecture}Recall that primitive words are
the orbit in $\F_{r}$ of the single-letter word $x$ under the action
of $\mathrm{Aut}\left(\F_{r}\right)$. As mentioned on Page \pageref{primitivity conjecture in S_n},
it was shown in \cite{PP15} that only primitive words induce uniform
measure on the symmetric group $S_{n}$ for all n. Is the same true
for unitary groups? Namely, if a word induces Haar measure on $\U\left(n\right)$
for all $n$, is the word necessarily primitive? In fact, the following
question raised by Tsachik Gelander a few years ago (by private communication)
is still open: if a word induces Haar measure on $\U\left(2\right)$,
is the word necessarily primitive?
\item Fix $j_{1},\ldots,j_{\ell}\in\mathbb{Z}$. Given $w\in\F_{r}$, is
there a nice criterion for determining whether the rational expression
for $\tr_{w^{j_{1}},\ldots,w^{j_{\ell}}}\left(n\right)$ has the same
value as for the primitive case when $w=x$? An illustrating example
is $\trw\left(n\right)$ -- we know it vanishes outside $\left[\F_{r},\F_{r}\right]$,
but it is not clear when it vanishes inside $\left[\F_{r},\F_{r}\right]$.
Another illustrating example is $\tr_{w,w^{-1}}\left(n\right)$: when
does it differ from $1$? Some examples for each are elaborated in
Table \ref{tab:examples} and on Page \pageref{subsec:Examples}.
\item Let $\Sigma$ be a connected, orientable surface with boundary, and
let $f\colon\Sigma\to\wedger$. We showed here that $\mcg\left(f\right)$
has a well-defined $L^{2}$-Euler-characteristic (Theorem \ref{thm:stabilizers have L2-EC})
and a finite cohomological dimension (Proposition \ref{prop:finitevcd}).
Does $\mcg\left(f\right)$ always have ``finite homological type''
as defined in \cite[Page 246]{BROWN}? And if so, does its ordinary
Euler characteristic coincide with the $L^{2}$-one?
\item \label{enu:rationality from Main}We deduced the rationality of $\trwl\left(n\right)$
in Theorem \ref{thm:trwl as finite sum} directly from Weingarten
calculus. The rationality means that the different $L^{2}$-Euler
characteristics appearing in Theorem \ref{thm:main} ``know'' about
each other. Is it possible to deduce the rationality of $\trwl\left(n\right)$
(i.e., Proposition \ref{prop:rational expression}) from our main
theorem, Theorem \ref{thm:main}? 
\item What can one say about the distribution of $\trw\left(n\right)$ when
$w$ is a long random word in $\left[\F_{r},\F_{r}\right]$? For example,
what is the distribution of the commutator length of $w$? Is it true
that for most words of a fixed length in $\left[\F_{r},\F_{r}\right]$,
the stabilizers $\mcg\left(f\right)$ of incompressible solutions
are trivial?
\item What can one systematically say about the $L^{2}$-Euler characteristic
of $\MCG(f)$? For which $f$ are they zero, negative, or positive?
The case when $f$ is incompressible is a natural starting point.
A sufficiently good understanding of this question would allow one
to make progress on Conjecture \ref{conj:aner}. The Euler characteristic
of the mapping class group of a closed surface was calculated by Harer-Zagier
\cite{HARERZAGIER} and the\emph{ sign} of the Euler characteristic
of the mapping class group was re-obtained by McMullen \cite{McMullen2000}
by different methods.
\end{enumerate}
\bibstyle{amsalpha} \bibliographystyle{amsalpha}
\bibliography{database_united}

\noindent Michael Magee, \\
Department of Mathematical Sciences,\\
Durham University, \\
Lower Mountjoy, DH1 3LE Durham,\\
United Kingdom

\noindent \texttt{michael.r.magee@durham.ac.uk}\\

\noindent Doron Puder, \\
School of Mathematical Sciences, \\
Tel Aviv University, \\
Tel Aviv, 6997801, Israel\\
\texttt{doronpuder@gmail.com }
\end{document}